\documentclass[a4paper,10pt,twoside]{amsart}

\usepackage[affil-it]{authblk}
\usepackage{titling}
\usepackage[latin1]{inputenc}
\usepackage{calligra}
\usepackage[OT1]{fontenc}
\usepackage[english]{babel}
\usepackage{amsfonts}
\usepackage{amsmath}
\usepackage{amssymb}
\usepackage{amsthm}
\usepackage{faktor}
\usepackage{float}
\usepackage{enumerate}
\usepackage{color}
\usepackage{pdfpages}
\usepackage{esint}
\usepackage{hyperref}
\usepackage{cite}
\usepackage{fancyhdr}
\usepackage{mathtools}
\usepackage{mathrsfs}

\usepackage{etoolbox}
\patchcmd{\section}{\scshape}{\bfseries}{}{}
\makeatletter
\renewcommand{\@secnumfont}{\bfseries}
\makeatother

\newcommand\testname{Abstract}
\makeatletter
\newenvironment{abs}{%
    \small
    \begin{center}%
        {\textbf \testname\vspace{-.2em}\vspace{\z@}}%
    \end{center}%
    \quote
    }
   {\endquote}
\makeatother

\DeclareMathOperator*{\Id}{Id}

\newcommand{\tr}{{}^\mathrm{t} }
\newcommand{\Div}{\mathrm{div}}
\newcommand{\trc}{\mathrm{tr}}

\newcommand{\Pp}{\mathcal{P}}
\newcommand{\ee}{\varepsilon}

\newcommand{\Aa}{\mathcal{A}}
\newcommand{\Cc}{\mathcal{C}}
\newcommand{\Bb}{\mathcal{B}}
\newcommand{\dd}{\mathrm{d}}
\newcommand{\Ss}{\mathscr{S}}

\newcommand{\SSS}{\mathbb{S}}

\newcommand{\J}{\mathcal{J}}
\newcommand{\I}{\mathcal{I}}
\newcommand{\MM}{\mathcal{M}}

\newcommand{\Ff}{\mathscr{F}}
\newcommand{\FF}{\mathcal{F}}

\newcommand{\DD}{\mathfrak{D}}
\newcommand{\RR}{\mathbb{R}}
\newcommand{\ZZ}{\mathbb{Z}}
\newcommand{\NN}{\mathbb{N}}

\newcommand{\BB}{\dot{B}}
\newcommand{\Hh}{\dot{H}}
\newcommand{\Dd}{\dot{\Delta}}
\newcommand{\Sd}{\dot{S}}
\newcommand{\Rd}{\dot{R}}
\newcommand{\Th}{\dot{T}}
\newtheorem{theorem}{Theorem}[section]

\newtheorem{prop}[theorem]{Proposition}
\newtheorem{lemma}[theorem]{Lemma}
\newtheorem{definition}[theorem]{Definition}
\newtheorem{remark}[theorem]{Remark}

\title{\Large 
	\textbf{\uppercase{A Global 2D Well-Posedness Result on the\\ Order Tensor \\Liquid Crystal Theory}}}
	
\author{Francesco De Anna$\quad$}
\affil{	\textsc{Universit\'e de Bordeaux} \\ 
		\small\textsc{Institut de Math\'ematiques de Bordeaux}\\ 
		\small{F-33405 Talence Cedex, France}
		\\ \vspace{0cm}\\
		\small\textnormal{Francesco.Deanna@math.u-bordeaux1.fr}}
\date{May 22, 2015}
\markleft{ FRANCESCO DE ANNA}
\begin{document}
\maketitle

\begin{abs}
	 In \cite{MR2864407} Paicu and Zarnescu have studied an order tensor system which describes the flow of a liquid crystal. They have proven the existence of weak solutions, the propagation of 
	 higher regularities, namely $H^s$ with $s>1$ and the weak-strong uniqueness in dimension two. This paper is devoted to the propagation of lower regularities, 
	 namely $H^s$ for $0<s\leq 1$ and to prove the uniqueness of the weak solutions. For the completeness of this research, we also propose an alternative approach in order to prove the existence 
	 of weak solutions.
\end{abs}

\vspace{0.2cm}

\vspace{0.1cm}

\section{Introduction and main results}

\subsection{Liquid Crystal}
The Theory of liquid crystal materials has attracted much attention over the recent decades. Generally, the physical state of a material can be determined by the motion degree of 
freedom about its molecules. Certainly, the widespread physical states of matter are the solid, the liquid and the gas ones. 
If the movement degree of freedom is almost zero, namely the forces which act on the molecules don't allow any kind of movement, forcing the material structure to be confined 
in a specific order, then we are classifying a solid material. If such degree still preserves a strong intermolecular force but it is not able to restrict the molecules 
to lie on a regular organization, then we are considering a fluid state of matter. Finally in the gas phase the forces and the distance between the molecules are weak and large respectively, so that the material is not confined and it is able to extend its volume.

\vspace{0.1cm}
\noindent
However, some materials possess some common liquid features  as well as some solid properties, namely the liquid crystals. As the name suggests, a liquid crystal is a compound of fluid molecules, which has a state of matter between the ordinary liquid one and the crystal solid one. The molecules have not a positional order but they assume an orientation which can be modified 
by the velocity flow. At the same time a variation of the alignment can induce a velocity field as well. In a common liquid (more correctly an isotropic liquid) if we consider the orientation of a single molecule then we should see the random variation of its position. Nevertheless, in a crystal liquid, we see an amount of orientational order. 

\vspace{0.1cm}
\noindent
It is well-documented that liquid crystals have been well-known for more than a century, however they have received a growth in popularity and much study only in recent decades, since they have attracted much attention thanks to their potential applications (see for instance \cite{refId0}).

\vspace{0.1cm}
\noindent
Commonly, in literature the liquid crystals are categorized by three sub-families, namely the nematics, the cholesterics and the smectics. On a nematic liquid crystal, the molecules have 
the same alignment with a preferred direction, however their positions are not correlated. 
On a cholesteric liquid crystal we have a foliation of the material where on each plaque the molecules orient themselves with the same direction (which could vary moving on the foliation). As in the nematic case, a cholesteric liquid crystal doesn't require any kind of relation 
between the positions of the molecules. 
At last, on a smectic liquid crystal we have still a privileged direction for all the molecules, as in the nematic case, however the position of them is bonded by a stratification. In addition to the orientational ordering, the molecules lie in layers.

\subsection{The Order Tensor Theory}
A first mathematical approach to model the generic liquid crystals has been proposed by Ericksen \cite{MR0137403} and Leslie \cite{MR1553506} over the period of 1958 through 1968. Even if they have presented a system which has been extensively studied in literature, for instance in \cite{MR1784963} and \cite{MR3021544}, several mathematical challenges and difficulties reside in such 
model. Hence, in 1994, Baris and Edwards \cite{Beris} proposed an alternative approach based on the concept of order Q-tensor, that one can find also in physical literature, for example \cite{PhysRevE.63.056702} and \cite{PhysRevE.67.051705}. The reader can find an exhaustive introduction to the Q-tensor Theory in a recent paper of Mottram and Newton \cite{09097abcb672450aa659cbf01670548d}, however we present here some hints in order to introduce the Q-tensor system. 

\vspace{0.1cm}
\noindent
Let us assume that our material lies on a domain $\Omega$ of $\RR^3$. A first natural strategy to model the molecules orientation is to introduce a vector field $n$, the so called director field (see for instance \cite{MR1883167}, which returns value on $\SSS^2$, the boundary of the unit sphere on $\RR^3$. Here $n(t,x)$ is a specific vector for any fixed time and for any $x\in \Omega$. An alternative approach is not to consider a precise position on $\SSS^2$ but to establish the probability that $n(t,x)$ belongs to some measurable subset $\Aa\subseteq\SSS^2$. Therefore we introduce a continuously distributed measure $\Pp$ on $\SSS^2$, driven by a density $\rho$
\begin{equation*}
	\Pp(\Aa) = \int_{\Aa} \rho(P)\dd \sigma(P) = \int_{\Aa}\dd \rho(P).
\end{equation*}
We supposed the molecules to be unpolar, so that there is no difference between the extremities of them, so mathematically the probability $\Pp(\Aa)$ is always equal to $\Pp(-\Aa)$, which yields that the first order momentum vanishes:
\begin{equation*}
	\int_{\SSS^2} \rho(P)\dd \sigma = 0.
\end{equation*}
Now considering the second order momentum tensor, given by
\begin{equation*}
	M := \int_{\SSS^2} P\otimes P \dd \rho(P) = \left( \int_{\SSS^2} P_iP_j \dd \rho(P) \right)_{i,j=1,2,3} \in \MM_{3}(\RR),
\end{equation*}
where $\MM_{3}(\RR)$ denotes the $3\times 3$ matrices with real coefficients, we observe that $M$ is a symmetric matrix and it has trace $\trc M=1$.

\noindent
In the presence of an isotropic liquid, the orientation of the molecules is uniform in every direction, hence in this case the probability $\Pp_0$ is given by 
\begin{equation*}
	\Pp_0(\Aa) = \int_{\Aa} \dd \sigma(P),
\end{equation*}
so that the corresponding second order momentum $M_0$ is exactly $\Id/3$. We denote by $Q$ the difference between a general $M$ and $M_0$ obtaining a tensor which is known as the de Gennes order parameter tensor. Roughly speaking, $Q$ interprets the deviation between a general liquid crystal and an isotropic one. From the definition, it is straightforward that $Q$ is a symmetric  tensor and moreover it has null trace. If $Q$ assumes the form $s_+(n\otimes n -\Id/3)$, where $s_+$ is a suitable constant, then the system which models the liquid crystal (and we are going to present) reduces to the general Ericksen-Leslie system (see for instance \cite{Beris}).

\subsection{The Q-Tensor System} The present paper is devoted to the global solvability issue for the following system as an evolutionary model for the liquid crystal hydrodynamics:
\begin{equation}\label{main_system}
	\tag{$P$}
	\begin{cases}
	\;	\partial_t Q +  u\cdot \nabla  Q - \Omega  Q +  Q \Omega = \Gamma H( Q)																&\RR_+	\times	\RR^2,\\
	\;	\partial_t u + 	 u \cdot \nabla u-\nu\Delta u +\nabla \Pi=  L\Div\,\{\;Q \Delta  Q - \Delta  Q  Q - \nabla  Q\odot \nabla Q\;\}		&\RR_+	\times	\RR^2,\\
	\;	\Div\,u = 0																															&\RR_+	\times	\RR^2,\\
	\;	(u,\,Q)_{t=0} = (u_0,\, Q_0)																										&\quad\quad\;\;\RR^2,
	\end{cases}
\end{equation}
Here $Q=Q(t,x)\in \MM_3(\RR)$ denotes the order tensor, $u=u(t,x)\in \RR^3$ represents the velocity field, $\Pi = \Pi(t,x)\in \RR$ stands for the pressure, everything depending on the time variable $t\in \RR$ and on the space variable $x\in \RR^2$. The symbol $\nabla Q\odot \nabla Q$ denotes the $3\times 3$ matrix whose $(i, j)$-th entry is given by $\trc (\partial_i Q\, \partial_jQ)$, for $i,\,j = 1,\,2,\,3$. Moreover $\Gamma$, $\nu$ and $L$ are three positive constants. 

\vspace{0.1cm}
\noindent The left hand side of the order tensor equation is composed by a classical transport time derivative while, defining $\Omega$ as the 
antisymmetric matrix $\Omega := (\nabla u  - \tr \nabla u)1/2$, $Q\Omega- \Omega Q$ is an Oldroyd time derivative and describes how the flow gradient rotates and stretches the order parameter.
On the right-hand-side, $H(Q)$ denotes
\begin{equation*}
	H(Q) :=	
		\underbrace{	
			-a 		Q  
			+b		\Big( 
						Q^2 - \trc (Q^2)\frac{\Id}{3} 
					\Big)  
			-c\,	\trc(Q^2)Q  
					}_{P(Q)}
		+L		\Delta Q,
\end{equation*}
and $P$ is the so called Landau-de Gennes thermotropic forces (more precisely it is a truncated taylor expansion about the original one, see for instance \cite{PhysRevLett592582}). Here $a$, $b$ and $c$ are real constant, and from here on we are going to assume $c$ to be positive.

\vspace{0.1cm}
\noindent In reality, Systems \eqref{main_system} is a simplification of a more general one. More precisely, fixing a real $\xi\in [0,1]$, we consider
\begin{equation}\label{main_system_xi}
	\tag{$P_\xi$}
	\begin{cases}
	\;	\partial_t Q +  u\cdot \nabla Q - S(\nabla u, Q) = \Gamma H( Q)														&\RR_+	\times	\RR^2,\\
	\;	\partial_t u + 	 u \cdot \nabla u-\nu\Delta u +\nabla \Pi=  \Div\,\{\tau + \sigma\}									&\RR_+	\times	\RR^2,\\
	\;	\Div\,u = 0																											&\RR_+	\times	\RR^2,\\
	\;	(u,\,Q)_{t=0} = (u_0,\, Q_0)																						&\quad\quad\;\;\RR^2,
	\end{cases}
\end{equation}
where $S(\nabla u,\,Q)$ stands for
\begin{equation*}
	S(Q, \nabla u)		:=	(\xi\,D + \Omega)		\Big( Q + \frac{\Id}{2} \Big)  +
													\Big( Q + \frac{\Id}{2} \Big) (\xi\,D - \Omega ) -
												2\xi\Big( Q + \frac{\Id}{2} \Big)	\trc( Q \nabla u ),
\end{equation*}
with $D:=	( \nabla u + \tr \nabla u)1/2$. Moreover $\tau$ and $\sigma$ are the symmetric and antisymmetric part of the the additional stress tensor respectively, namely
\begin{align*}
	\tau				&:=	 -\xi	( 
										Q + \frac{\Id}{2} 
									)H(Q) -	
							\xi H(Q)( 
										Q + \frac{\Id}{2} 
									)
							+2\xi	( 
										Q + \frac{\Id}{2} 
									)	
									\trc \{ Q H(Q) \}
							-L		\{	
										\nabla Q\odot \nabla Q + \frac{\Id}{3} | Q |^2	
									\},\\
	\sigma				&:=	Q	H(Q)	-	H(Q)	Q.
\end{align*}
Here $\xi$ is a molecular parameter which describes the rapport between the tumbling and aligning effect that a shear flow exert over the liquid crystal directors. In all this paper we are going to consider the simplest case $\xi=0$, namely system \eqref{main_system}, however we predict that all our results are available for the general case and we will prove them in a forthcoming paper. 

\vspace{0.1cm}
\noindent
Before going on, let us recall what we mean by a weak solution of system \eqref{main_system}.

\begin{definition}\label{def_weak_sol}
	Let $Q_0 $ and $u_0$ be a $3\times 3$ matrix a $3$-vector respectively, whose components belong to $L^2( \RR^2 )$. We say that $(u,\,Q)$ is a weak solution for 
	\eqref{main_system} if $u$ belongs to $L^\infty_{loc}(\RR_+ , L^2_x)\cap L^2_{loc}(\RR_+,\Hh^1)$, $Q$ belongs to $C(\RR_+, H^1)\cap L^2_{loc}(\RR_+,\Hh^2)$ and \eqref{main_system} is 
	fulfilled in the distributional sense. 
\end{definition}

\subsection{Some Developments in the order tensor Theory} Although the Q-tensor theory has received much attention in several disciplines as Physics \cite{refId0}, numerical analysis \cite{MR3319854}, mathematical analysis \cite{09097abcb672450aa659cbf01670548d}, the solvability study of the related system has not received numerous investigations, yet. We recall here 
some recent results.

\vspace{0.1cm}
\noindent in  \cite{MR3238138}, D. Wang, X. Xu and C. Yu have developed the existence and long time dynamics of globally defined weak solution. In their paper, system \eqref{main_system} has been considered in the compressible and inhomogeneous setting, the fluid density $\rho$ not necessarily constant, described by a transport equation, and moreover a pressure dependent on $\rho$.

\vspace{0.1cm}
\noindent In  \cite{MR3081935} J. Fahn and T. Ozawa prove some regularity criteria for a local strong solution of system \eqref{main_system}.

\vspace{0.1cm}
\noindent In \cite{MR2864407}, M. Paicu and A. Zarnescu first show the existence of a Lyapunov functional for system \eqref{main_system}. Then they prove the existence of a weak solution thanks to a Friedrichs scheme. They also show the propagation of higher regularity, namely $H^{s}(\RR^2)\times H^{1+s}(\RR^2)$ for $(u,\,Q)$, with $s>1$. At last they established an uniqueness result on the condition that one of the two considered solutions is a strong-solution, that is they prove the weak-strong uniqueness.

\vspace{0.1cm}
\noindent In \cite{doi10113710079224X} M. Paicu and A. Zarnescu prove the same results as in \cite{MR2864407} for system \eqref{main_system_xi} when $\xi$ is a general value of $[0,\xi_0]$ for some $0<\xi_0<1$.

\vspace{0.1cm}
\noindent In \cite{doi10113713095015X} F. G. Guill\'en-Gonz\`alez and L. A. Rodr\'iquez-Bellido show the existence and uniqueness of a local in time weak solution on a bounded domain. They also give a regularity criterion which yields such solutions to be global in time. Moreover they prove the global existence and uniqueness of a strong solution provided a viscosity large enough. 

\vspace{0.1cm}
\noindent In \cite{MR3274285} F. G. Guill\'en-Gonz\`alez and L. A. Rodr\'iquez-Bellido prove the existence of global in time weak-solutions, an uniqueness criteria and a mximum principle for $Q$. They also established the traceless and symmetry for $Q$, for any weak solution.

\subsection{Main Results}
Article \cite{MR2864407} is probably one of the best-known research interesting the solvability of \eqref{main_system}, globally in time and in the whole space. However the author's results present some gaps, therefore this article is mainly devoted to fill them, and complete their paper. Now, let us go into the details. 

\vspace{0.1cm}
\noindent
First Paicu and Zarnescu have proven an uniqueness result on the condition that at least one of the considered solutions is a strong solution. This is due to the necessity to control  $(u(t), \,\nabla Q(t))$ in $L^\infty(\RR^2)$, which is strictly correlated to control $(u(t),\,\nabla Q(t))$ in $H^s(\RR^2)$ with $s>1$, thanks to the Sobolev Embedding. However, such necessity turns out from the attempt to estimate the difference between two solutions in the same space the solutions belong to. Here, we are able to overcome the drawbacks thanks to a strategy which is inspired by \cite{MR1813331} and \cite{MR2309504}. Indeed, since the difference between two solutions has a null initial datum, then it is possible to estimate such  difference in less regular spaces than the ones related to the existence part. Hence the cited difficulties disappear and we are able to prove the uniqueness of the weak solutions. Then, our first result reads as follows: 
\begin{theorem}\label{Thm_Uniqueness}
	Let us assume that system \eqref{main_system} admits two weak solutions $(u_i,\,Q_i),\,i=1,\,2$, in the sense of of definition \ref{def_weak_sol}. Then such solutions are equal, 
	$(u_1,\,Q_1) \equiv (u_2,\,Q_2)$.
\end{theorem}

\vspace{0.1cm}
\noindent
The second (and last) gap concerns the propagation of regularity. Paicu and Zarnescu consider initial data $(u_0,\, Q_0)$ in $H^s(\RR^2)\times H^{1+s}(\RR^2)$, with s greater than $1$. Then, they are able to prove that such high-regularity is preserved by the related solution of \eqref{main_system}. Denoting by 
\begin{equation*}
	f(t) := \|u(t)\|_{\Hh^s}^2 + \|\nabla  Q(t)\|_{\Hh^s}^2 ,\quad\quad g(t) :=  \|\nabla u(t)\|_{\Hh^{s}}^2 + \|\Delta  Q(t)\|_{\Hh^s}^2,
\end{equation*}
the major part of their proof releases on the Osgood Theorem, applied on an inequality of the following type:
\begin{equation*}
	\frac{\dd}{\dd t}f(t) +g(t) \leq C f(t)\ln\{e+f(t)\},\quad t\in\RR_+,
\end{equation*}
for a suitable positive constant $C$. However such estimate requires again to control $\|(u(t), \,\nabla Q(t))\|_{L^\infty}$ by $\|(u(t), \,\nabla Q(t))\|_{H^s}$, and this is true only if $s$ is greater than $1$. We fix such lack, namely we extend the propagation for $0<s$, passing through an alternative approach. Indeed we control the $L^\infty$-norm by a different method (see Lemma \ref{Lemma_pre_Osgood} and \eqref{sec3_est12}). Thus, our second result reads as follows:
\begin{theorem}\label{Thm_Regularity}
	Assume that $(u_0,\,Q_0)$ belongs to $H^s(\RR^2)\times H^{1+s}(\RR^2)$, with $0<s$. Then, the solution $(u,\,Q)$ given by Theorem \ref{Thm_Existence} fulfills 
	\begin{equation*}
		(u,\,\nabla Q)\in L^\infty_{t,loc}\Hh^s(\RR^2)\cap L^2_{t,loc}\Hh^{s+1}(\RR^2) .
	\end{equation*}
\end{theorem}

\vspace{0.1cm}
\noindent
Now, we have also chosen to perform an existence result, for the completeness of this project, although it has already been proven by Paicu and Zarnescu. Nevertheless, here we use an alternative approach to prove the theorem. Indeed in \cite{MR2864407}, the authors utilize a Friedrichs scheme, regularizing every equation of \eqref{main_system}, while our method is based on a coupled technique between the Friedrichs scheme and the Schaefer's fixed point theorem, so that we have only to regularize the momentum equation of \eqref{main_system}. This method is inspired by \cite{MR1329830}, where F. Lin use a modified Galerkin method coupled with the Schauder fixed point theorem, in the proof of an existence result. Then we are going to prove the following Theorem:
\begin{theorem}\label{Thm_Existence}
	Assume that $(u_0,\,Q_0)$ belongs to $L^2(\RR^2)\times H^1(\RR^2)$, then system \eqref{main_system} admits a global in time weak solution $(u,\,Q)$, in the sense 
	of definition \ref{def_weak_sol}. 
\end{theorem}

\vspace{0.2cm}
\noindent
The structure of this article is over simplistic: in the next section we recall some classical tools which are useful for our proofs, in section three we deal with Theorem \ref{Thm_Existence}, the existence of weak solutions, in section four and five we establish Theorem \ref{Thm_Uniqueness}, i.e. such solutions are unique, and finally in section six we determine Theorem \eqref{Thm_Regularity}, proving the propagation of regularities. We put forward in the appendix some technical details, for the simplicity of the reader.

\section{Preliminaries and Notations}

\noindent
In this section we illustrate some widely recognized mathematical tools and moreover we report some notations which are going to be extensively utilized in this research. 

\subsection{Sobolev and Besov Spaces}
First, let us introduce the spaces we are going to work with (we refer the reader to  \cite{MR2768550} for an exhaustive study and more details) . We recall the well-known definition of Homogeneous Sobolev space $\Hh^s$ and Non-Homogeneous Sobolev Space $H^s$:
\begin{definition}
	Let $s\in\RR$, the Homogeneous Sobolev Space $\Hh^s$ (also denoted $\Hh^s(\RR^2)$) is the space of tempered distribution $u$ over $\RR^2$, the Fourier transform of 
	which belongs to $L^1_{loc}(\RR^2)$ and it fulfills
	\begin{equation*}
		\| u \|_{\Hh^s} := \int_{\RR^2} |\xi|^{2s} |\hat{u}(\xi)|^2\dd \xi < \infty.
	\end{equation*}
	Moreover $u$ belongs to the Non-Homogeneous Sobolev Space $H^s$ (or $H^s(\RR^2)$) if $\hat{u}\in L^2_{loc}(\RR^2)$ and 
	\begin{equation*}
		\| u \|_{\Hh^s} := \int_{\RR^2} (1+|\xi|)^{2s} |\hat{u}(\xi)|^2\dd \xi < \infty.
	\end{equation*}
\end{definition}
\noindent
$H^s$ is an Hilbert space for any real $s$,  while $\Hh^s$ requires $s<d/2$, otherwise it is Pre-Hilbert. Their inner products are
\begin{equation*}
	\langle	u,\,v \rangle_{H^s} = \int_{\RR^2}  (1+|\xi|)^{2s} \hat{u}(\xi)\overline{\hat{v}(\xi)}\dd \xi
	\quad \text{and} \quad 
	 \langle	u,\,v \rangle_{\Hh^s} = \int_{\RR^2}  |\xi|^{2s} \hat{u}(\xi)\overline{\hat{v}(\xi)}\dd \xi,
\end{equation*}
respectively. Even if such dot-products are the most common ones, from here on we are going to use the ones related to the Besov Spaces (at least for the homogeneous case). Hence, first 
we need to define them. In order to do that, it is fundamental to introduce the Dyadic Partition. Let $\chi=\chi(\xi)$ be a smooth function whose support is inside the the ball $|\xi|\leq 1$. Let us assume that $\chi$ is identically equal to $1$ in $|\xi|\leq 3/4$, then, imposing $\varphi_q (\xi ) := \chi(\xi2^{-q-1}) - \chi\big(\xi2^{-q})$ for any $q\in\ZZ$, we define the Homogeneous Litlewood-Paley Block $\Dd_q$ by
\begin{equation*}
	\Dd_q f := \Ff^{-1}(\varphi_q \hat{f} ) \in \Ss', \quad \text{for any}\; f\in \Ss'.
\end{equation*}
Moreover we denote by $\Sd_j$ the operator $\sum_{q\leq j-1} \Dd_q$, for any $j\in\ZZ$. 
We can now present the definition of Homogeneous Besov Space
\begin{definition}
	For any $s\in \RR$ and $(p,r)\in [1,\infty]^2$, we define $\BB_{p,r}^s$ as the set of tempered 
	distribution $f$ such that
	\begin{equation*}
		\|f\|_{\BB_{p,r}^s}:= \|2^{sq}\|\dot{\Delta}_q f\|_{L^p_x}\|_{l^r(\ZZ)}
	\end{equation*}
	and for all smooth compactly supported function $\theta$ on $\RR^2$ we have
	\begin{equation*}
		\lim_{\lambda\rightarrow +\infty} \theta (\lambda D)f = 0\quad\text{in}\quad L^\infty(\RR^2).
	\end{equation*}
\end{definition}

\noindent
It is straightforward that the space $\BB_{2,2}^s$ and $\Hh^s$ coincides for any real $s$, and their norms are equivalent, so we will use the following abuse of notation from here on:
\begin{equation*}
	\langle u,\,v \rangle_{\Hh^s} := \langle u,\,v \rangle_{\BB^s_{2,2}} = \sum_{q\in \ZZ} 2^{2qs} \langle \Dd_q u, \, \Dd_q v\rangle_{L^2_x},
\end{equation*}
where $\langle \cdot, \cdot \rangle_{L^2}$ is the common inner product of $L^2_x:=L^2(\RR^2)$.

\vspace{0.1cm}
\noindent A profitable feature of the Homogeneous Besov space with negative index $s$ is the following one (see Proposition $2.33$ of \cite{MR2768550})
\begin{prop}\label{prop_besov_s_negative}
	Let $s<0$ and $1\leq p,r\leq \infty$. Then $u$ belongs to $\BB_{p,r}^s$ if and only if
	\begin{equation*}
		\big(2^{qs}\|\Sd_q u \|_{L^p_x} \big)_{q\in\ZZ} \in L^r(\ZZ).
	\end{equation*}
	Moreover there exists two positive constant $c_s$ and $C_S$ such that
	\begin{equation*}
		c_s\| u	\|_{\BB_{p,r}^s} \leq \| \big(2^{qs}\|\Sd_q u \|_{L^p_x} \big)_{q\in\ZZ} \|_{l^r(\ZZ)} \leq C_s \| u	\|_{\BB_{p,r}^s}.
	\end{equation*}	 
\end{prop}

\subsection{Homogeneous Paradifferential Calculus}
\noindent
In this subsection we give some hints about how the product acts between $\Hh^s$ and $\Hh^t$, for some appropriate real $s$ and $t$. We present several tools which will play a major part in all our proofs. First, let us begin with the following Theorem, whose proof is put forward in the appendix:
\begin{theorem}\label{product_Hs_Ht}
	Let $s$ and $t$ be two real numbers such that $ |s|$ and $|t|$ belong to $[0,1)$. Let us assume that $s+t$ is positive, then for every 
	$a\in \Hh^s$ and for every $b\in \Hh^t$, the product $ab$ belongs to $\Hh^{s+t-1}$ and there exists a positive 
	constant (not dependent on $a$ and $b$) such that
	\begin{equation*}
		\|a b\|_{\Hh^{s+t-1}}\leq C\|a\|_{\Hh^{s}}\|b\|_{\Hh^t}
	\end{equation*}
\end{theorem}
\noindent
We have already remarked that $\Hh^s$ coincides with $\BB_{2,2}^s$ and this correlation allows us to incorporate in our tools the so-called Bony decomposition:
\begin{equation*}
	f g = \Th_fg + \Th_g h +\Rd(f,\,g),\quad 
	\text{with}\quad 
	\Th_fg := \sum_{q\in\ZZ} \Sd_{q-1}f \Dd_q g \quad\text{and}\quad
	\Rd(f,\,g):= \sum_{q\in\ZZ,\,|l|\leq 1}\Dd_q f \Dd_{q+l}g.
\end{equation*}
However, such decomposition is not going to be useful for every challenging estimation, so that we are going to use the so called symmetric decomposition, in order to overcome the 
drawbacks. Here, we present directly the matrix-formulation of such decomposition, since it will be used on such framework. Let $q$ be an integer, and $A$, $B$ be $N\times N$ matrices, whose components are homogeneous temperate distributions. Denoting by
\begin{equation}\label{simmmetric_decomposition}
\begin{array}{ll}
	 \J_q^1(A,B)	:=\sum_{|q-q'|\leq 5	}	[\Dd_q, \,		\Sd_{q'-1}A]			\Dd_{q'}	B,	
	&\J_q^3(A,B)	:=							\Sd_{q -1}	A							\Dd_{q}		B,\\
	\\
	 \J_q^2(A,B)	:=\sum_{|q-q'|\leq 5	}	(	\Sd_{q'-1}A	-\Sd_{q-1}A)	\Dd_{q}	\Dd_{q'}	B,		
	&\J_q^4(A,B)	:=\sum_{ q'   \geq q- 5 }	\Dd_q(\Dd_{q'}A\,						\Sd_{q'+2}	B),
\end{array}
\end{equation}
the following classical feature concerning the product  $AB$, is satisfied:
\begin{equation}\label{product_decomposition}
\Dd_q(AB)	=		\J_q^1(A,B)
				+	\J_q^2(A,B)
				+	\J_q^3(A,B)
				+	\J_q^4(A,B),
\end{equation}
for any integer $q$.

\subsection{The Frobenius Norm}
Before beginning with the proofs of our main results, let us give the following remark:
\begin{remark}
	The most common inner product defined on $\MM_3(\RR)$ (the $3\times3$ real matrices) is determined by:
	\begin{equation*}
		A\cdot B = \sum_{i,j=1}^3A_{ij}B_{ij} = \trc \{ \tr A B\},\quad\quad\text{for any}\quad A,\,B\in \MM_3(\RR).
	\end{equation*}
	Hence, if at least one of the two matrices is symmetric, for instance $A$, then we obtain
	\begin{equation}\label{Frobenius_inner_product}
		A\cdot B = \trc\{AB\},
	\end{equation}
	which determines the well-known Frobenius norm of a matrix $|A|:= \sqrt{\trc\{A^2\}}$. Since any solution $(u,\,Q)$ for \eqref{main_system} fulfills
	\begin{equation*}
		Q(t,x)\in S_0 := \big\{ A\in \MM_3(\RR),\,\trc\{ A \} = 0\quad\text{and}\quad \tr A = A \big\},
	\end{equation*}
	for almost every $(t,x) \in \RR_+\times \RR^2$ (see \cite{MR2864407} and \cite{MR3274285}), then from here on we will repeatedly use \eqref{Frobenius_inner_product}.
\end{remark}
\noindent Moreover, we will use the symbol $\lesssim$ (instead of $\leq$) which is defined as follows: for any non-negative real numbers $a$ and $b$, $a\lesssim b$ if and only if there exists a positive constant $C$ (not dependent on $a$ and $b$) such that $a \leq C\,b$.

\section{Weak Solutions}
\noindent
This section deals with the existence of weak solutions for \eqref{main_system} in the sense of definition \ref{def_weak_sol}. As we have already explained, we are going to proceed with a coupled method between the Friedrichs scheme and the Schaefer's Theorem. Hence, before going on, let us recall the widely recognized Schaefer's fixed point Theorem
\begin{theorem}
	Let $\Psi$ be a continuous and compact mapping of a Banach Space $X$ into itself, 
	such that the set $\{\,x\in X\,:\:x=\lambda\, \Psi x\;\,\text{for some}\;\, 0
	\leq \lambda\leq 1\}$ is bounded. Then $T$ has a fixed point.
\end{theorem}
\noindent
First, we introduce one of the key ingredient of our proofs, namely the mollifying 
operator $J_n$ defined by
\begin{equation*}
	\Ff( J_n f)(\xi) = 1_{[\frac{1}{n},\,n]}(\xi)  	\quad\quad 
													\text{for } \xi \in \RR^2_\xi,
\end{equation*}
which erases the high and the low frequencies. 

\noindent
We claim the existence and uniqueness of a solution for the following system 
\begin{equation}\label{system_friedrichs}
	\tag{$P_n$}
	\begin{cases}
	\;	\partial_t Q + ( J_nu\cdot \nabla  Q) - 
			J_n\Omega  Q +  Q J_n\Omega = \Gamma H( Q)	&[0,T) \times	\RR^2,\\
	\;	\partial_t u + 	J_n \Pp(\, J_nu \cdot 
									\nabla J_n u\, )
		-\nu\Delta u =  LJ_n\Pp\Div\,\{\;Q \Delta  Q - 
						\Delta  Q  Q 
						- \nabla  Q\odot \nabla Q\;\}
														&[0,T)	\times	\RR^2,\\
	\;	\Div\,u = 0										&[0,T)	\times	\RR^2,\\
	\;	(u,\,Q)_{t=0} = (u_0,\, Q_0)					&\quad\quad\quad\;\;\RR^2,
	\end{cases}
\end{equation}
where $\Pp$ stands for the Leray projector operator, which is determined by
\begin{equation*}
	\Ff\{\,\Pp f\,\}(\xi) := 
	\hat{f}(\xi) - \frac{\xi}{|\xi|}\,\frac{\xi}{|\xi|}\cdot \hat{f}(\xi), 
	\quad\quad
	\text{for}\quad f \in (L^p_x)^2, \quad\text{with}\quad 1<p<\infty,
\end{equation*}
and $T$ is a positive real number. It is well known that $\Pp$ is a bounded operator of $(L^p_x)^2$ into itself 
when $p\in (1,\infty)$.
\begin{remark}\label{remark_weak_sol}
	We say $(u,\,Q)$ is a weak solution of the problem \eqref{system_friedrichs}, 
	provided that
	\begin{equation*}
		u \in C([0,T], L^2_x)\cap L^2(0,T;\, \Hh^1)\,,\quad \quad
		Q \in C([0,T];H^1)\cap L^2(0,T; \Hh^2)
	\end{equation*}
	and \eqref{system_friedrichs} is valid in the distributional sense.
\end{remark}
\noindent
The following proposition plays a major part in our main proof, since it allows us 
to control the $L^p_x$-norm of $Q$ only by $Q_0$.

\begin{prop}\label{prop_friedrichs}
	Suppose that $u\in C([0,T], L^2_x)\cap L^2(0,T; \Hh^1)$ and moreover that  
	$Q\in C([0,T],H^1)\cap L^2(0,T; \Hh^2)$ is a 
	weak solution of
	\begin{equation*}
		\partial_t Q +  u\cdot \nabla  Q - \Omega  Q +  Q \Omega - \Gamma L\Delta Q = 
		\Gamma P( Q)\quad\text{in}\quad[0,T)\times \RR^2,
		\quad\quad\text{and}\quad\quad
		Q_{t=0} = Q_0 \in H^1.
	\end{equation*}
	Then, for every $2\leq q<\infty$, the following estimate is fulfilled
	\begin{equation*}
		\| Q(t) \|_{L^q_x} \leq \|Q_0\|_{H^1}\exp\{Ct\},
	\end{equation*}	
	for a suitable positive constant $C$ dependent only on $q$, $\Gamma$, $a$, $b$ and $c$.
\end{prop}
\begin{proof}
	Fixing $p\in (1,\infty)$, We multiply both left and right-hand side by $2pQ\,\trc\{Q^2\}^{p-1}$, we take the trace and we 
	integrate in $\RR^2$, obtaining that
	\begin{equation*}
		\frac{\dd}{\dd t}\|Q(t)\|_{L^{2p}_x}^{2p} - 
		\Gamma 2Lp \langle Q(t)\trc\{Q(t)^2\}^{p-1},\Delta Q(t)\rangle_{L^2_x} = 
		2\Gamma p \int_{\RR^2}\trc\{Q(t)^2\}^{p-1} \trc\{\, P(Q(t,x))Q(t,x) \,\}\dd x, 
	\end{equation*}
	for almost every $t\in (0,T)$, where we have used $\Div\, u = 0$ and $\trc\{Q\Omega Q-\Omega Q^2\}=0$ . 
	First, analyzing the second term on the left-hand side, integrating by parts, we determine the following identity:
	\begin{align*}
		-\langle 2pQ\trc\{Q^2\}^{p-1},\,&\Delta Q \rangle_{L^2}
		=
		\sum_{i=1}^2\Big[
		2p \int_{\RR^2}\trc\{Q^2\}^{p-1}\trc\{(\partial_i Q)^2\}  + 
		2p \int_{\RR^N}\partial_i[\trc\{Q^2\}^{p-1}] \trc\{Q\partial_iQ \}\Big]\\
		&=2p \int_{\RR^N} \trc\{Q^2\}^{p-1}|\nabla Q|^2  + 
		4p(p-1) \int_{\RR^N}\trc\{Q^2\}^{p-2} |\nabla [\trc\{Q^2\}] |^2\geq 0,
	\end{align*}
	which allows us to obtain
	\begin{equation*}
		\frac{\dd}{\dd t}\|Q(t)\|_{L^{2p}}^{2p}\leq \Gamma \int_{\RR^2}2p\trc\{Q^2\} \trc\{\, P(Q(t,x))Q(t,x) \,\}\dd x.
	\end{equation*}		
	Now, we deal with the right-hand side by a direct computation, observing that
	\begin{align*}
		\int_{\RR^2}\trc\{Q^2\}^{p-1} \trc\{ P(Q)Q\}\dd x 
		&= 
		\Gamma \int_{\RR^2}
		\Big[ \,
			- a\, \trc\{ Q^2 \}^{p} 
			+ b\, \trc\{Q^2\}^{p-1}\trc\{ Q^3 \} 
			- c\, \trc\{ Q^2 \}^{p+1}\,
		\Big]\\
		&\lesssim 
		\|Q\|_{L^{2p}_x}^{2p} -\frac{c}{2} \|Q\|_{L^{2(p+1)}_x}^{2(p+1)}
		\lesssim
		\|Q\|_{L^{2p}_x}^{2p},
	\end{align*}
	where we have used the following feature about a symmetric matrix with null trace:
	\begin{equation*}
		\big|\int_{\RR^2}\trc\{Q^2\}^{p-1}\trc\{Q^3\} \big|\leq  \ee\|Q\|_{L^{2(p+1)}}^{2(p+1)} + \frac{1}{\ee}\|Q\|_{L^{2p}}^{2p},
	\end{equation*}
	for a positive real $\ee$, small enough.
	Indeed, if $Q$ has $\lambda_1$, $\lambda_1$, and $-\lambda_1-\lambda_2$ as eigenvalues, we achieve that 
	$\trc\{Q^3\}=-3\lambda_1\lambda_2(\lambda_1+\lambda_2)$ and $\trc\{Q^2\} = 2(\lambda_1^2+\lambda_2^2+\lambda_1\lambda_2)$, hence
	\begin{equation*}
		|\trc\{Q^3\}|
		\lesssim 
		\ee\lambda_1^2\lambda_2^2 + \frac{1}{\ee}(\lambda_1^2+ \lambda_2^2+2\lambda_1\lambda_2)
		\lesssim
		\ee(\lambda_1^2+\lambda_2^2+\lambda_1\lambda_2)^2 + \frac{1}{\ee}(\lambda_1^2+ \lambda_2^2+\lambda_1\lambda_2)
		\lesssim
		\ee\trc\{Q^2\}^2+\frac{1}{\ee}\trc\{Q^2\}.
	\end{equation*}
	Therefore, we deduce that
	\begin{equation}\label{bound_int_Q^3}
		\big|\int_{\RR^2}\trc\{Q^2\}^{p-1}\trc\{Q^3\} \big| \lesssim 
		\ee
		\int_{\RR^N}\trc\{Q^2\}^{(p+1)}+
		\frac{1}{\ee}
		\int_{\RR^N}\trc\{Q^2\}^{p}.
	\end{equation}
	Summarizing the previous consideration, we get
	\begin{equation*}
		\frac{1}{2}\frac{\dd}{\dd t}\|Q(t)\|_{L^{2p}_x}^{2p} \lesssim 
		\|Q(t)\|_{L^{2p}_x}^{2p},
	\end{equation*}
	so that the statement is proved, thanks to the Gronwall's inequality.
\end{proof}

\noindent
Now, let us focus on one of the main theorems of this section, which reads as follows:
\begin{theorem}
	Let $n$ be a positive integer and assume that $(u_0,\,Q_0)$ belongs to $L^2_x$. 
	Then, system \eqref{system_friedrichs} admits a unique local weak solution.
\end{theorem}
\begin{proof}
	The key method of the proof relies on the Schauder's Theorem. We define the 
	compact operator $\Psi$ from $C([0,T], L^2_x)^2\cap L^2(0,T;\, \Hh^1)^2$ to itself 
	as follows: $(\Psi(u),\, Q)=: (\tilde{u},\,Q)$ is the unique weak solution (in 
	the sense 	of remark \ref{remark_weak_sol}) of the following Cauchy problem:
	\begin{equation*}
	\begin{cases}
	\;	\partial_t Q + (J_n u\cdot \nabla  Q)
	 	- J_n\Omega  Q +  QJ_n\Omega = \Gamma H( Q)			&[0,T) \times	\RR^2,\\
	\;	\partial_t \tilde{u} + 
		J_n \Pp(\, J_n\tilde{u} \cdot 
						\nabla J_n \tilde{u}\, ) 	
		- \nu\Delta \tilde{u} = LJ_n\Pp\Div\,\{\;Q \Delta Q 
		- \Delta  Q  Q - \nabla  Q\odot \nabla Q\;\}		&[0,T)	\times	\RR^2,\\
	\;	\Div\,\tilde{u} = 0									&[0,T)	\times	\RR^2,\\
	\;	(\tilde{u},\,Q)_{t=0} = (u_0,\, Q_0)				&\quad\quad\quad\;\;\RR^2.
	\end{cases}
	\end{equation*}

	\noindent
	We claim that the hypotheses of the Schauder's Theorem are fulfilled, namely 
	$\Psi$ is a compact mapping of $X:=C([0,T], L^2_x)\cap L^2(0,T;\, \Hh^1)$ into 
	itself, and the set 
	$\{\,u=\lambda\,\Psi (u)\;\,\text{for some}\;\,0\leq \lambda\leq 1\}$.
	is bounded. First we deal with the compactness of $\Psi$. Considering a bounded 
	family $\FF$ of $X$, we claim that the closure of $\Psi(\FF)$ is compact in $X$. 
	If we prove that $\Psi(\FF)$ is an uniformly bounded and equicontinuous 
	family of $C([0,T]; L^2_x)$ and moreover that 
	$\{ \,\Psi(u)(t)\;\text{with }t\in [0,T]\;\text{and }u\in \FF\}$ is a compact 
	set of $L^2_x$, then the result is at least valid as $\Psi$ mapping of $X$ into 
	$C([0,T], L^2)$, thanks to the Arzel\`{a}-Ascoli Theorem. Multiplying the first 
	equation by $Q-\Delta Q$ and integrating in $\RR^2$, we get
	\begin{align*}
		\frac{1}{2}
		&\frac{\dd}{\dd t}
		\Big[ 
			\| \nabla Q \|_{L^2_x}^2 +
			\|		  Q \|_{L^2_x}^2
		\Big] + 
		\Gamma L 
			\big(
			\|\nabla Q\|_{L^2_x}^2+
			\|\Delta Q\|_{L^2_x}^2 
			\big)
		= 
		\int_{\RR^2}
		\Big[			 
			\trc\{(J_n \Omega Q -  Q J_n\Omega)\Delta Q\,\} -\\&-
			\trc\{(J_n u\cdot \nabla Q )	\Delta Q \,\}
		\Big] +
		\Gamma L
		\int_{\RR^2}
		\Big[ \,
			a\, \trc\{ Q	\Delta Q \} -  
			b\, \trc\{ Q^2 	\Delta Q \} +
			c\, \trc\{ Q	\Delta Q \}\trc\{ Q^2 \}	\,
		\Big]+ \\&
		\quad\quad\quad\quad\quad\quad\quad\quad\quad\quad\quad\quad\quad\quad\quad\quad\quad\quad+
		\Gamma \int_{\RR^2}
		\Big[ \,
			a\, \trc\{ Q^2 \} - 
			b\, \trc\{ Q^3 \} +
			c\, \trc\{ Q^2 \}^2	\,
		\Big],
	\end{align*}
	almost everywhere in $(0,T)$, which allows us to achieve
	\begin{align*}
		\frac{\dd}{\dd t}
		&\Big[ 
			\| 	\nabla	Q(t)	\|_{L^2_x}^2 +
			\|		  	Q(t)	\|_{L^2_x}^2
		\Big] + 
		\Gamma L 
			\|	\Delta	Q(t)	\|_{L^2_x}^2 \leq\\&
		\leq
		C_n
		\big(	1					+	\| u(t) 		\|_{L^2_x}^2	\big)
		\big(	\| Q(t) \|_{L^2}^2	+	\| \nabla Q(t) 	\|_{L^2_x}^2	\big)	+
		\frac{\Gamma L}{100}
		\|	\Delta	Q(t)	\|_{L^2_x},	
	\end{align*}
	where $C_n$ is a positive constant dependent on $n$. Therefore, we realize that 
	the family composed by $Q=Q(u)$ as $u$ ranges on $\FF$ is a bounded family 
	in $C([0,T];H^1)\cap L^2(0,T; \Hh^2)$.  Now, multiplying the second equation by 
	$\tilde{u}$ we get the following equality:
	\begin{equation*}
		\frac{1}{2}
		\frac{\dd}{\dd t}
		\|					\tilde{u}(t)	\|_{L^2_x}^2 +
		\nu
		\|	 \nabla			\tilde{u}(t)	\|_{L^2_x}^2
		=
		L\int_{\RR^2}
		\trc\{\,
		\big(
			\nabla Q\odot \nabla Q + 
			Q \Delta Q - \Delta Q\, Q 
		\big) \nabla \tilde{u} \,\}(t,x)\dd x=: F(t),
	\end{equation*}
	for almost every $t\in (0,T)$. Thus it turns out that
	\begin{equation}\label{Thm_Friedrichs_est_tilde_u}
		\frac{\dd}{\dd t}
		\|					\tilde{u}(t)			\|_{L^2_x}^2 +
		\nu
		\|	 \nabla			\tilde{u}(t)			\|_{L^2_x}^2
		\leq
		|F(t)|
		\leq
		\|	(Q(t),\, \nabla Q(t),\, \Delta Q(t)) 	\|_{L^2_x}^2 +
		C_n
		\|					\tilde{u}(t)			\|_{L^2_x}^2,	
	\end{equation}
	where $C_n>0$ depends on $n$. Here, we have used the feature $J_n \tilde{u} = 
	\tilde{u}$, which comes from the uniqueness of the solution for the second equation,
	so that $\| \nabla \tilde{u}\|_{L^\infty_x}\leq C_n \|\tilde{u}\|_{L^2_x}$. Summarizing 
	the previous considerations and thanks to the Gonwall's inequality we discover that 
	$\Psi(\FF)$ is a bounded family in $X$, so in $C([0,T], L^2)$. Moreover, from 
	\eqref{Thm_Friedrichs_est_tilde_u} and the previous result, 
	it turns out that $|F(t)|$ is bounded on $[0,T]$, uniformly in $u\in \FF$. 
	Hence $\Psi(\FF)$ is an equicontinuous family of $C([0,T]; L^2_x)$. Finally, 
	because $J_n \tilde{u} = \tilde{u}$, we get that 
	$\{ \,\Psi(u)(t)\;\text{with }t\in [0,T]\;\text{and }u\in \FF\}$ is a subset of a 
	bounded $L^2_x$-family composed by functions with Fourier-transform supported in the 
	anulus $\Cc(1/n, n)$, which is a compact family of $L^2_x$. Summarizing all the 
	previous consideration, we get that $\Psi(\FF)$ is compact in $C([0,T], L^2_x)$ 
	thanks to the Arzelà-Ascoli Theorem.
	
	\noindent It remains to prove that $\Psi(\FF)$ is compact in $L^2(0,T;\Hh^1)$, so 
	that $\Psi$ is a compact mapping of $X$ into itself. Since $J_n \Psi(u(t)) = 
	\Psi(u(t))$ for every $u\in \FF$ and $t\in (0,T)$, the precompactness of 
	$\Psi(\FF)$ in $L^2(0,T;\Hh^1)$ is equivalent to the precompactness of  $\Psi(\FF)$ 
	in $L^2( (0,T)\times \RR^2\,)$. Recalling that $\Psi(\FF)$ is precompact in 
	$C([0,T], L^2_x)$  which is embedded in $L^2( (0,T)\times \RR^2\,)$ (for $T$ finite), 
	then we determine the result, so that, in conclusion $\Psi$ is a compact operator 
	from $X$ to itself.
	
	\noindent
	Now, we deal with the Schaefer's Theorem hypotheses, namely the set 
	$\{\, u=\lambda \Psi(u) \;\text{for some}\;\lambda \in (0,1)\}$  
	is a bounded family of $X$. First, we point out that if $u=\lambda \psi(u)$, then 
	the couple $(u,\,Q)$ is a solution for
	\begin{equation*}
	\begin{cases}
	\;	\partial_t Q + \lambda J_n u\cdot \nabla  Q
	 	- \lambda J_n\Omega Q + \lambda QJ_n\Omega = \Gamma H( Q)
	 														&[0,T) \times	\RR^2,\\
	\;	\partial_t u + 
		J_n \Pp(\, J_n u \cdot 
						\nabla J_n u\, ) 	
		- \nu\Delta u = LJ_n\Pp\Div\,\{\;Q \Delta Q 
		- \Delta  Q  Q - \nabla  Q\odot \nabla Q\;\}		&[0,T)	\times	\RR^2,\\
	\;	\Div\,u = 0											&[0,T)	\times	\RR^2,\\
	\;	(u,\,Q)_{t=0} = (u_0,\, Q_0)						&\quad\quad\quad\;\;\RR^2.
	\end{cases}
	\end{equation*}
	We multiply the first equation by $Q-\Delta Q$, the second equation by 
	$u$, we integrate everything in $\RR^N$ and we sum the results, obtaining:
	\begin{align*}
		\frac{\dd }{\dd t}
		\Big[\,
			\|			Q	\|_{L^2}^2	+
			\|	\nabla	Q	\|_{L^2}^2	+
			\|			u	\|_{L^2}^2 	
		\Big]							+
		\Gamma L
			\|	\nabla	Q	\|_{L^2}^2	+
		\Gamma L	
			\|	\Delta	Q	\|_{L^2}^2	+
		\nu \|	\nabla 	u	\|_{L^2}^2	=
		 \lambda 
		 \langle J_n u\cdot \nabla Q,			\, Q - \Delta Q	\rangle_{L^2} + \\ +
		 \lambda
		 \langle J_n \Omega Q - Q J_n\Omega,	\, Q - \Delta Q	\rangle_{L^2} +
		 \Gamma
		 \langle P(Q),							\, Q - \Delta Q	\rangle_{L^2} -
		 \langle J_n u\cdot \nabla J_n u, 		\,\nabla J_n u	\rangle_{L^2} + \\ +
		L\langle Q\Delta Q - \Delta Q \,Q,		\,\nabla J_n u	\rangle_{L^2} +
		L\langle \nabla Q\odot\nabla Q,			\,\nabla J_n u	\rangle_{L^2}.
	\end{align*}
	According to $\| J_nu \|_{L^\infty} + \|\nabla J_n \|_{L^\infty} \leq C_n \| u \|_{L^2}$, 
	up to a positive constant $C_n$ dependent on $n$, it is not computationally demanding to 
	achieve	the following estimate:
	\begin{align*}
		\frac{\dd }{\dd t}
		\Big[\,
			\|			Q	\|_{L^2_x}^2	+
			\|	\nabla	Q	\|_{L^2_x}^2	+
			\|			u	\|_{L^2_x}^2 	
		\Big]			&				+	
		\Gamma L	
			\|	\Delta	Q	\|_{L^2_x}^2	+
		\nu \|	\nabla 	u	\|_{L^2_x}^2
		\leq	\\ &\leq
		\tilde{C}_n
		\Big[\,
			\|			Q	\|_{L^2_x}^2	+
			\|	\nabla	Q	\|_{L^2_x}^2	+
			\|			u	\|_{L^2_x}^2 	
		\Big]							+
		\frac{\nu}{100}
		\| \nabla u 	\|_{L^2_x}^2+
		\frac{\Gamma L}{100}
		\|	\Delta Q	\|_{L^2_x}^2.
	\end{align*}
	Therefore, thanks to the Gronwall's inequality, we detect the following estimate:
	\begin{align*}
		\|			Q	\|_{L^\infty(0,T;L^2_x)}^2	 +
		\|	\nabla	Q	\|_{L^\infty(0,T;L^2_x)}^2	&+
		\|			u	\|_{L^\infty(0,T;L^2_x)}^2	 + \\&+
		\|	\Delta	Q	\|_{L^2(0,T;L^2_x)}^2	+
		\|	\nabla 	u	\|_{L^2(0,T;L^2_x)}^2 
		\lesssim
		\| (u_0,\,Q_0,\, \nabla Q_0)\|_{L^2_x}
		e^{C_n T},
	\end{align*}
	so that, the family $\{\, u=\lambda \Psi(u) \;\text{for some}\;0\leq \lambda \leq 1\}$ is 
	bounded in $X$. Hence, applying the Schaefer's fixed point Theorem, we conclude that there exists
	a fixed point for $\Psi$, namely there exists a weak solution $(u,\,Q)$ 
	(in the sense of remark \ref{remark_weak_sol}) for the system \eqref{system_friedrichs}.
\end{proof}
\begin{remark}
	In the previous proof $T$ has only to be bounded, and it has no correlation with the initial data, so that the solution $(u^n,\,Q^n)$ of system 
	\eqref{system_friedrichs}, given by Proposition \ref{prop_friedrichs}, it should be supposed to belong to 
	\begin{equation*}
		C(\RR_+,L^2_x)\cap L^2_{loc}(\RR_+,\Hh^1)\times C(\RR_+,H^1)\cap L^2_{loc}(\RR_+,\Hh^2).
	\end{equation*}
\end{remark}

\noindent
We are now able to prove our main existence result, namely Theorem \ref{Thm_Existence}.

\begin{proof}[Proof of Theorem \ref{Thm_Existence}]
	Let us fix a positive real $T$ and let $(u^n,\,Q^n)$ be the solution of \eqref{system_friedrichs} given by Proposition \ref{prop_friedrichs}, for any 
	positive integer $n$. We analyse such solutions in order to develop some $n$-uniform bound for their norms, which will allow us to apply some classical 
	methods about compactness and weakly convergence.
	
	\noindent
		We multiply the first equation of \eqref{system_friedrichs} by $Q^n-L\Delta Q^n$, the second one by $u^n$, we integrate in $\RR^2$ and finally we sum the results, 
	obtaining the following identity
	\begin{equation}\label{neo}
	\begin{aligned}
		\frac{\dd}{\dd t}\Big[	\|	u^n	\|_{L^2_x} &+ \|	Q^n	\|_{L^2_x} + \Gamma L\|	\nabla Q^n \|^2_{L^2_x}	\Big] + 
		\nu		\|	\nabla u^n	\|_{L^2_x} + 
		\Gamma L	\|	\nabla Q^n	\|_{L^2_x} + 
		\Gamma L^2	\|	\Delta Q^n	\|_{L^2_x} =  \\& =
	\underbracket[0.5pt][1pt]{
		- 			\langle	u_n	\cdot \nabla Q_n,\, 					Q_n	\rangle_{L^2_x} 
	}_{=0}	
	\underbracket[0.5pt][1pt]{	
		+ L			\langle	u_n	\cdot \nabla Q_n,\,				\Delta	Q_n \rangle_{L^2_x}
	}_{\Bb}
	\underbracket[0.5pt][1pt]{
		+			\langle \Omega_n Q_n - Q_n \Omega_n,\,				Q_n	\rangle_{L^2_x}
	}_{=0} -\\&
	\underbracket[0.5pt][1pt]{
		- L			\langle \Omega_n Q_n - Q_n \Omega_n,\,		\Delta	Q_n	\rangle_{L^2_x}
	}_{\Aa}	 	
		+ \Gamma 	\langle	P(Q_n),\,									Q_n	\rangle_{L^2_x} 
		- \Gamma L	\langle	P(Q_n),\,							\Delta	Q_n	\rangle_{L^2_x} -\\&
	\underbracket[0.5pt][1pt]{
		- 			\langle u_n\cdot\nabla u_n,\,						u_n	\rangle_{L^2_x} 
	}_{=0} 
	\underbracket[0.5pt][1pt]{
		- L			\langle Q_n \Delta Q_n - \Delta Q_n Q_n,\,	\nabla	u_n	\rangle_{L^2_x} 
	}_{\Aa \Aa}
	\underbracket[0.5pt][1pt]{
		- L			\langle\Div\{\nabla Q_n\odot\nabla Q_n\},\,			u_n	\rangle_{L^2_x}
	}_{\Bb\Bb}.
	\end{aligned}
	\end{equation}
	First, let us observe that $\Aa + \Aa\Aa = 0$ thanks to Lemma \ref{apx_lemma_omega_Q_u}. Moreover $\langle u_n\cdot \nabla Q_n,\,Q_n\rangle_{L^2_x}$ and 
	$\langle u_n\cdot \nabla u_n, u_n\rangle_{L^2_x}$ are null , because of the divergence-free condition of $u_n$, while 
	$\langle \Omega_n Q_n - Q_n \Omega_n,\,\Delta Q_n\rangle_{L^2_x}$ is zero since $Q_n$ is symmetric. Furthermore $\Bb+\Bb\Bb = 0$ since the following identity is 
	satisfied:
	\begin{equation*}
		\trc\{u_n\cdot\nabla Q_n\,\Delta Q_n\} = \Div \{\nabla Q_n\odot\nabla Q_n\}\cdot u_n - \Div \{ u_n (\,|\partial_1 Q_n|^2 + |\partial_2 Q_n|^2) \}.
	\end{equation*}
	Recalling \eqref{bound_int_Q^3} with $p=1$, it turns out that
	\begin{equation*}
		 \Gamma 	\langle	P(Q_n),\,									Q_n	\rangle_{L^2} \lesssim 
		 \|	Q_n \|_{L^2_x}^2 -\frac{c}{2}\| Q_n \|_{L^4_x}^4\leq \|	Q_n \|_{L^2_x}^2,
	\end{equation*}
	while, by a direct computation and thanks to Proposition \ref{prop_friedrichs}, we deduce
	\begin{align*}
		 \Gamma L	\langle	P(Q_n),\,							\Delta	Q_n	\rangle_{L^2_x} 
		 &\lesssim
		 \| \nabla Q_n \|^2_{L^2_x} + \|Q_n\|_{L^6}^3\|\Delta Q_n\|_{L^2_x}
		 \lesssim
		  \| \nabla Q_n \|^2_{L^2_x} + \|Q_0\|_{H^1}^6e^{6Ct}+ 
		C_{\Gamma,L}\|\Delta Q_n\|_{L^2_x}^2,
	\end{align*}
	where $C$ is positive real constant, not dependent on $n$ and $C_{\Gamma,L}>0$ is a suitable small enough constant which will allow to absorb $\|\Delta Q_n\|_{L^2_x}^2$ by the left-hand 
	side of \eqref{neo}. Thus, summarizing the previous considerations, we get
	\begin{align*}
		\frac{\dd}{\dd t}\Big[	\|	u^n	\|_{L^2_x}^2 + \|	Q^n	\|_{L^2_x}^2 + \Gamma L\|	\nabla Q^n \|^2_{L^2_x}\Big] + 
		\nu		\|	\nabla u^n	\|_{L^2}^2 &+ 
		\Gamma L^2	\|	\Delta Q^n	\|_{L^2_x}^2\lesssim \\&
		\lesssim \|Q^n\|_{L^2_x}^2 + \|\nabla Q^n\|_{L^2_x}^2 + \|Q_0\|_{H^1}^6e^{6Ct},
	\end{align*}
	which yields
	\begin{equation}\label{Thm_fr_estimates}
	\begin{aligned}
		\|	(u^n,\,Q^n,\,\nabla Q^n)	\|_{L^\infty(0,T;L^2_x)} &+ \|(\nabla u^n,\,\Delta Q^n) \|_{L^2(0,T;L^2_x)} \lesssim \\&\lesssim
		(\| u_0 \|_{L^2_x} +\|Q_0\|_{L^2_x} +  \|Q_0\|_{H^1}^6 ) \exp\{\tilde{C}t\},
	\end{aligned}	
	\end{equation}
	for a suitable positive constant $\tilde{C}$, independent on $n$.
	
	\noindent
	Thanks to the previous control, we carry out to pass to the limit as $n$ goes to $+\infty$, and we claim to found a weak solution for system \eqref{main_system}.
	We fix at first a bounded domain $\Omega$ of $\RR^2$, with a smooth enough boundary.
	At first we claim that $(Q^n)_\NN$ is a Cauchy sequence in $C([0,T], L^2(\Omega))$, and the major part of the proof releases in the Arzelà-Ascoli Theorem. We have already 
	proven that $(Q^n)_\NN$ is bounded in such space, moreover, since $Q_n(t)$ belongs to $H^1(\Omega)$ which is compactly embedded in $L^2(\Omega)$, we get that
	$\{ Q_n(t)\,:\, n\in \NN\;\text{and}\; t\in [0,T]\}$ is a compact set of $L^2(\Omega)$. Moreover, observing that
	\begin{align*}
		\| \partial_t 	&Q^n \|_{L^2(\Omega)} 
		\leq 
		\|				u^n	\|_{L^4_x		}
		\|	\nabla		Q^n	\|_{L^4_x		}				+
		\|	\nabla 		u^n	\|_{L^2_x		}
		\|				Q^n	\|_{L^\infty_x}				+
		\|			P(	Q^n)\|_{L^2_x		}\\
		&\leq 
		\|				u^n	\|_{L^2_x		}^{\frac{1}{2}}
		\|	\nabla		u^n	\|_{L^2_x		}^{\frac{1}{2}}
		\|	\nabla	 	Q^n	\|_{L^2_x		}^{\frac{1}{2}}
		\|	\Delta		Q^n	\|_{L^2_x	}^{\frac{1}{2}}	+
		\|	\nabla	 	u^n	\|_{L^2_x		}
		\|				Q^n	\|_{H^2		}				+
		\|				Q^n	\|_{L^2_x		}				+
		\|				Q^n	\|_{L^4_x		}^2				+
		\|				Q^n	\|_{L^6_x		}^3.		
	\end{align*}
	Therefore, it turns out that $(\partial_t Q^n)_\NN$ is an uniformly bounded sequence in $L^1(0,T;L^2_x)$ which yields that $(Q^n)_\NN$ is uniformly 
	equicontinuous in $C([0,T], L^2_x)$, so that, applying the Arzelà-Ascoli Theorem, there exists $Q\in C([0,T], L^2_x)$ such that $Q^n$ 
	strongly converges to $Q$, up to a subsequence. Moreover, thanks to \eqref{Thm_fr_estimates}, we also obtain that $\nabla Q$ and $\Delta Q$ belong to
	$L^\infty(0,T;L^2_x)$ and $L^2(0,T;L^2_x)$ respectively, and we have:
	\begin{equation*}
		\nabla Q_n 	\rightharpoonup \nabla Q	\quad\quad w - L^2(0,T; L^2_x)\quad\quad\text{and}\quad\quad\quad
		\Delta Q^n	\rightharpoonup	\Delta Q	\quad\quad w - L^2(0,T; L^2_x),
	\end{equation*}
	up to a subsequence. Now, let us fix a bounded smooth domain  $\Omega$ of $\RR^2$. Then $\nabla Q^n(t)$ weakly converges to $\nabla Q(t)$ in $H^1(\Omega)$,
	for almost every $t\in (0,T)$, up to a subsequence, so that, from the compact embedding $H^1(\Omega)\hookrightarrow\hookrightarrow L^2(\Omega)$, we deduce 
	that $\nabla Q^n(t)$ strongly converges to $\nabla Q(t)$ in $L^2(\Omega)$, for almost every $t\in (0,T)$. Moreover $\|\nabla Q_n -\nabla Q \|_{L^2(\Omega)}$ belongs to 
	$L^\infty(0,T)$ and its norm is uniformly bounded in $n$. Hence applying the dominated convergence Theorem, we get  
	\begin{equation*}
		\lim_{n\rightarrow \infty}\int_0^T \|\nabla Q_n(t) -\nabla Q(t) \|_{L^2}^2\dd t = 
		\int_0^T\lim_{n\rightarrow \infty} \|\nabla Q_n(t) -\nabla Q(t) \|_{L^2}^2\dd t = 0,
	\end{equation*}
	namely $\nabla Q^n$ strongly converges to $\nabla Q$ in $L^2(0,T;L^2(\Omega))$. Since $\nabla Q^n$ is bounded in $L^2(0,T;L^6_x)$ (from the embedding 
	$H^1\hookrightarrow  L^6_x$ we get also that $\nabla Q^n$ weakly converges to $\nabla Q$ in $w-L^2(0,T;L^6_x)$, so that $\nabla Q^n$ strongly converges to 
	$\nabla Q$ in $L^2(0,T;L^4(\Omega))$, by interpolation. This range of convergences finally show that $\nabla Q\odot \nabla Q$ and $Q\Delta Q - \Delta Q$ are 
	the limits of $\nabla Q^n\odot \nabla Q^n$ and $Q^n\Delta Q^n - \Delta Q^n Q^n$, as $n$ goes to infinity, respectively in $L^1(0,T;L^4(\Omega))$ and 
	$L^1(0,T; L^{4/3}(\Omega))$.  the strongly convergence of $P(Q^n)$ to $P(Q)$ in $L^2(0,T;L^2(\Omega))$ is straightforward, while, with a similar strategy, we are
	able to prove the existence of $u\in L^\infty(0,T; L^2_x)$ with $\nabla u\in L^2(0,T;L^2_x)$ such that $u^n$ strongly converges to $u$ in $L^2(0,T; L^4(\Omega))$ 
	and $\nabla u^n$ weakly converges to $\nabla u$ in $L^2(0,T; L^{2}(\Omega))$ (everything up to a subsequence). Hence $u^n\cdot\nabla u^n$ and 
	$\Omega^n Q^n- Q^n\Omega^n$
	weakly converges in $L^1(0,T;L^{4/3}(\Omega) )$ to $u\cdot\nabla u$ and $\Omega Q - Q\Omega$ respectively. Finally $u^n\cdot \nabla Q^n$ strongly converges to 
	$u\cdot \nabla Q$ in $L^1(0,T; L^2(\Omega) )$.
	
	\noindent Now, $J_n \phi$ strongly converges to $\phi$ in $L^\infty(0,T; L^p_x)$, for any $\phi\in \DD(\,(0,T)\times \Omega\,)$ and for any $1\leq p<\infty$. Considering all the  
	previous convergences and since $(u^n,\,Q^n)$ is a weak solution of \eqref{system_friedrichs}, namely
	\begin{align*}
		-\int_0^T\int_{\RR^N} \trc\{Q^n \partial_t\Psi\} - \int_{\RR^N}&\trc\{ Q_0 \Psi(0,\cdot) \} + \int_0^T \int_{\RR^N}\trc\{ (u^n\cdot\nabla Q^n)\Psi\} +\\&+
		 \int_0^T\int_{\RR^N} \trc\{(\Omega^n Q^n - Q^n \Omega^n )\Psi\} =\Gamma \int_0^T \int_{\RR^N} \trc\{H(Q^n)\Psi\},
	\end{align*}
	for every $N\times N$-matrix $\Psi$ with coefficients in $\DD([0,T)\times \Omega)$ and
	\begin{align*}
		-\int_0^T\int_{\RR^N} u^n \cdot \partial_t \psi - \int_{\RR^N} u_0 \cdot \psi(0,\cdot) &+ \int_0^T\int_{\RR^N} (u^n\cdot \nabla u^n) \cdot \Pp J_n\psi 
		-\nu \int_0^T \int_{\RR^N} u^n\cdot \Delta \psi = \\ &=
		-L\int_0^T \int_{\RR^N} [ Q^n\Delta Q^n - \Delta Q^nQ^n - \nabla Q^n\odot \nabla Q^n ]\cdot \Pp J_n\nabla \psi,
	\end{align*}
	for any $N$-vector $\psi$ with coefficients in $\DD([0,T)\times \Omega)$, we pass through the limit as $n$ goes to $\infty$, obtaining
	\begin{align*}
		-\int_0^T\int_{\RR^N} \trc\{Q \partial_t\Psi\} - \int_{\RR^N}\trc\{ Q_0 \Psi(0,\cdot) \} &+ \int_0^T \int_{\RR^N}\trc\{ (u\cdot\nabla Q)\Psi\} +\\&+
		 \int_0^T\int_{\RR^N} \trc\{(\Omega Q - Q \Omega )\Psi\} =\Gamma \int_0^T \int_{\RR^N} \trc\{H(Q)\Psi\}
	\end{align*}
	and
	\begin{align*}
		-\int_0^T\int_{\RR^N} u \cdot \partial_t \psi - \int_{\RR^N} u_0 \cdot \psi(0,\cdot) &+ \int_0^T\int_{\RR^N} (u\cdot \nabla u) \cdot \Pp \psi 
		-\nu \int_0^T \int_{\RR^N} u\cdot \Delta \psi = \\ &=
		-L\int_0^T \int_{\RR^N} [ Q\Delta Q - \Delta Q\,Q - \nabla Q\odot \nabla Q ]\cdot \Pp \nabla \psi.
	\end{align*}
	From the arbitrariness of $T$ and $\Omega$, we finally achieve that $(u,\,Q)$ is a weak solution for \eqref{main_system} in the sense of definition \ref{def_weak_sol}.
\end{proof}

\section{The Difference Between Two Solutions}

\noindent
This section is devoted to an important remark which plays a major part in our uniqueness result. We deal with the difference between two weak solutions $(u_i,\,Q_i)$, $i=1,2$, of \eqref{main_system} in the sense of definition \ref{def_weak_sol}. Denoting by $(\delta u,\,\delta Q)$ the difference between the first and the second one, we claim that such element belongs to a lower regular space than the one the solutions belong to.
\begin{prop}\label{Proposition_regularity_difference_between_two_solutions}
	For any finite positive $T$, $\delta u$ and $\nabla \delta Q$ belong to $L^\infty(0,T; \Hh^{-1/2})$.
\end{prop}
\begin{remark}
	In virtue of Proposition \ref{Proposition_regularity_difference_between_two_solutions} and since $(\nabla \delta u,\, \Delta \delta Q)$ belongs to $L^2_t L^2_x$ then
	\begin{equation*}
		(\nabla \delta u,\, \Delta \delta Q) \in L^2(0,T; \Hh^{-1/2} ),
	\end{equation*}
	for any finite positive $T$, thanks to a classical real interpolation method:
	\begin{align*}
		\|		\nabla 	\delta u	\|_{\Hh^{-\frac{1}{2}}} 
		&\lesssim 
		\|		\nabla	\delta u	\|_{\Hh^{-\frac{3}{2}}	}^{\frac{1}{3}}
		\|		\nabla	\delta u	\|_{L^2_x				}^{\frac{2}{3}}
		\lesssim
		\|				\delta u	\|_{\Hh^{-\frac{1}{2}}	}+
		\|		\nabla	\delta u	\|_{L^2_x				},\\
		\|		\Delta 	\delta Q	\|_{\Hh^{-\frac{1}{2}}} 
		&\lesssim 
		\|		\Delta	\delta Q	\|_{\Hh^{-\frac{3}{2}}	}^{\frac{1}{3}}
		\|		\Delta	\delta Q	\|_{L^2_x				}^{\frac{2}{3}}
		\lesssim
		\|		\nabla	\delta Q	\|_{\Hh^{-\frac{1}{2}}	}+
		\|		\Delta	\delta Q	\|_{L^2_x				}.
	\end{align*}
\end{remark}
\begin{proof}[Proof of Proposition \ref{Proposition_regularity_difference_between_two_solutions}]
	Fixing $T>0$ we are going to prove that $\delta u$ belongs to $L^\infty(0,T; \Hh^{-1/2})$ and $\delta Q$ belongs to $L^\infty(0,T; \Hh^{1/2})$.
	We denote by $f_1$ and $f_2$
	\begin{align*}
		f_1 :=	- u_1 \cdot \nabla  Q_1 + u_2 \cdot  \nabla Q_2 & +  \Omega_1  Q_1 - \Omega_2  Q_2 -  Q_1 \Omega_1 +   Q_2  \Omega_2 \,+ \\&+
				\Gamma 
				\Big\{ \;
					\frac{b}{3}
					\Big( 
						 Q_1^2 - Q_2^2 - \trc\{  Q_1^2 - Q_2^2  \}\frac{\Id}{3} 
					\Big) - 
					c\, \trc \{  Q_1^2 	\} 	Q_1 +
					c\, \trc \{ Q_2^2 	\} 	Q_2 \;
				\Big\} ,
	\end{align*}
	\begin{align*}
		f_2 := \Pp\big[\, -  \Div\{ u_1 \otimes  u_1 &-  u_2 \otimes u_2 \} + L\Div\,\{\; Q_1 \Delta  Q_1 -  Q_2 \Delta  Q_2 -\\
				&- \Delta  Q_1  Q_1 + \Delta  Q_2  Q_2 - \nabla  Q_1\odot \nabla Q_1 + \nabla Q_2\odot \nabla Q_2\;\}\big],
	\end{align*}
	respectively. Then $\delta Q$ and $\delta u$ are weak solutions of the following Cauchy Problems:
	\begin{equation*}
		\partial_t \delta Q- \Gamma L\Delta \delta Q+\Gamma a \,\delta Q = f_1	\quad \text{and} \quad \partial_t \delta u- \nu \Delta \delta u = f_2
		\quad\quad\quad \text{in}\quad [0,T)\times\RR^2,
	\end{equation*}
	with null initial data. Then, by the classical Theory of Evolutionary Parabolic Equation, it is sufficient to prove that $f_1$ and $f_2$ belong to $L^2(0,T; \Hh^{-1/2})$ and 
	$L^2(0,T;\Hh^{-3/2})$ respectively in order to obtain
	\begin{equation*}
		\|(\nabla \delta u,\, \Delta \delta Q) 	\|_{L^\infty(0,T; \Hh^{-\frac{1}{2}})}\lesssim \| f_1 \|_{L^2(0,T; \Hh^{-\frac{1}{2}})} + \| f_2 \|_{L^2(0,T; \Hh^{-\frac{3}{2}})}, 
	\end{equation*}
	and conclude the proof. We start by $f_1$ and Theorem \ref{product_Hs_Ht} plays a major part. For any $i=1,\,2$, we get
	\begin{equation*}
	\begin{alignedat}{5}
		\| 			u_i 	\cdot	\nabla 	Q_i		\|_{	\Hh^{-\frac{1}{2}}	} 
		&\lesssim 
		\|			u_i								\|_{	\Hh^{ \frac{1}{2}}	}
		\|							\nabla	Q_i		\|_{	L^2_x				}
		&&\lesssim	
		\|			u_i								\|_{	L^2_x				}^\frac{1}{2}
		\|	\nabla 	u_i								\|_{	L^2_x				}^\frac{1}{2}
		\|							\nabla	Q_i		\|_{	L^2_x				} 
		&&&&\in L^4(0, T),\\
		\|		\Omega_i	\,				Q_i		\|_{	\Hh^{-\frac{1}{2}}	}
		&\lesssim
		\|	\nabla u_i								\|_{	L^2_x				}
		\|									Q_i		\|_{	\Hh^{ \frac{1}{2}}	}
		&&\lesssim
		\|	\nabla u_i								\|_{	L^2_x				}
		\|									Q_i		\|_{	L^2_x				}^\frac{1}{2}
		\|							\nabla	Q_i		\|_{	L^2_x				}^\frac{1}{2} 
		&&&&\in L^2(0, T),\\
		\|									Q_i^2	\|_{	\Hh^{-\frac{1}{2}}	} 
		&\lesssim
		\|									Q_i		\|_{	\Hh^{ \frac{1}{2}}	} 
		\|									Q_i		\|_{	L^2_x				}
		&&\lesssim
		\|									Q_i		\|_{	L^2_x				}
		\|							\nabla	Q_i		\|_{	L^2_x				}^\frac{1}{2}
		\|									Q_i		\|_{	L^2_x				}^\frac{1}{2}
		&&&&\in L^\infty(0,T),\\
		\|					\trc\{ Q_i^2\}	Q_i		\|_{	\Hh^{-\frac{1}{2}}	}
		&\lesssim
		\|									Q_i^2	\|_{	L^2_x				}
		\|									Q_i		\|_{	\Hh^{ \frac{1}{2}}	}
		&&\lesssim
		\|									Q_i		\|_{	L^4_x				}^2
		\|							\nabla	Q_i		\|_{	L^2_x				}^\frac{1}{2}
		\|									Q_i		\|_{	L^2_x				}^\frac{1}{2}
		\lesssim
		\|							\nabla	Q_i		\|_{	L^2_x				}^\frac{3}{2}
		\|									Q_i		\|_{	L^2_x				}^\frac{3}{2}
		&&&&\in L^\infty(0,T).
	\end{alignedat}
	\end{equation*}
	Then, summarizing the previous estimates, we finally deduce that $f_1$ belongs to $L^2(0, T ; \Hh^{-1/2})$. Now, let us handle the terms of $f_2$:
	\begin{equation*}
	\begin{alignedat}{5}
		\| 	\Div\{	u_i 	\otimes		 	u_i	\}	\|_{	\Hh^{-\frac{3}{2}}	}
		&\lesssim
		\| 			u_i 	\otimes		 	u_i		\|_{	\Hh^{-\frac{1}{2}}	}
		\lesssim
		\|			u_i								\|_{	\Hh^{ \frac{1}{2}}	}
		\|			u_i								\|_{	L^2_x				}
		&&\lesssim
		\|			u_i								\|_{	L^2_x				}^\frac{1}{2}
		\|	\nabla	u_i								\|_{	L^2_x				}^\frac{1}{2}
		\|			u_i								\|_{	L^2_x				}
		&&&\in L^4(0,T),\\
		\|	\Div\{	Q_i	\Delta Q_i	\}				\|_{	\Hh^{-\frac{3}{2}}	}
		&\lesssim
		\|					Q_i		\Delta Q_i		\|_{	\Hh^{-\frac{1}{2}}	}
		\lesssim
		\|					Q_i						\|_{	\Hh^{ \frac{1}{2}}	}
		\|							\Delta Q_i		\|_{	L^2_x				}
		&&\lesssim
		\|					Q_i						\|_{	L^2_x				}^\frac{1}{2}
		\|			\nabla	Q_i						\|_{	L^2_x				}^\frac{1}{2}
		\|							\Delta Q_i		\|_{	L^2_x				}
		&&&\in L^\infty(0,T)
	\end{alignedat}
	\end{equation*}
	and moreover
	\begin{align*}
		\|	\Div\{ \nabla Q_i \odot \nabla Q_i \}	\|_{	\Hh^{-\frac{3}{2}}	}
		\lesssim
		\|		\nabla	Q_i	\odot	\nabla	Q_i 	\|_{	\Hh^{-\frac{1}{2}}	}
		&\lesssim
		\|		\nabla	Q_i							\|_{	\Hh^{ \frac{1}{2}}	}
		\|		\nabla	Q_i							\|_{	L^2_x				}\\
		&\lesssim
		\|		\nabla	Q_i							\|_{	L^2_x				}^\frac{1}{2}
		\|		\Delta	Q_i							\|_{	L^2_x				}^\frac{1}{2}
		\|		\nabla	Q_i							\|_{	L^2_x				}
		\in L^4(0,T),	
	\end{align*}
	which finally implies that $f_2$ belongs to $L^2(0,T; \Hh^{-\frac{3}{2}})$. This concludes the proof of Proposition \ref{Proposition_regularity_difference_between_two_solutions}.
\end{proof}

\section{Uniqueness}
\noindent
In this section we present our first original result. We are going to prove Theorem \ref{Thm_Uniqueness}, namely the uniqueness of the weak solutions, given by Theorem \ref{Thm_Existence}. We implement the uniqueness result of Paicu and Zarnescu in \cite{MR2864407}, concerning the weak-strong uniqueness. Indeed the authors suppose that at least one of the solutions is a classical solution. The leading cause of such restriction relies on the choice to control the difference between two solutions in an $L^2_x$-setting. However, this requires to estimate the $L^\infty_x$-norm of one of the solutions, $\|(u,\,\nabla Q)\|_{L^\infty_x}$, for instance by a Sobolev embedding, therefore the necessity to put $(u(t),\,\nabla Q(t))$ in some $\Hh^s$ with $s>1$, for any real $t$.

\vspace{0.1cm}
\noindent
In this article we overcome this drawback, performing the weak-weak uniqueness, thanks to an alternative approach which is inspired by \cite{MR1813331} and \cite{MR2309504}. The main idea is to evaluate the difference between two weak solutions in a functional space which is less regular than $L^2_x$. Considering two weak solutions $(u_1,\,\nabla Q_1)$ and $(u_2,\,\nabla Q_2)$, we define $(\delta u, \delta Q)$ as the difference between the first one and the second one. It is straightforward that such difference is a weak solution for the following system:
\begin{equation}\label{delta_main_system}
	\tag{$\delta P$}
	\begin{cases}
	\;	\partial_t \delta Q + \delta u\cdot \nabla  Q_1 + u_2 \cdot  \nabla \delta Q 
		- \delta S(\nabla u,\,Q) -\Gamma L \Delta \delta Q = \Gamma \delta P( Q)															&\RR_+	\times	\RR^2,\\
	\;	\partial_t \delta u + \delta u \cdot \nabla u_1 +  u_2 \cdot \nabla \delta u -\nu \Delta \delta u +\nabla \delta \Pi= 
		L\Div\,\big\{\;\delta Q \Delta  Q_1 +  Q_2 \Delta \delta Q - \\ 
		\quad\quad\quad\quad\quad\quad\quad\quad\quad\quad\quad\quad
		- \Delta \delta Q  Q_1 - \Delta  Q_2 \delta Q - 
		\nabla \delta Q\odot \nabla Q_1 - \nabla Q_2\odot \nabla \delta Q\;\big\}															&\RR_+	\times	\RR^2,\\
	\;	\Div\,\delta u = 0																													&\RR_+	\times	\RR^2,\\
	\;	(\delta u,\,\delta Q)_{t=0} = (0,\, 0)																								&\quad\quad\;\;\RR^2,
	\end{cases}
\end{equation}
where we have also defined
\begin{equation*}
	 \delta \Omega 			 := \Omega_1 - \Omega_2	,	\quad
	 \delta \Pi 			 := \Pi_1 	- \Pi_2		,	\quad
	 \delta	P(Q)			 := P(Q_1)	- P(Q_2)	.
\end{equation*}
and moreover
\begin{equation*}
	 \delta S(Q,\,\nabla u)	 := \Omega_1  Q_1 - Q_1 \Omega_1 + \Omega_2  Q_2 - Q_2 \Omega_2
	 						  = \delta Q \delta \Omega - \delta \Omega \delta Q  + 	\delta \Omega Q_2 - Q_2 \delta \Omega  +  \Omega_2 \delta Q-\delta Q \Omega_2. 
\end{equation*}
Recalling the previous subsection, we take the $\Hh^{-1/2}$-inner product between the first equation of \eqref{delta_main_system} and  $-L\Delta \delta Q$ and moreover we consider the scalar product in $\Hh^{-1/2}$ between the second one and $\delta u$:
\begin{equation}\label{uniqueness_energy_equality}
\begin{split}
	\frac{\dd}{\dd t}
	\Big[&
		\frac{1}{2}	\|			\delta u	\|_{\dot{H}^{-\frac{1}{2}}}^2 + 
		L			\|	\nabla	\delta Q	\|_{\dot{H}^{-\frac{1}{2}}}^2
	\Big]+
	\nu				\|	\nabla 	\delta u	\|_{\dot{H}^{-\frac{1}{2}}}^2 +
	\Gamma L^2		\|	\Delta	\delta Q	\|_{\dot{H}^{-\frac{1}{2}}}^2
	= \\=\;&
	L\Gamma \langle		\delta	P(Q) 				 ,	\Delta	\delta	Q	\rangle_{\dot{H}^{-\frac{1}{2}}} -	
	L		\langle		\delta	u\cdot \nabla Q_1	 ,	\Delta	\delta	Q	\rangle_{\dot{H}^{-\frac{1}{2}}} +
	L		\langle		u_2	\cdot \nabla \delta Q	 ,	\Delta	\delta	Q	\rangle_{\dot{H}^{-\frac{1}{2}}} +	\\&+
	L		\langle		\delta S(Q, \nabla u)	 	 ,	\Delta	\delta	Q	\rangle_{\dot{H}^{-\frac{1}{2}}} -	
	L		\langle		\nabla \delta Q \odot
						\nabla Q_1					 ,	\nabla	\delta	u	\rangle_{\dot{H}^{-\frac{1}{2}}} +
	L		\langle		\nabla Q_2 		\odot 
						\nabla \delta Q				 ,	\nabla	\delta	u	\rangle_{\dot{H}^{-\frac{1}{2}}} -	\\&
	\quad\quad\quad\quad\quad-
			\langle		\delta u\cdot\nabla u_1		 ,			\delta	u	\rangle_{\dot{H}^{-\frac{1}{2}}} -	 
			\langle		u_2\cdot\nabla \delta u		 ,			\delta	u	\rangle_{\dot{H}^{-\frac{1}{2}}} +
	L		\langle		\delta Q \Delta \delta Q -  
						\Delta \delta Q \delta Q	 ,	\nabla	\delta	u	\rangle_{\dot{H}^{-\frac{1}{2}}} +	\\&
	\quad\quad\quad\quad\quad\quad\quad\quad\quad\quad\,+
	L		\langle		Q_2 \Delta \delta Q	-		 
						\Delta \delta Q   Q_2 		 ,	\nabla	\delta	u	\rangle_{\dot{H}^{-\frac{1}{2}}}+
	L		\langle		\delta Q \Delta   Q_2-		
						\Delta Q_2 \delta Q 		 ,	\nabla	\delta	u	\rangle_{\dot{H}^{-\frac{1}{2}}}. 
\end{split}
\end{equation}
Denoting by 
		$\Phi(t)= 
		\frac{1}{2}	\|			\delta u(t)	\|_{\dot{H}^{-\frac{1}{2}}}^2 +
		L			\|	\nabla	\delta Q(t)	\|_{\dot{H}^{-\frac{1}{2}}}^2 $
we claim that
\begin{equation*}
	\frac{\dd}{\dd t}\Phi(t)\leq \chi(t)\Phi(t),\quad\quad\text{for almost every}\quad t\in \RR_+,
\end{equation*}
where $\chi\geq 0 $ belongs to $L^1_{loc}(\RR_+)$.
Hence, uniqueness holds thanks to the Gronwall Lemma and since $\Phi(0)$ is null. Thus, we need to analyze every terms of the right-hand side of \eqref{uniqueness_energy_equality}. From here on $C_{\Gamma, L}$ and $C_\nu$ are suitable positive constants which will be determined in the end of the proof.
\subsection*{Simpler Terms} First, we begin evaluating every term which is handleable by Theorem \ref{product_Hs_Ht}. 

\vspace{0.2cm}
\noindent \underline{Estimate of $	\Gamma L \langle\delta	P(Q) ,
									\Delta\delta Q\rangle_{\dot{H}^{-\frac{1}{2}}}	$}

\vspace{0.1cm}
\noindent From the definition of $\delta P(Q)$, and since $\trc\{\Delta Q\} $ is null, we need to control 
\begin{equation*}
\begin{split}
	  \Gamma L 			\langle 		\delta	P(Q) 								, \Delta\delta Q   &\rangle_{\Hh^{-\frac{1}{2}} }
	=
	- \Gamma L a		\|					\nabla \delta Q													 \|_{\Hh^{-\frac{1}{2}} }^2   
	+ \Gamma L b		\langle			\delta Q\,Q_1 + Q_2  \delta Q				, \Delta \delta Q 	\rangle_{\Hh^{-\frac{1}{2}} } \\&
	- \Gamma L c		\langle 		\delta Q\trc\{Q_1^2\}						, \Delta \delta Q	\rangle_{\Hh^{-\frac{1}{2}} } 
	- \Gamma L c		\langle			\trc\{\,\delta Q\,Q_1 + Q_2\delta Q\,\}Q_1	, \Delta \delta Q	\rangle_{\Hh^{-\frac{1}{2}} }. 
\end{split}
\end{equation*}
We overcome the second term in the right hand-side of the equality as follows:
\begin{equation*}
\begin{split}
	\Gamma L b
	\langle	
		\delta Q\,Q_1 + Q_2  \delta Q	, \Delta \delta Q 	
	\rangle_{\Hh^{-\frac{1}{2}} }
	&\lesssim
	\|			\delta 	Q			\|_{	\Hh^{ \frac{1}{2} } }
	\|					(Q_1,\,Q_2)	\|_{	L^2_x				}
	\|	\Delta	\delta	Q			\|_{	\Hh^{-\frac{1}{2} } }\\
	&\lesssim
	\|	\nabla	\delta 	Q			\|_{	\Hh^{-\frac{1}{2} } }^2
	\|					(Q_1,\,Q_2)	\|_{	L^2_x				}^2+
	C_{\Gamma, L}
	\|	\Delta	\delta 	Q			\|_{	\Hh^{-\frac{1}{2} }	}^2.
\end{split}
\end{equation*}
Furthermore, we observe that
\begin{equation*}
\begin{split}
	\Gamma& L c			
	\langle 
		\delta Q\trc\{Q_1^2\}, \Delta \delta Q	
	\rangle_{\Hh^{-\frac{1}{2}} }
	\lesssim
	\|			\delta	Q		\|_{	\Hh^{ \frac{1}{2} } }
	\|					Q_1^2	\|_{	L^2_x}
	\|	\Delta	\delta	Q		\|_{	\Hh^{-\frac{1}{2} } }
	\lesssim
	\|	\nabla	\delta	Q		\|_{	\Hh^{-\frac{1}{2} } }
	\|					Q_1		\|_{	L^4_x				}^2
	\|	\Delta	\delta	Q		\|_{	\Hh^{-\frac{1}{2} } }\\
	&\lesssim
	\|	\nabla	\delta	Q		\|_{	\Hh^{-\frac{1}{2} } }
	\|					Q_1		\|_{	L^2_x				}
	\|	\nabla			Q_1		\|_{	L^2_x				}
	\|	\Delta	\delta	Q		\|_{	\Hh^{-\frac{1}{2} } }
	\lesssim
	\|	\nabla	\delta	Q		\|_{	\Hh^{-\frac{1}{2} } }^2
	\|					Q_1		\|_{	L^2_x				}^2
	\|	\nabla			Q_1		\|_{	L^2_x				}^2+
	C_{\Gamma, L}
	\|	\Delta	\delta 	Q		\|_{	\Hh^{-\frac{1}{2} }	}^2	
\end{split}
\end{equation*}
and moreover
\begin{equation*}
\begin{alignedat}{6}
	\Gamma &L c			
	\langle	
		\trc\{\,\delta Q\,Q_1 + Q_2\delta Q\,\}Q_1	, \Delta \delta Q	
	\rangle_{\Hh^{-\frac{1}{2}} }
	&&\lesssim
	\|			\delta 	Q			\|_{	\Hh^{ \frac{1}{2} } }
	\Big(
		\|				|Q_1|^2		\|_{	L^2_x				} +
		\|				|Q_2||Q_1|	\|_{	L^2_x				}
	\Big)
	\|	\Delta	\delta	Q			\|_{	\Hh^{-\frac{1}{2} } }\\
	&\lesssim
	\|	\nabla	\delta 	Q			\|_{	\Hh^{-\frac{1}{2} } }
	\|					(Q_1,\,Q_2)	\|_{	L^4_x				}^2
	\|	\Delta	\delta	Q			\|_{	\Hh^{-\frac{1}{2} } }
	&&\lesssim
	\|	\nabla	\delta 	Q			\|_{	\Hh^{-\frac{1}{2} } }
	\|					(Q_1,\,Q_2)	\|_{	L^2_x				}
	\|	\nabla			(Q_1,\,Q_2)	\|_{	L^2_x				}
	\|	\Delta	\delta	Q			\|_{	\Hh^{-\frac{1}{2} } }\\
	&\quad\quad\quad\quad\quad\quad\quad\quad\quad\quad\quad\quad
	\lesssim
	\|	\nabla	\delta 	Q			\|_{	\Hh^{-\frac{1}{2} } }^2	&&
	\|					(Q_1,\,Q_2)	\|_{	L^2_x				}^2
	\|	\nabla			(Q_1,\,Q_2)	\|_{	L^2_x				}^2+
	C_{\Gamma,L}
	\|	\Delta	\delta 	Q		\|_{	\Hh^{-\frac{1}{2} }	}^2.
\end{alignedat}
\end{equation*}
Finally, summarizing the previous inequality, we get
\begin{equation*}
	\Gamma L 	\langle 		\delta	P(Q) 	, \Delta\delta Q   \rangle_{\Hh^{-\frac{1}{2}} }
	\lesssim
	\|	\nabla	\delta 	Q			\|_{	\Hh^{-\frac{1}{2} } }^2
	\|					(Q_1,\,Q_2)	\|_{	L^2_x				}^2
	\|	\nabla			(Q_1,\,Q_2)	\|_{	L^2_x				}^2+
	C_{\Gamma,L}
	\|	\Delta	\delta 	Q		\|_{	\Hh^{-\frac{1}{2} }	}^2
\end{equation*}

\vspace{0.2cm}
\noindent \underline{Estimate of $L		\langle		\delta	u\cdot \nabla Q_1	,	\Delta	\delta	Q	\rangle_{\dot{H}^{-\frac{1}{2}}}$}
\begin{equation*}
\begin{split}
	L		\langle		&\delta	u\cdot \nabla Q_1	,	\Delta	\delta	Q	\rangle_{\dot{H}^{-\frac{1}{2}}}
	\lesssim
	\|			\delta	u		\|_{\Hh^{-\frac{1}{4}} }
	\|	\nabla 			Q_1		\|_{\Hh^{ \frac{3}{4}} }
	\|	\Delta	\delta	Q		\|_{\Hh^{-\frac{1}{2}} }
	\lesssim
	\|			\delta	u		\|_{\Hh^{-\frac{1}{2}} }^{\frac{3}{4}}
	\|	\nabla 	\delta	u		\|_{\Hh^{-\frac{1}{2}} }^{\frac{1}{4}}
	\|	\nabla 			Q_1		\|_{	L^2_x		   }^{\frac{1}{4}}{\scriptstyle \times} \\&{\scriptstyle \times}
	\|	\Delta 			Q_1		\|_{	L^2_x		   }^{\frac{3}{4}}
	\|	\Delta 	\delta	Q		\|_{\Hh^{-\frac{1}{2}} }
	\lesssim 
	C_{\Gamma,L}
	\|	\Delta 	\delta	Q		\|_{\Hh^{-\frac{1}{2}} }^2 +
	C_{\nu}
	\|	\nabla 	\delta u		\|_{\Hh^{-\frac{1}{2}}}^2 +
	\|	\nabla 			Q_1		\|_{	L^2_x		  }^{\frac{2}{3}}
	\|	\Delta 			Q_1		\|_{	L^2_x		  }^2
	\|			\delta	u		\|_{\Hh^{-\frac{1}{2}}}^2.
\end{split}
\end{equation*}

\vspace{0.2cm}
\noindent \underline{Estimate of $L		\langle		u_2	\cdot \nabla \delta Q	,	\Delta	\delta	Q	\rangle_{\dot{H}^{-\frac{1}{2}}}$}
\begin{equation*}
\begin{split}
	L		\langle		u_2	\cdot \nabla \delta Q	,	\Delta	\delta	Q	\rangle_{\dot{H}^{-\frac{1}{2}}}
	&\lesssim
	\|u_2\|_{\Hh^{\frac{3}{4}}}\|\nabla \delta Q\|_{\Hh^{-\frac{1}{4}}}\|\Delta \delta Q\|_{\Hh^{-\frac{1}{2}}}
	\lesssim
	\|u_2\|_{L^2_x}^\frac{1}{4}\|\nabla u_2\|_{L^2_x}^\frac{3}{4}
	\|\nabla \delta Q\|_{\Hh^{-\frac{1}{2}}}^\frac{3}{4}\|\Delta \delta Q\|_{\Hh^{-\frac{1}{2}}}^\frac{5}{4}\\
	&\lesssim
	C_{\Gamma,L}\|\Delta \delta Q\|_{\Hh^{-\frac{1}{2}}}^2 +
	\|u_2\|_{L^2_x}^{\frac{2}{3}}\|\nabla u_2\|_{L^2_x}^{2}\|\nabla \delta Q\|_{\Hh^{-\frac{1}{2}}}^{2}.
\end{split}
\end{equation*}

\vspace{0.2cm}
\noindent \underline{Estimate of $L\langle 	\delta Q \delta \Omega - \delta \Omega \delta Q	,\Delta\delta Q	\rangle_{\dot{H}^{-\frac{1}{2}}}$}
\begin{equation*}
\begin{aligned}
	L\langle 	
		\delta Q \delta \Omega - \delta \Omega \delta Q	,\Delta\delta Q
	\rangle_{\dot{H}^{-\frac{1}{2}}}
	&\lesssim
	\|			\delta	Q		\|_{	\Hh^{ \frac{1}{2} } }
	\|			\delta	\Omega	\|_{	L^2_x				}
	\|	\Delta	\delta	Q		\|_{	\Hh^{-\frac{1}{2} } }\\
	&\lesssim
	\|	\nabla	(u_1,\,u_2)		\|_{	L^2_x				}^2
	\|	\nabla	\delta	Q		\|_{	\Hh^{-\frac{1}{2} } }^2+
	C_{\Gamma,L}
	\|	\Delta	\delta 	Q		\|_{	\Hh^{-\frac{1}{2} }	}^2.
\end{aligned}
\end{equation*}

\vspace{0.2cm}
\noindent \underline{Estimate of $L\langle		\Omega_2 \delta Q-\delta Q \Omega_2      ,\Delta\delta Q	\rangle_{\dot{H}^{-\frac{1}{2}}}$}
\begin{equation*}
	L\langle		
		\Omega_2 \delta Q-\delta Q \Omega_2      ,\Delta\delta Q	
	\rangle_{\dot{H}^{-\frac{1}{2}}}
	\lesssim
	\|				\Omega_2	\|_{	L^2_x				}
	\|			\delta	Q		\|_{	\Hh^{ \frac{1}{2} } }
	\|	\Delta	\delta	Q		\|_{	\Hh^{-\frac{1}{2} } }
	\lesssim
	\|	\nabla			u_2		\|_{	L^2_x				}^2
	\|	\nabla	\delta	Q		\|_{	\Hh^{-\frac{1}{2} } }^2+
	C_{\Gamma,L}
	\|	\Delta	\delta 	Q		\|_{	\Hh^{-\frac{1}{2} }	}^2.
\end{equation*}

\vspace{0.2cm}
\noindent \underline{Estimate of $L 	\langle	\nabla \delta Q \odot \nabla Q_1,	\nabla	\delta	u	\rangle_{\dot{H}^{-\frac{1}{2}}}$}
\begin{equation*}
\begin{split}
	L 	\langle	
			\nabla \delta Q \odot \nabla Q_1,	\nabla	\delta	u	
		\rangle_{\dot{H}^{-\frac{1}{2}}}
	&\lesssim
	\|	\nabla \delta	Q	\|_{\Hh^{-\frac{1}{4}}}
	\|	\nabla 			Q_1	\|_{\Hh^{ \frac{3}{4}}}
	\|	\nabla \delta	u	\|_{\Hh^{-\frac{1}{2}}}\\
	&\lesssim
	\|	\nabla \delta 	Q	\|_{\Hh^{-\frac{1}{2}}}^\frac{3}{4}
	\|	\Delta \delta 	Q	\|_{\Hh^{-\frac{1}{2}}}^\frac{1}{4}
	\|	\nabla 			Q_1	\|_{	L^2_x		  }^\frac{1}{4}
	\|	\Delta 			Q_1	\|_{	L^2_x		  }^\frac{3}{4}
	\|	\nabla \delta 	u	\|_{\Hh^{-\frac{1}{2}}}\\
	&\lesssim
	C_{\Gamma,L}	\|\Delta \delta Q\|_{\Hh^{-\frac{1}{2}}}^2 +
	C_{\nu}	\|\nabla \delta u\|_{\Hh^{-\frac{1}{2}}}^2 +
	\|	\nabla 			Q_1	\|_{	L^2_x		  }^\frac{2}{3}
	\|	\Delta 			Q_1	\|_{	L^2_x		  }^2
	\|	\nabla \delta 	Q	\|_{\Hh^{-\frac{1}{2}}}^2	
\end{split}
\end{equation*}

\vspace{0.2cm}
\noindent \underline{Estimate of $L 	\langle	\nabla Q_2\odot\nabla\delta Q, \,	\nabla\delta u	\rangle_{\dot{H}^{-\frac{1}{2}}}$}
\begin{equation*}
\begin{split}
	L 	\langle	
			\nabla Q_2\odot\nabla\delta Q,\,	\nabla\delta u	
		\rangle_{\dot{H}^{-\frac{1}{2}}}
	&\lesssim
	\|	\nabla \delta	Q	\|_{\Hh^{-\frac{1}{4}}}
	\|	\nabla 			Q_2	\|_{\Hh^{ \frac{3}{4}}}
	\|	\nabla \delta	u	\|_{\Hh^{-\frac{1}{2}}}\\
	&\lesssim
	\|	\nabla \delta 	Q	\|_{\Hh^{-\frac{1}{2}}}^\frac{3}{4}
	\|	\Delta \delta 	Q	\|_{\Hh^{-\frac{1}{2}}}^\frac{1}{4}
	\|	\nabla 			Q_2	\|_{	L^2_x		  }^\frac{1}{4}
	\|	\Delta 			Q_2	\|_{	L^2_x		  }^\frac{3}{4}
	\|	\nabla \delta 	u	\|_{\Hh^{-\frac{1}{2}}}\\
	&\lesssim
	C_{\Gamma,L}	\|\Delta \delta Q\|_{\Hh^{-\frac{1}{2}}}^2 +
	C_{\nu}	\|\nabla \delta u\|_{\Hh^{-\frac{1}{2}}}^2 +
	\|	\nabla 			Q_2	\|_{	L^2_x		  }^\frac{2}{3}
	\|	\Delta 			Q_2	\|_{	L^2_x		  }^2
	\|	\nabla \delta 	Q	\|_{\Hh^{-\frac{1}{2}}}^2	
\end{split}
\end{equation*}

\vspace{0.2cm}
\noindent \underline{Estimate of $\langle	\delta u\cdot\nabla u_1		,	\delta	u	\rangle_{\dot{H}^{-\frac{1}{2}}}$}
\begin{equation*}
\begin{split}
	\langle		
		\delta u\cdot\nabla u_1		,	\delta	u
	\rangle_{\dot{H}^{-\frac{1}{2}}}
	\lesssim
	\|			\delta 		u	\|_{\Hh^{ \frac{1}{2}}}
	\|	\nabla 				u_1	\|_{	L^2_x		  }
	\|			\delta 		u	\|_{\Hh^{-\frac{1}{2}}}
	\lesssim
	C_{\nu}
	\|	\nabla 	\delta 		u	\|_{\Hh^{-\frac{1}{2}}}^2 +
	\|	\nabla 				u_1	\|_{	L^2_x		  }^2
	\|			\delta 		u	\|_{\Hh^{-\frac{1}{2}}}^2
\end{split}
\end{equation*}

\vspace{0.2cm}
\noindent \underline{Estimate of $\langle		u_2	\cdot	\nabla \delta u	,	\delta	u	\rangle_{\dot{H}^{-\frac{1}{2}}}$}
\begin{equation*}
\begin{split}
	\langle		
		u_2	\cdot	\nabla \delta u	,	\delta	u	
	\rangle_{\dot{H}^{-\frac{1}{2}}}
	\lesssim
	\|			\delta 		u	\|_{\Hh^{ \frac{1}{2}}}
	\|	\nabla 				u_2	\|_{	L^2			  }
	\|			\delta 		u	\|_{\Hh^{-\frac{1}{2}}}
	\lesssim
	C_{\nu}
	\|	\nabla 	\delta 		u	\|_{\Hh^{-\frac{1}{2}}}^2 +
	\|	\nabla 				u_2	\|_{	L^2			  }^2
	\|			\delta 		u	\|_{\Hh^{-\frac{1}{2}}}^2
\end{split}
\end{equation*}

\vspace{0.2cm}
\noindent \underline{Estimate of $L \langle \delta Q \Delta \delta Q - \Delta \delta Q \delta Q		 ,	\nabla	\delta	u	\rangle_{\dot{H}^{-\frac{1}{2}}}$}

\vspace{0.1 cm}
\begin{equation*}
\begin{split}
	 L\langle \delta Q \Delta \delta Q-\Delta \delta Q \delta Q 		 ,	\nabla	\delta	u	\rangle_{\dot{H}^{-\frac{1}{2}}}
	 &\lesssim
	 \|	\Delta		(Q_1,Q_2)	\|_{	L^2_x				}
	 \|	\nabla	\delta	Q		\|_{	\Hh^{-\frac{1}{2} } }
	 \|	\nabla	\delta	u		\|_{	\Hh^{-\frac{1}{2} } }\\
	 &\lesssim
	 \|	\Delta		(Q_1,Q_2)	\|_{	L^2_x				}^2
	 \|	\nabla	\delta	Q		\|_{	\Hh^{-\frac{1}{2} } }^2+
	C_{\nu}
	 \|	\nabla	\delta	u		\|_{	\Hh^{-\frac{1}{2} } }^2.
\end{split}
\end{equation*}

\vspace{0.2cm}
\noindent \underline{Estimate of $L \langle \delta Q \Delta Q_2 - \Delta Q_2 \delta Q		 ,	\nabla	\delta	u	\rangle_{\dot{H}^{-\frac{1}{2}}}$}

\vspace{0.1 cm}
\begin{equation*}
\begin{split}
	 L\langle \delta Q \Delta Q_2-\Delta Q_2 \delta Q 		 ,	\nabla	\delta	u	\rangle_{\dot{H}^{-\frac{1}{2}}}
	 &\lesssim
	 \|	\Delta			Q_2		\|_{	L^2_x				}
	 \|	\nabla	\delta	Q		\|_{	\Hh^{-\frac{1}{2} } }
	 \|	\nabla	\delta	u		\|_{	\Hh^{-\frac{1}{2} } }\\
	 &\lesssim
	 \|	\Delta			Q_2		\|_{	L^2_x				}^2
	 \|	\nabla	\delta	Q		\|_{	\Hh^{-\frac{1}{2} } }^2+
	C_{\nu}
	 \|	\nabla	\delta	u		\|_{	\Hh^{-\frac{1}{2} } }^2.
\end{split}
\end{equation*}

\subsection*{The Residual Terms}

\vspace{0.1cm}
\noindent Now we deal with the terms in the right-hand side of \eqref{uniqueness_energy_equality} which we have not evaluated yet, namely
\begin{equation}\label{uniqueness_challenging_term}
	L\langle		\delta \Omega Q_2 - Q_2 \delta \Omega      ,\Delta\delta Q	\rangle_{\dot{H}^{-\frac{1}{2}}} +
	L		\langle		Q_2 \Delta \delta Q	-		 
						\Delta \delta Q  Q_2 		 ,	\nabla	\delta	u	\rangle_{\dot{H}^{-\frac{1}{2}}}.
\end{equation}
Here, the difference between the two solutions appears with the higher derivative-order, more precisely the inner product is driven by $\nabla \delta u$ ( i.e. $\delta \Omega$) and $\Delta \delta Q$. This clearly generates a drawback if we want to analyze every remaining term, proceeding as the previous estimates. Let us remark that if we consider the $L^2_x$-inner product instead of the $\Hh^{-1/2}$-one, then this last sum is null, thanks to Lemma \ref{apx_lemma_omega_Q_u}. However the $\Hh^{-1/2}$-setting force us to analyze such sum, and we overcome the described obstacle, first considering the equivalence between $\Hh^{-1/2}$ and $\BB_{2,2}^{-1/2}$, and moreover thanks to decomposition \eqref{simmmetric_decomposition}, namely
\begin{equation*}
\begin{array}{ll}
	 \J_q^1(A,B)	:=\sum_{|q-q'|\leq 5	}	[\Dd_q, \,		\Sd_{q'-1}A]			\Dd_{q'}	B,	
	&\J_q^3(A,B)	:=							\Sd_{q -1}	A							\Dd_{q}		B,\\
	\\
	 \J_q^2(A,B)	:=\sum_{|q-q'|\leq 5	}	(	\Sd_{q'-1}A	-\Sd_{q-1}A)	\Dd_{q}	\Dd_{q'}	B,		
	&\J_q^4(A,B)	:=\sum_{ q'   \geq q- 5 }	\Dd_q(\Dd_{q'}A\,						\Sd_{q'+2}	B),
\end{array}
\end{equation*}
with
\begin{equation*}
\Dd_q(AB)	=		\J_q^1(A,B)
				+	\J_q^2(A,B)
				+	\J_q^3(A,B)
				+	\J_q^4(A,B),
\quad\quad\text{for any integer}\; q.				
\end{equation*}

\vspace{0.2cm}
\noindent
First, let us begin with
\begin{equation*}
	L\langle	\delta \Omega Q_2 ,\Delta\delta Q	\rangle_{\dot{H}^{-\frac{1}{2}}} 
	= 
	\sum_{q\in \ZZ} 2^{-q} 	L\langle \Dd_q ( \delta \Omega Q_2)		, \Dd_q \Delta \delta Q\rangle_{L^2_x}
	= 
	\sum_{q\in \ZZ} 
	\sum_{i=1}^4	2^{-q}	L\langle \J_q^i ( \delta \Omega,\, Q_2)	, \Dd_q \Delta \delta Q\rangle_{L^2_x}.
\end{equation*}
First we separately study the case $i=1,2,4$. The term related to $i=3$ is the challenging one and we are not able to evaluate it. However, we will see how such term is going to be erased.
Let us begin with $i=1$ then
\begin{equation*}
\begin{split}
	\I_q^1:=
	2^{-q}L\langle \J_q^1 ( \delta \Omega,\, Q_2)	, \Dd_q \Delta \delta Q\rangle_{L^2_x} 
	&= 
	L 2^{-q}
	\sum_{|q-q'|\leq 5	}
	\langle
		[\Dd_q,\, \Sd_{q'-1}Q_2 ]\Dd_{q'}\delta \Omega, \Dd_q \Delta \delta Q
	\rangle_{L^2_x}\\
	&\lesssim
	\sum_{|q-q'|\leq 5	}
	2^{-q}
	\|	[\Dd_q,\, \Sd_{q'-1}Q_2 ]\Dd_{q'}\delta \Omega	\|_{L^2_x		}
	\|	\Dd_q \Delta \delta Q							\|_{L^2_x		}.
\end{split}
\end{equation*} 
Hence, applying the commutator estimate (see Lemma $2.97$ in \cite{MR2768550}) we get
\begin{equation*}
\begin{split}
	\I_q^1
	\lesssim
	\sum_{|q-q'|\leq 5	}
	2^{-2q}
	\|	\Sd_{q'-1}\nabla Q_2							\|_{L^4_x	}
	\|	\Dd_{q'}\delta \Omega							\|_{L^4_x		}
	\|	\Dd_q \Delta \delta Q							\|_{L^2_x		}
	\lesssim
	\sum_{|q-q'|\leq 5	}
	\|	\Sd_{q'-1}\nabla 	Q_2							\|_{L^2_x		}^{\frac{1}{2}}
	\|	\Sd_{q'-1}\Delta 	Q_2							\|_{L^2_x		}^{\frac{1}{2}} {\scriptstyle \times} \\ {\scriptstyle \times}
	2^{-\frac{q'}{2}}
	\|	\Dd_{q'}\delta 	  	u							\|_{L^4_x		}
	2^{-\frac{q}{2}}
	\|	\Dd_q \Delta \delta	Q							\|_{L^2_x		}
	\lesssim
	\sum_{|q-q'|\leq 5	}
	\|	\nabla 				Q_2							\|_{L^2_x		}^{\frac{1}{2}}
	\|	\Delta 				Q_2							\|_{L^2_x		}^{\frac{1}{2}}
	\|	\Dd_{q'}\delta 	  	u							\|_{L^2_x		}
	2^{-\frac{q}{2}}
	\|	\Dd_q \Delta \delta	Q							\|_{L^2_x		},
\end{split}
\end{equation*}
which finally yields
\begin{equation*}
\begin{split}
	\I_q^1
	&\lesssim
	\|	\nabla		Q_2			\|_{			L^2_x			  }^\frac{1}{2}
	\|	\Delta		Q_2			\|_{			L^2_x			  }^\frac{1}{2}
	\|			\delta	u		\|_{			L^2_x			  }
	\|	\Delta	\delta	Q		\|_{	\Hh^{	-\frac{1}{2}	} }
	\lesssim
	\|	\nabla		Q_2			\|_{			L^2_x			  }^\frac{1}{2}
	\|	\Delta		Q_2			\|_{			L^2_x			  }^\frac{1}{2}
	\|			\delta	u		\|_{	\Hh^{	-\frac{1}{2}	}  }^\frac{1}{2}
	\|	\nabla	\delta	u		\|_{	\Hh^{	-\frac{1}{2}	}  }^\frac{1}{2}
	\|	\Delta	\delta	Q		\|_{	\Hh^{	-\frac{1}{2}	} },
\end{split}
\end{equation*}
that is
\begin{equation}\label{uniqueness_est_J_1}
\begin{split}
	L\sum_{q\in \ZZ}2^{-q}\langle \J_q^1 ( \delta \Omega,\, Q_2)	,& \Dd_q \Delta \delta Q\rangle_{L^2_x}
	\lesssim\\
	&\lesssim
	\|	\nabla			Q_2		\|_{	L^2_x					  }^2
	\|	\Delta			Q_2		\|_{	L^2_x					  }^2		
	\|			\delta	u		\|_{	\Hh^{-\frac{1}{2}		} }^2+
	C_{\nu}
	\|	\nabla	\delta	u		\|_{	\Hh^{-\frac{1}{2}		} }^2+
	C_{\Gamma, L}
	\|	\Delta	\delta	Q		\|_{	\Hh^{-\frac{1}{2} } }^2.
\end{split}
\end{equation}
Now, let us handle the case $i=2$. We argued almost as before:
\begin{equation*}
\begin{split}
	\I_q^2:= 2^{-q}L\langle \J_q^2 ( \delta \Omega,\, Q_2)	, \Dd_q \Delta \delta Q\rangle_{L^2_x} 
	&=
	L 2^{-q}
	\sum_{|q-q'|\leq 5	}
	\langle
		(\Sd_{q'-1}Q_2 - \Sd_{q-1}Q_2)\Dd_q \Dd_{q'}\delta \Omega, \Dd_q \Delta \delta Q
	\rangle_{L^2_x}\\
	&\lesssim
	2^{-q}
	\|	(\Sd_{q'-1}Q_2 - \Sd_{q-1}Q_2)					\|_{L^\infty_x}
	\|	\Dd_q \Dd_{q'}\delta \Omega						\|_{L^2_x}
	\|	\Dd_q \Delta \delta Q							\|_{L^2_x},
\end{split}
\end{equation*}
so that, observing that $\Sd_{q'-1}Q_2 - \Sd_{q-1}Q_2$ fulfills
\begin{equation*}
	\|	\Sd_{q'-1}			Q_2	-	\Sd_{q-1}		Q_2 \|_{L^\infty_x}
	\lesssim
	2^{-2q}
	\|	\Sd_{q'-1}	\Delta	Q_2	-	\Sd_{q-1}\Delta Q_2 \|_{L^\infty_x}
	\lesssim
	2^{- q}
	\|	\Sd_{q'-1}	\Delta	Q_2	-	\Sd_{q-1}\Delta Q_2 \|_{L^2_x	  },
\end{equation*}
then we obtain
\begin{equation*}
\begin{split}
	\I_q^2
	\lesssim
	2^{-2q}
	\sum_{|q-q'|\leq 5	}
	\|	&(\Sd_{q'-1}\Delta Q_2 - \Sd_{q-1}\Delta Q_2)	\|_{L^2_x}
	\|	\Dd_q \Dd_{q'}\delta \Omega						\|_{L^2_x}
	\|	\Dd_q \Delta \delta Q							\|_{L^2_x}
	\lesssim
	2^{-2q}
	\sum_{|q-q'|\leq 5	}
	\|	\Delta Q_2							\|_{L^2_x}{\scriptstyle \times} \\ &{\scriptstyle \times}
	\|	\Dd_{q'}\delta \Omega							\|_{L^2_x} 
	\|	\Dd_q \Delta \delta Q							\|_{L^2_x}
	\lesssim
	\sum_{|q-q'|\leq 5	}
	2^{-\frac{q'}{2}}
	\|	\Dd_{q'}		\delta u						\|_{L^2_x}
	2^{-\frac{q}{2}}
	\|	\Dd_q \Delta 	\delta Q						\|_{L^2_x}
	\|	\Delta Q_2										\|_{L^2_x}.
\end{split}	
\end{equation*}
Thus, it turns out that
\begin{equation}\label{uniqueness_est_J_2}
	L\sum_{q\in \ZZ} 2^{-q}\langle \J_q^2 ( \delta \Omega,\, Q_2)	, \Dd_q \Delta \delta Q\rangle_{L^2_x} \lesssim
	\|		\Delta			Q_2			\|_{	L^2_x				}^2
	\|				\delta	u			\|_{	\Hh^{-\frac{1}{2} } }^2+
	C_{\Gamma,L}
	\|		\Delta	\delta	Q			\|_{	\Hh^{-\frac{1}{2} } }^2
\end{equation}

\vspace{0.1cm}
\noindent
Now, we take into consideration the case $i=4$. Here we will use a convolution method and the Young inequality, since the sum in $q'$ is not finite. Then, let us observe that
\begin{equation*}
\begin{split}
	\I_q^4 := 
	2^{-q}L\langle \J_q^4 ( \delta \Omega,\, Q_2)	, \Dd_q \Delta \delta Q\rangle_{L^2_x} 
	&= L
		2^{-q}
		\sum_{q-q'\leq 5}
		\langle
			\Dd_{q'} Q_2 \Sd_{q'+2}\delta \Omega,\, \Dd_q \Delta \delta Q
		\rangle_{L^2_x}\\
	&\lesssim
		2^{-q}
		\sum_{q-q'\leq 5}
		\|	\Dd_{q'} Q_2			\|_{L^\infty_x	}
		\|	\Sd_{q'+2}\delta \Omega	\|_{L^2_x		}
		\|	\Dd_q \Delta \delta Q	\|_{L^2_x		}.
\end{split}
\end{equation*}
Observing that 
$\|	\Dd_{q'} Q_2			\|_{L^\infty_x	}	\lesssim 2^{q'}\|	\Dd_{q'}	 Q_2			\|_{L^2_x	}	\lesssim 2^{-q'}\|	\Dd_{q'}\Delta Q_2			\|_{L^2_x	}$ and 
$\|	\Dd_q \Delta \delta Q	\|_{L^2_x		}\lesssim 2^q\|	\Dd_q \nabla \delta Q	\|_{L^2_x		}$, it turns out that 
\begin{equation*}
\begin{split}
	\I_q^4 
	&\lesssim
		2^{-q}
		\sum_{q-q'\leq 5}
		2^{-q'}
		\|	\Dd_{q'}\Delta Q_2			\|_{L^2_x	}
		\|	\Sd_{q'+2}\delta \Omega		\|_{L^2_x		}
		2^{q}
		\|	\Dd_q   \nabla \delta Q		\|_{L^2_x		}\\
		&\lesssim
		\sum_{q-q'\leq 5}
		2^{\frac{q-q'}{2}}
		\|	\Dd_{q'}\Delta Q_2			\|_{	L^2_x				}
		2^{-\frac{q'+2}{2}}
		\|	\Sd_{q'+2}\delta \Omega		\|_{	L^2_x				}
		2^{-\frac{q}{2}}
		\|	\Dd_q   \nabla \delta Q		\|_{	L^2_x				}\\
		&\lesssim
		\|	\Delta Q_2					\|_{	L^2_x				}
		\|	\nabla \delta Q				\|_{	\Hh^{-\frac{1}{2} } }
		\sum_{q-q'\leq 5}
		2^{\frac{q-q'}{2}}
		2^{-\frac{q'+2}{2}}
		\|	\Sd_{q'+2}\delta \Omega		\|_{	L^2_x				}.		
\end{split}
\end{equation*}
Then, by convolution, the Young inequality and Proposition \ref{prop_besov_s_negative}, we finally obtain
\begin{equation}\label{uniqueness_est_J_4}
\begin{aligned}
	L\sum_{q\in \ZZ}2^{-q}\langle \J_q^4 ( \delta \Omega,\, Q_2)	, \Dd_q \Delta \delta Q\rangle_{L^2_x} 
	&\lesssim
	\|	\Delta Q_2					\|_{	L^2_x				}
	\|	\nabla \delta Q				\|_{	\Hh^{-\frac{1}{2} } }
	\|	\nabla \delta u				\|_{	\Hh^{-\frac{1}{2} } }\\
	&\lesssim
	\|	\Delta Q_2					\|_{	L^2_x				}^2
	\|	\nabla \delta Q				\|_{	\Hh^{-\frac{1}{2} } }^2 +
	C_{\nu}
	\|	\nabla \delta u				\|_{	\Hh^{-\frac{1}{2} } }^2
\end{aligned}
\end{equation}
Summarizing \eqref{uniqueness_est_J_1}, \eqref{uniqueness_est_J_2} and \eqref{uniqueness_est_J_4} and recalling the definition of $J^3_q(\delta \Omega, Q_2)$,  we finally get
\begin{equation*}
\begin{aligned}
	 L\langle	\delta \Omega Q_2 ,\Delta\delta Q	\rangle_{\dot{H}^{-\frac{1}{2}}} 
	-\sum_{q\in\ZZ}2^{-q}&\langle \Sd_{q-1} \delta \Omega\,\Dd_q Q_2, \Dd_q \Delta \delta Q\rangle_{L^2_x}
	\lesssim\\
	&\lesssim
	\tilde{\chi}_1\,\Phi + 
	C_{\nu}
	\|	\nabla	\delta	u		\|_{	\Hh^{-\frac{1}{2}		} }^2 +
	C_{\Gamma,L}
	\|	\Delta	\delta	Q		\|_{	\Hh^{-\frac{1}{2} } }^2,
\end{aligned}
\end{equation*}
where $\tilde{\chi}_1$ belongs to $L^1_{loc}(\RR_+)$. Hence, we need to analyze 
\begin{equation*}
	L\sum_{q\in\ZZ}2^{-q}\langle \Sd_{q-1} \delta \Omega\,\Dd_q Q_2, \Dd_q \Delta \delta Q\rangle_{L^2_x}
\end{equation*}
and this term is going to disappear by a simplification.

\vspace{0.1cm}
\noindent Now we handle the term $\langle Q_2 \delta\Omega, \Delta \delta Q\rangle_{\Hh^{-\frac{1}{2}}}$ of \eqref{uniqueness_challenging_term}. Observing that it is equal to 
$\langle  \tr(Q_2 \delta\Omega ), \tr \Delta \delta Q\rangle_{\Hh^{-\frac{1}{2}}}$, that is $-\langle   \delta\Omega Q_2, \Delta \delta Q \rangle_{\Hh^{-\frac{1}{2}}}$ then we proceed 
exactly as before, obtaining
\begin{equation}\label{uniqueness_first_estimate}
\begin{aligned}
	 L\langle	\delta \Omega Q_2 - Q_2\delta \Omega ,\Delta\delta Q	\rangle_{\dot{H}^{-\frac{1}{2}}} 
	-\sum_{q\in\ZZ}2^{-q}\langle \Sd_{q-1} \delta \Omega\,&\Dd_q Q_2 - \Dd_q Q_2\,\Sd_{q-1} \delta \Omega, \Dd_q \Delta \delta Q\rangle_{L^2_x}
	\lesssim\\
	&\lesssim
	\tilde{\chi}\,\Phi + 
	C_{\nu}
	\|	\nabla	\delta	u		\|_{	\Hh^{-\frac{1}{2}		} }^2 +
	C_{\Gamma,L}
	\|	\Delta	\delta	Q		\|_{	\Hh^{-\frac{1}{2} } }^2,
\end{aligned}
\end{equation}
so that it remains to control
\begin{equation}\label{uniqueness_first_term_to_control}
	L\sum_{q\in\ZZ}2^{-q}\langle \Sd_{q-1} \delta \Omega\,\Dd_q Q_2 - \Dd_q Q_2\,\Sd_{q-1} \delta \Omega, \Dd_q \Delta \delta Q\rangle_{L^2_x}.
\end{equation}

\vspace{0.1cm}
\noindent Now, we focus on $L\langle		Q_2 \Delta \delta Q	,	\nabla	\delta	u	\rangle_{\dot{H}^{-\frac{1}{2}}}$ of \eqref{uniqueness_challenging_term} and we use again decomposition \eqref{simmmetric_decomposition} as follows 
\begin{equation*}
	L\langle		Q_2 \Delta \delta Q	,	\nabla	\delta	u	\rangle_{	\dot{H}^{-\frac{1}{2} }	}
	=
	L\sum_{q\in\ZZ}2^{-q}
	\langle	\Dd_q(	Q_2 \Delta \delta Q	),	\Dd_q \nabla	\delta	u	\rangle_{	L^2_x		}
	=
	L\sum_{q\in\ZZ}
	\sum_{i=1}^4
	2^{-q}
	\langle	\J_q^i(	Q_2,\, \Delta \delta Q	),	\Dd_q \nabla	\delta	u	\rangle_{	L^2_x	}.
\end{equation*}
As before, we estimate the terms related to $i=1,2,4$ while when $i=3$ the associated term is going to be erased. When $i=1$ we get
\begin{equation*}
\begin{split}
	L2^{-q}\langle	\J_q^1(	Q_2,\, \Delta \delta Q	),	\Dd_q \nabla	&\delta	u	\rangle_{	L^2_x	} 
	=L\sum_{|q-q'|\leq 5} 2^{-q}
	\langle
		[\Dd_q,\, \Sd_{q'-1}Q_2 ]\Dd_{q'}\Delta \delta Q, \Dd_q \nabla \delta u
	\rangle_{L^2_x}\\
	&\lesssim
	\sum_{|q-q'|\leq 5}
	2^{-q}
	\|	[\Dd_q,\, \Sd_{q'-1}Q_2 ]\Dd_{q'}\Delta \delta Q	\|_{L^2_x	}
	\|	\Dd_q \nabla \delta u								\|_{L^2_x	}\\
	&\lesssim
	\sum_{|q-q'|\leq 5}
	2^{-2q}
	\|	\Sd_{q'-1}\nabla Q_2								\|_{L^4_x	}
	\|	\Dd_{q'}\Delta \delta Q								\|_{L^4_x	}
	\|	\Dd_q \nabla \delta u								\|_{L^2_x	}\\
	&\lesssim
	\sum_{|q-q'|\leq 5}
	2^{-q}
	\|	\Sd_{q'-1}\nabla 	Q_2								\|_{L^2_x	}^\frac{1}{2}
	\|	\Sd_{q'-1}\Delta 	Q_2								\|_{L^2_x	}^\frac{1}{2}
	\|	\Dd_{q'}\nabla	\delta Q							\|_{L^4_x	}
	\|	\Dd_q \nabla \delta u								\|_{L^2_x	}\\
	&\lesssim
	\sum_{|q-q'|\leq 5}
	\|				\nabla 	Q_2								\|_{L^2_x	}^\frac{1}{2}
	\|				\Delta 	Q_2								\|_{L^2_x	}^\frac{1}{2}
	\|	\Dd_{q'}\nabla	\delta Q							\|_{L^2_x	}
	2^{-\frac{q}{2}}
	\|	\Dd_q \nabla \delta u								\|_{L^2_x	}
\end{split}
\end{equation*}
Hence, taking the sum as $q\in\ZZ$,
\begin{equation*}
\begin{split}
	L\sum_{q\in\ZZ}2^{-q}\langle	\J_q^1(	Q_2,\, \Delta \delta Q	),	&\Dd_q \nabla	\delta	u	\rangle_{	L^2_x	}
	\lesssim
	\|				\nabla 	Q_2		\|_{	L^2_x					  }^\frac{1}{2}
	\|				\Delta 	Q_2		\|_{	L^2_x					  }^\frac{1}{2}
	\|	\nabla	\delta	Q			\|_{	L^2_x					  }
	\|	\nabla	\delta	u			\|_{	\Hh^{	-\frac{1}{2}	} }\\
	&\lesssim
	\|				\nabla 	Q_2		\|_{	L^2_x					  }^\frac{1}{2}
	\|				\Delta 	Q_2		\|_{	L^2_x					  }^\frac{1}{2}
	\|	\nabla	\delta	Q			\|_{	\Hh^{	-\frac{1}{2}	} }^\frac{1}{2}
	\|	\Delta	\delta	Q			\|_{	\Hh^{	-\frac{1}{2}	} }^\frac{1}{2}
	\|	\nabla	\delta	u			\|_{	\Hh^{	-\frac{1}{2}	} }\\
	&\lesssim
	\|	\nabla		Q_2				\|_{	L^2_x					  }^2	
	\|	\Delta		Q_2				\|_{	L^2_x					  }^2		
	\|	\nabla	\delta	Q			\|_{	\Hh^{-\frac{1}{2}		} }^2+
	C_{\nu}
	\|	\nabla	\delta	u			\|_{	\Hh^{	-\frac{1}{2}	} }^2+
	C_{\Gamma,L}
	\|	\Delta	\delta	Q			\|_{	\Hh^{	-\frac{1}{2}	} }^2.
\end{split}
\end{equation*}

\vspace{0.1cm}
\noindent
We evaluate the term related to $i=2$ as follows:
\begin{equation*}
\begin{split}
	L2^{-q}\langle	\J_q^2(	Q_2,\, \Delta \delta Q	),	&\Dd_q \nabla	\delta	u	\rangle_{	L^2_x	} 
	= L
	2^{-q}
	\langle
		(\Sd_{q'-1}Q_2 - \Sd_{q-1}Q_2)\Dd_q \Dd_{q'}\Delta \delta Q, \Dd_q \nabla \delta u
	\rangle_{L^2_x}\\
	&\lesssim
	\sum_{|q-q'|\leq 5}
	2^{-q}
	\|	(\Sd_{q'-1}Q_2 - \Sd_{q-1}Q_2)					\|_{L^\infty_x}
	\|	\Dd_q \Dd_{q'}\Delta Q							\|_{L^2_x}
	\|	\Dd_q \nabla \delta u							\|_{L^2_x},
\end{split}
\end{equation*}
so that
\begin{equation*}
\begin{split}
	L2^{-q}\langle	\J_q^2(	Q_2,\, \Delta \delta Q	),	&\Dd_q \nabla	\delta	u	\rangle_{	L^2_x	} \\
	&\lesssim
	\sum_{|q-q'|\leq 5}
	2^{-2q}
	\|	(\Sd_{q'-1}\Delta Q_2 - \Sd_{q-1}\Delta Q_2)	\|_{L^2_x}
	\|	\Dd_q \Dd_{q'}\Delta \delta Q					\|_{L^2_x}
	\|	\Dd_q \nabla \delta u							\|_{L^2_x}\\
	&\lesssim
	\sum_{|q-q'|\leq 5}
	2^{-2q}
	\|	\Sd_{q'-1}\Delta Q_2							\|_{L^2_x}
	\|	\Dd_{q}\Delta\delta Q							\|_{L^2_x}
	\|	\Dd_q \nabla \delta u							\|_{L^2_x}\\
	&\lesssim
	\sum_{|q-q'|\leq 5}
	2^{-\frac{q}{2}}
	\|	\Dd_{q'}		\delta u						\|_{L^2_x}
	2^{-\frac{q}{2}}
	\|	\Dd_q \Delta 	\delta Q						\|_{L^2_x}
	\|	\Delta Q_2										\|_{L^2_x}
\end{split}	
\end{equation*}
Thus, taking the sum in $q$, it turns out that
\begin{equation*}
	L\sum_{q\in\ZZ}2^{-q}\langle	\J_q^2(	Q_2,\, \Delta \delta Q	),	\Dd_q \nabla	\delta	u	\rangle_{	L^2_x	} 
	\lesssim
	\|		\Delta			Q_2			\|_{	L^2_x				}^2
	\|				\delta	u			\|_{	\Hh^{-\frac{1}{2} } }^2+
	C_{\Gamma,L}
	\|		\Delta	\delta	Q			\|_{	\Hh^{-\frac{1}{2} } }^2
\end{equation*}
At last, when $i=4$, 
\begin{equation*}
\begin{split}
	L2^{-q}\langle	\J_q^4(	Q_2,\, \Delta \delta Q	),	&\Dd_q \nabla	\delta	u	\rangle_{	L^2_x	} 
	= L
		2^{-q}
		\sum_{q-q'\leq 5}
		\langle
			\Dd_{q'} Q_2 \Sd_{q'+2}\Delta \delta Q,\, \Dd_q \nabla \delta u
		\rangle_{L^2_x}\\
	&\lesssim
		2^{-q}
		\sum_{q-q'\leq 5}
		\|	\Dd_{q'} Q_2				\|_{L^\infty_x	}
		\|	\Sd_{q'+2}\Delta \delta Q	\|_{L^2_x		}
		\|	\Dd_q \nabla \delta u		\|_{L^2_x		}\\
	&\lesssim
		2^{-q}
		\sum_{q-q'\leq 5}
		2^{-q'}
		\|	\Dd_{q'}\Delta Q_2			\|_{L^2_x	}
		\|	\Sd_{q'+2}\Delta \delta Q	\|_{L^2_x		}
		2^{q}
		\|	\Dd_q   \delta u			\|_{L^2_x		}\\
		&\lesssim
		\sum_{q-q'\leq 5}
		2^{\frac{q-q'}{2}}
		\|	\Dd_{q'}\Delta Q_2			\|_{	L^2_x				}
		2^{-\frac{q'+2}{2}}
		\|	\Sd_{q'+2}\Delta \delta Q		\|_{	L^2_x				}
		2^{-\frac{q}{2}}
		\|	\Dd_q    \delta u			\|_{	L^2_x				}\\
		&\lesssim
		\|	\Delta Q_2					\|_{	L^2_x				}
		\|		 \delta u				\|_{	\Hh^{-\frac{1}{2} } }
		\sum_{q-q'\leq 5}
		2^{\frac{q-q'}{2}}
		2^{-\frac{q'+2}{2}}
		\|	\Sd_{q'+2}\Delta \delta Q		\|_{	L^2_x				}		
\end{split}
\end{equation*}
Hence, by convolution, the Young inequalities and Proposition \ref{prop_besov_s_negative}, we obtain
\begin{align*}
	\sum_{q\in \ZZ}2^{-q}\langle	\J_q^4(	Q_2,\, \Delta \delta Q	),	\Dd_q \nabla	\delta	u	\rangle_{	L^2_x	} 
	&\lesssim
	\|	\Delta Q_2					\|_{	L^2_x				}
	\|	 \delta u					\|_{	\Hh^{-\frac{1}{2} } }
	\|	\Delta \delta Q				\|_{	\Hh^{-\frac{1}{2} } }\\
	&\lesssim
	\|	\Delta Q_2					\|_{	L^2_x				}^2
	\|	 \delta u					\|_{	\Hh^{-\frac{1}{2} } }^2+
	C_{\Gamma,L}
	\|	\Delta \delta Q				\|_{	\Hh^{-\frac{1}{2} } }^2
\end{align*}

\vspace{0.1cm}
\noindent
Since 
$\langle \Delta \delta Q Q_2, \nabla \delta u\rangle_{\Hh^{_\frac{1}{2}}} 
 = \langle \tr(\Delta \delta Q Q_2), \tr \nabla \delta u\rangle_{\Hh^{_\frac{1}{2}}} 
 = \langle Q_2 \Delta \delta Q , \tr \nabla \delta u\rangle_{\Hh^{_\frac{1}{2}}}$,  then we proceed as for estimate $\langle  Q_2\Delta \delta Q, \nabla \delta u\rangle_{\Hh^{_\frac{1}{2}}}$, so that we obtain the following control
\begin{equation}\label{uniqueness_second_estimate}
\begin{aligned}
	L		\langle		Q_2 \Delta \delta Q	-		 
						\Delta \delta Q  Q_2 		 ,	\nabla	\delta	u	\rangle_{\dot{H}^{-\frac{1}{2}}} - 
	L\sum_{q\in\ZZ}2^{-q}\langle \Sd_{q-1} Q_2\,\Dd_q \Delta \delta Q - \Dd_q \Delta \delta Q\,\Sd_{q-1} Q_2 \delta \Omega, \Dd_q \nabla u	\rangle_{L^2_x}
	\lesssim\\
	\lesssim
	\tilde{\chi}_2\,\Phi + 
	C_{\nu}
	\|	\nabla	\delta	u		\|_{	\Hh^{-\frac{1}{2}		} }^2 +
	C_{\Gamma,L}
	\|	\Delta	\delta	Q		\|_{	\Hh^{-\frac{1}{2} } }^2
\end{aligned}
\end{equation}
where $\chi_2$ belongs to $L^1_{loc}(\RR_+)$. Now, the term we need to erase is
\begin{equation}\label{uniqueness_second_term_to_control}
	L\sum_{q\in\ZZ}2^{-q}\langle \Sd_{q-1} Q_2\,\Dd_q \Delta \delta Q - \Dd_q \Delta \delta Q\,\Sd_{q-1} Q_2 \delta \Omega, \Dd_q \nabla u	\rangle_{L^2_x}.
\end{equation} 
Thus, summing  \eqref{uniqueness_first_term_to_control} and \eqref{uniqueness_second_term_to_control}, we obtain
\begin{equation*}
	L \sum_{q\in \ZZ}2^{-q}
	\Big\{
	\langle
		\Sd_{q-1} Q_2 \Dd_q \delta \Omega - \Dd_q\delta \Omega \,\Sd_{q-1} Q_2 , \Delta \Dd_q \delta Q
	\rangle_{L^2_x} 
	+
	\langle		\Sd_{q-1}Q_2  \Delta \Dd_q \delta Q	-		 
						\Delta \Dd_q \delta Q \,\Sd_{q-1} Q_2 		 ,	\nabla	\delta	u	
	\rangle_{L^2_x}
	\Big\},
\end{equation*}
which is a series with every coefficients null, thanks to Lemma \ref{apx_lemma_omega_Q_u}. In virtue of this last result, recalling \eqref{uniqueness_first_estimate} and 
\eqref{uniqueness_second_estimate}, we finally obtain
\begin{equation*}
	L\langle		\delta \Omega Q_2 - Q_2 \delta \Omega      ,\Delta\delta Q	\rangle_{\dot{H}^{-\frac{1}{2}}} +
	L		\langle		Q_2 \Delta \delta Q	-		 
						\Delta \delta Q  Q_2 		 ,	\nabla	\delta	u	\rangle_{\dot{H}^{-\frac{1}{2}}}
	\lesssim
	\tilde{\chi} \Phi + 
	C_{\nu}
	\|	\nabla	\delta	u		\|_{	\Hh^{-\frac{1}{2}		} }^2 +
	C_{\Gamma,L}
	\|	\Delta	\delta	Q		\|_{	\Hh^{-\frac{1}{2} } }^2.
\end{equation*}

\subsection*{Conclusion} Recalling \eqref{uniqueness_energy_equality} and summarizing all the estimate of the previous two sub-sections, we conclude that there exists a function $\chi$ which 
belongs to $L^1_{loc}(\RR_+)$ such that
\begin{equation*}
	\frac{\dd}{\dd t}\Phi(t) + \nu \|	\nabla	\delta	u		\|_{	\Hh^{-\frac{1}{2}		} }^2 + \Gamma L^2 \|	\Delta	\delta	Q		\|_{	\Hh^{-\frac{1}{2} } }^2
	\lesssim \chi(t)\Phi(t)+ C_{\nu}
	\|	\nabla	\delta	u		\|_{	\Hh^{-\frac{1}{2}		} }^2 +
	C_{\Gamma,L}
	\|	\Delta	\delta	Q		\|_{	\Hh^{-\frac{1}{2} } }^2
\end{equation*}
for almost every $t\in \RR_+$. Thus, choosing $C_{\Gamma,L}$ and $C_{\nu}$ small enough, we absorb the last two terms in the right-hand side by the left-hand side, finally obtaining
\begin{equation*}
	\frac{d}{\dd t}\Big[\frac{1}{2}	\|			\delta u(t)	\|_{\dot{H}^{-\frac{1}{2}}}^2 +
		L			\|	\nabla	\delta Q(t)	\|_{\dot{H}^{-\frac{1}{2}}}^2\Big] 
	\lesssim 
	\chi \Big[\frac{1}{2}	\|			\delta u(t)	\|_{\dot{H}^{-\frac{1}{2}}}^2 +
		L			\|	\nabla	\delta Q(t)	\|_{\dot{H}^{-\frac{1}{2}}}^2\Big].
\end{equation*}
Since the initial datum is null and thanks to the Gronwall inequality, we deduce that $(\delta u, \nabla \delta Q)=0$ which yields 
$(\delta u, \delta Q)=0$, since $\delta Q(t)$ decades to $0$ at infinity for almost every $t$. Hence, we have finally achieved the uniqueness of the weak solution for 
system \eqref{main_system}. 

\section{Regularity Propagation}
\noindent
We now handle the propagation of low regularity, namely we prove Theorem \ref{Thm_Regularity}.

\begin{proof}[Proof of Theorem \ref{Thm_Regularity}]
Let us consider the following sequence of system:
\begin{equation}\label{Friedrichs_scheme_complete}
\tag{$\tilde{P}_n$}
\begin{cases}
	\; \partial_t Q^n + J_n  \Pp \big( J_nu^n \nabla J_n Q^n \big) - J_n \Pp \big(  J_n \Omega^n J_n Q^n \big) +\\ 
		\quad\quad\quad\quad\quad\quad\quad\quad\quad\quad\quad\quad\;
		+\,J_n \Pp\big( J_n Q^n  J_n \Omega^n \big) - \Gamma L \Delta J_n Q^n =P^n(Q^n)									&\RR_+ \times \RR^2,\\
	\;	\partial_t u^n + J_n\Pp \big( J_n u^n \nabla J_n u^n \big ) - \nu \Delta J_n u^n = \\
		\quad\quad\quad\quad
		= \Gamma L\Div J_n \Pp \{ J_n Q^n \Delta J_n Q^n - \Delta J_n Q^n J_n Q^n - \nabla J_n Q^n \odot \nabla J_n Q^n\}	&\RR_+ \times \RR^2,\\
	\;	\Div\, u^n = 0																									&\RR_+ \times \RR^2,\\
	\;	(u^n,\,Q^n)_{|t=0} = (u_0,\,Q_0)																				&\quad\quad\;\;\RR^2,
\end{cases}
\end{equation}
where
\begin{equation*}
	P^n(Q^n):=    -a J_n Q^n + b \big[ J_n (J_n Q^n J_n Q^n ) - \trc\{J_n( J_nQ^n J_nQ^n)\}\frac{\Id}{3}\big] - cJ_nQ^n \trc\{J_n( J_nQ^n J_nQ^n)\}.
\end{equation*}
Moreover we recall that $J_n$ is the regularizing operator defined by
\begin{equation*}
	\hat{J_n f}(\xi) = 1_{[\frac{1}{n},\,n]}(\xi)\hat{f}(\xi)
\end{equation*}
and $\Pp$ stands for the Leray projector. The Friedrichs scheme related to \eqref{Friedrichs_scheme_complete} is not much different to the \eqref{system_friedrichs}-one, however here the $Q$-tensor equation has been regularized, as well. System \eqref{system_friedrichs} has been utilized in \cite{MR2864407} and the authors have proven the existence of a strong solution $(u^n,\,Q^n)$  which converges to a weak solution for \eqref{Friedrichs_scheme_complete}, as $n$ goes to $\infty$ (up to a subsequence).  Thanks to our uniqueness result, Theorem \ref{Thm_Uniqueness}, we deduce that such solution is exactly the one determined by Theorem \ref{Thm_Existence} and it is unique. Hence, instead of proceeding by a priori estimate (as in \cite{MR2864407}), we formalize our proof, evaluating directly the \eqref{Friedrichs_scheme_complete}-scheme. We will establish some estimates, which are uniformly in $n$, which yields that the weak-solution of \eqref{main_system} fulfills them as well. This is only a strategy in order to formalize the a priori-estimate, while the major part of our proof releases on the inequalities we are going to proof.

\vspace{0.1cm}
\noindent
Since $(J_nu^n,\,J_nQ^n) = (u^n,\,Q^n)$ (by uniqueness), then $(u^n(t), Q^n(t))$ belongs to $H^{1+s} \times H^{2+s}$ for almost every $t\in \RR_+$ and for every $n\in\NN$. We apply $\Dd_q$ to the first and the second equations of \eqref{Friedrichs_scheme_complete}, then we apply $\langle\, \cdot\, ,\Dd_q u^n\rangle_{L^2}$ to the first one and 
$-L\langle\, \cdot\, ,\Dd_q \Delta  Q^n\rangle_{L^2}$ to the second one, obtaining the following identity:
\begin{align*}
	&\frac{\dd}{\dd t}\Big[ \|\Dd_q u^n \|_{L^2}^2 + L \| \Dd_q \nabla Q^n \|_{L^2}^2 \Big] 
	+	\nu 		\| 	\Dd_q \nabla u^n \|_{L^2}^2
	+	\Gamma L^2 	\|	\Dd_q \Delta Q^n \|_{L^2}^2
	= \\&=
		\langle \Dd_q( 	\Delta Q^n  Q^n - Q^n \Delta Q^n 	)	,  \Dd_q	\nabla	u^n		\rangle_{L^2}
	-	\langle \Dd_q(	u^n \cdot \nabla u^n	 			) 	,  \Dd_q			u^n		\rangle_{L^2}
	+	\langle \Dd_q( 	\nabla Q^n \odot \nabla Q^n			)	,  \Dd_q 	\nabla	u^n		\rangle_{L^2} +\\&\quad
	+L	\langle	\Dd_q(	u^n \cdot \nabla Q^n				)	,  \Dd_q 	\Delta 	Q^n		\rangle_{L^2}
	+L	\langle	\Dd_q(	\Omega^n Q^n -Q^n \Omega^n			)	,  \Dd_q 	\Delta 	Q^n		\rangle_{L^2}
	-L	\langle	\Dd_q			P^n(Q^n) 			  			,  \Dd_q 	\Delta 	Q^n		\rangle_{L^2}.
\end{align*}
Multiplying both left-hand and the right-hand sides by $2^{2qs}$ and taking the sum as $q\in \ZZ$ we obtain
\begin{equation}\label{sec3_eq_energy_Dq}
\begin{aligned}
	&\frac{\dd}{\dd t}\Big[ \|u^n \|_{\Hh^s}^2 + L \|\nabla Q^n \|_{\Hh^s}^2 \Big] 
	+	\nu 		\| 	\nabla u^n \|_{\Hh^s}^2
	+	\Gamma L^2 	\|	\Delta Q^n \|_{\Hh^s}^2
	= \\&=
	 L	\langle 	\Delta Q^n  Q^n - Q^n \Delta Q^n 			,   		\nabla	u^n		\rangle_{\Hh^s}
	-	\langle 	u^n \cdot \nabla u^n	 					,  					u^n		\rangle_{\Hh^s}
	+L	\langle  	\nabla Q^n \odot \nabla Q^n					,   		\nabla	u^n		\rangle_{\Hh^s} +\\&\quad
	+L	\langle		u^n \cdot \nabla Q^n						,  			\Delta 	Q^n		\rangle_{\Hh^s}
	+L	\langle		\Omega^n Q^n -Q^n \Omega^n					,   		\Delta 	Q^n		\rangle_{\Hh^s}
	-L	\langle		P^n(Q^n) 			  						,   		\Delta 	Q^n		\rangle_{\Hh^s}.
\end{aligned}
\end{equation}
The key part of our proof relies on the Osgood inequality, therefore we need to estimate all the terms of the right-hand side of \eqref{sec3_eq_energy_Dq}. First, let us proceed estimating the easier terms.

\vspace{0.2cm}
\noindent
\underline{Estimate of $\langle 	u^n \cdot \nabla u^n	 					,  					u^n		\rangle_{\Hh^s}$ }
\vspace{0.1cm}

\noindent
We begin with $\langle \Dd_q(	u^n \cdot \nabla u^n	 			) 	,  \Dd_q			u^n	\rangle_{L^2}$, with $q\in \ZZ$. Passing through the 
Bony decomposition
\begin{align*}
	\langle &\Dd_q(	u^n \cdot \nabla u^n	 													,\,  \Dd_q			u^n	\rangle_{L^2}= \\ &=
\underbrace{	
	\sum_{|q-q'|\leq 5}
	\langle \sum_{i=1}^2 \Dd_q T_{u^n_i}\partial_i u^n + \Dd_q T_{\partial_i u^n}u^n_i		 	,\,  \Dd_q			u^n	\rangle_{L^2}}_{\Aa_q} +
\underbrace{	
	\sum_{q'\geq q-5}
	\langle \sum_{i=1}^2 \Dd_q R( 				u^n_i,\,\partial_i u^n						) 	,\, \Dd_q			u^n	\rangle_{L^2}}_{\Bb_q}
\end{align*}
We handle the term $\Aa_q$ as follows:
\begin{align*}
	\Aa_q 	&\lesssim	
				\sum_{|q-q'|\leq 5}
				\Big[ 
					\| 	\Sd_{q'-1} 			u^n 	\|_{L^\infty} 
					\| 	\Dd_{q'	 } 	\nabla 	u^n 	\|_{L^2		}	+ 
					\| 	\Sd_{q'-1} 	\nabla	u^n 	\|_{L^\infty} 
					\| 	\Dd_{q'	 } 		 	u^n 	\|_{L^2		}
				\Big]
					\|	\Dd_q				u^n		\|_{L^2		}\\
			&\lesssim	
				\sum_{|q-q'|\leq 5}
				\Big[ 
					\| 	\Sd_{q'-1} 			u^n 	\|_{L^\infty} 
					\| 	\Dd_{q'	 } 	\nabla 	u^n 	\|_{L^2		}	+ 
					2^{q'}
					\| 	\Sd_{q'-1} 			u^n 	\|_{L^\infty} 
					\| 	\Dd_{q'	 } 		 	u^n 	\|_{L^2		}
				\Big]
					\|	\Dd_q				u^n		\|_{L^2		}\\
			&\lesssim	
				\sum_{|q-q'|\leq 5}
				\Big[ 
					\| 	\Sd_{q'-1} 			u^n 	\|_{L^\infty} 
					\| 	\Dd_{q'	 } 	\nabla 	u^n 	\|_{L^2		}	+
					\| 	\Sd_{q'-1} 			u^n 	\|_{L^\infty} 
					\| 	\Dd_{q'	 } 	\nabla 	u^n 	\|_{L^2		}
				\Big]
					\|	\Dd_q				u^n		\|_{L^2}\\
			&\lesssim	
					\|						u^n		\|_{L^\infty}
					\|	\Dd_q				u^n		\|_{L^2		}
				\sum_{|q-q'|\leq 5}
					\| 	\Dd_{q'	 } 	\nabla 	u^n 	\|_{L^2		},
\end{align*}
so that, multiplying by $2^{2sq}$ and taking the sum as $q\in\ZZ$,
\begin{equation}\label{sec3_est1}
	\sum_{q\in\ZZ} 
		2^{2qs} \Aa_q
	\lesssim
					\|						u^n		\|_{L^\infty}
	\sum_{q\in\ZZ} 
	\Big\{
		2^{2qs}		\|	\Dd_q				u^n		\|_{L^2		}
				\sum_{|q-q'|\leq 5}
					\| 	\Dd_{q'	 } 	\nabla 	u^n 	\|_{L^2		}
	\Big\}
	\lesssim
					\|						u^n		\|_{L^\infty}
					\|						u^n		\|_{\Hh^s	}
					\|				\nabla	u^n		\|_{\Hh^s	}.
\end{equation}
The control of $\Bb_q$ relies on convolution and the Young inequality, indeed
\begin{equation*}
	\Bb_q 	\lesssim 	
				\sum_{\substack{	q'\geq q-5\\|l|\leq 1}	}  
					\|	\Dd_{q'+l} 			u^n 	\|_{L^\infty}
					\|	\Dd_{q'	 }	\nabla 	u^n		\|_{L^2}
					\|	\Dd_{q 	 } 			u^n		\|_{L^2} 
			 \lesssim	
			  		\|	 					u^n 	\|_{L^\infty}
					\|	\Dd_{q	 } 			u^n		\|_{L^2		}
				\sum_{		 		q'\geq q-5				}  
					\|	\Dd_{q'	 }	\nabla 	u^n		\|_{L^2		},
\end{equation*}
hence
\begin{align*}
		\sum_{q\in\ZZ} 
			2^{2qs} \Bb_q
	&\lesssim
					\|	 					u^n 	\|_{L^\infty}
		\sum_{q\in\ZZ}
		\Big\{	
			2^{2qs}
					\|	\Dd_{q	 } 			u^n		\|_{L^2		}
			\sum_{		 		q'\geq q-5				}  
					\|	\Dd_{q'	 }	\nabla 	u^n		\|_{L^2		}
		\Big\}\\
	&\lesssim
					\|	 					u^n 	\|_{L^\infty}
		\sum_{q\in\ZZ}
		\Big\{	
			2^{qs}
					\|	\Dd_{q	 } 			u^n		\|_{L^2		}
			\sum_{		 		q'\geq q-5				}
			2^{(q-q')s}
			2^{   q' s} 
					\|	\Dd_{q'	 }	\nabla 	u^n		\|_{L^2		}
		\Big\}\\
	&\lesssim
					\|	 					u^n 	\|_{L^\infty}
					\|	 					u^n 	\|_{\Hh^s	}
		\sum_{q\in\ZZ}
		\Big\{	
			2^{qs}
					\|	\Dd_{q	 } 			u^n		\|_{L^2		}
			\sum_{		 		q'\in\ZZ				}
			2^{(q-q')s}1_{(-\infty,\,5)}(q-q')
			b_{q'}
		\Big\},
\end{align*}
where $(b_{q'})_\ZZ$ belongs to $l^2(\ZZ)$. Thus, we obtain
\begin{equation}\label{sec3_est2}
	\sum_{q\in\ZZ} 
			2^{2qs} \Bb_q
	\lesssim
					\|	 					u^n 	\|_{L^\infty}
					\|	 					u^n 	\|_{\Hh^s	}
					\|	 			\nabla	u^n 	\|_{\Hh^s	},
\end{equation}
thanks to the Young inequality. Finally, summarizing \eqref{sec3_est1} and \eqref{sec3_est2}, we obtain
\begin{equation}\label{sec3_major_est1}
	\langle 	u^n \cdot \nabla u^n	 					,  					u^n		\rangle_{\Hh^s}
	=
	\sum_{q\in\ZZ} 
		2^{2qs}
		\langle 	
			\Dd_q(	u^n \cdot \nabla u^n	)	,\,  \Dd_q	u^n	
		\rangle_{L^2} 
	\lesssim
					\|	 					u^n 	\|_{L^\infty}
					\|	 					u^n 	\|_{\Hh^s	}
					\|	 			\nabla	u^n 	\|_{\Hh^s	}.	
\end{equation}

\vspace{0.2cm}
\noindent
\underline{Estimate of $\langle 	u^n\cdot \nabla Q^n	 					,  					\Delta Q^n	\rangle_{\Hh^s}$ }
\vspace{0.1cm}

\noindent
Arguing exactly as for proving \eqref{sec3_major_est1}, we obtain
\begin{equation}\label{sec3_major_est3}
	\langle 	u^n\cdot \nabla Q^n	 					,  					\Delta Q^n	\rangle_{\Hh^s} =
	\sum_{q\in\ZZ} 
		2^{2qs} 
			\langle \Dd_q( u^n\cdot \nabla Q^n,\,\Dd_q \Delta Q^n\rangle_{L^2}
	\lesssim
					\| 						u^n	\|_{L^\infty}
					\|				\nabla 	Q^n	\|_{\Hh^s	}
					\|				\Delta 	Q^n	\|_{\Hh^s	}.
\end{equation} 

\vspace{0.2cm}
\noindent
\underline{Estimate of $\langle \nabla Q^n \odot \nabla Q^n  ,\, 	\nabla u^n \rangle_{\Hh^s}$ }
\vspace{0.1cm}

\noindent
We keep on our control, evaluating the term $\langle \Dd_q (\nabla Q^n \odot \nabla Q^n ),\, \Dd_q \nabla u^n \rangle_{L^2}$, with $q\in\ZZ$. The explicit integral formula of such term is the following one:
\begin{align*}
	\int_{\RR^2} 
	&
		\sum_{i,k=1}^2
			\Dd_q (\,\trc\{\partial_i Q \partial_k Q \}\,) \Dd_q \partial_k u^n_i 
	=	
	\int_{\RR^2} 
		\sum_{i,k=1}^2	\sum_{j,l=1}^3
			\Dd_q[\, \partial_i Q_{jl}^n\, \partial_k Q_{lj}^n \,] \Dd_q \partial_k u^n_i \\
	&=
\underbrace{	
	\int_{\RR^2} 
		\sum_{i,k=1}^2	\sum_{j,l=1}^3
			\Dd_q[\, \dot{T}_{\partial_i Q_{jl}^n}\partial_k Q_{lj}^n + \dot{T}_{\partial_k Q_{lj}^n}\partial_i Q_{jl}^n \,] \Dd_q \partial_k u^n_i	}_{\Cc_q} +
\underbrace{
		\int_{\RR^2} 	\sum_{i,k,j,l}
			\Dd_q \dot{R}(\partial_i Q_{jl}^n,\,\partial_k Q_{lj}^n)  \Dd_q \partial_k u^n_i													}_{\DD_q},	
\end{align*}
where we have used the Bony decomposition again. First, let us observe that
\begin{equation*}
	\Cc_q	
	\lesssim	
		\sum_{|q-q'|\leq 5}	
					\|	S_{q-1}		\nabla Q^n	\|_{L^\infty}
					\|	\Dd_{q'}	\nabla Q^n	\|_{L^2		}
					\|	\Dd_q		\nabla u^n	\|_{L^2		}
	\lesssim
					\|				\nabla Q^n	\|_{L^\infty}
					\|	\Dd_q		\nabla u^n	\|_{L^2		}
		\sum_{|q-q'|\leq 5}	
					\|	\Dd_{q'}	\nabla Q^n	\|_{L^2		},											
\end{equation*}
which yields
\begin{equation}\label{sec3_est3}
\begin{aligned}
	\sum_{q\in\ZZ}
		2^{2qs}	\Cc_q	
	&\lesssim 	
					\|				\nabla Q^n	\|_{L^\infty}
	\sum_{q\in\ZZ}
		2^{2qs}
	\Big\{
					\|	\Dd_q		\nabla u^n	\|_{L^2		}					
			\sum_{|q-q'|\leq 5}		
					\|	\Dd_{q'}	\nabla Q^n	\|_{L^2		}
	\Big\}\\
	&\lesssim
					\|				\nabla Q^n	\|_{L^\infty}
					\|				\nabla u^n	\|_{\Hh^s	}
					\|				\nabla Q^n	\|_{\Hh^s	}.
\end{aligned}
\end{equation}
Moreover, considering $\DD_q$, we get
\begin{align*}
	\DD_q 	
	&\lesssim 	
		\sum_{\substack{q'\geq q-5 \\ |l|\leq 5}}	
					\|	\Dd_{q'+l}	\nabla Q^n	\|_{L^\infty}
					\|	\Dd_{q'  }	\nabla Q^n	\|_{L^2		}
					\|	\Dd_{q	 }	\nabla u^n	\|_{L^2		}\\
	&\lesssim 	
					\|				\nabla Q^n	\|_{L^\infty}
					\|	\Dd_{q	 }	\nabla u^n	\|_{L^2		}
		\sum_{			q'\geq q-5 				}	
					\|	\Dd_{q'  }	\nabla Q^n	\|_{L^2		},				
\end{align*}
so that, proceeding as in the proof of \eqref{sec3_est2},
\begin{equation}\label{sec3_est4}
\begin{aligned}
	\sum_{q\in\ZZ}
		2^{2qs} \DD_q
	&\lesssim
					\|				\nabla Q^n	\|_{L^\infty}
	\sum_{q\in\ZZ}
	\Big\{
		2^{qs}
					\|	\Dd_{q	 }	\nabla u^n	\|_{L^2		}
		\sum_{			q'\in\ZZ 				}	
		2^{(q'-q)s}
		2^{ q'   s}
					\|	\Dd_{q'  }	\nabla Q^n	\|_{L^2		}
	\Big\}\\
	&\lesssim
					\|				\nabla Q^n	\|_{L^\infty}
					\|				\nabla u^n	\|_{\Hh^s	}
					\|				\nabla Q^n	\|_{\Hh^s	},
\end{aligned}				
\end{equation}
thanks to the Young inequality. Thus, summarizing \eqref{sec3_est3} and \eqref{sec3_est4}, we achieve
\begin{equation}\label{sec3_major_est2}
	\sum_{q\in\ZZ} 2^{2qs}\langle \Dd_q (\nabla Q^n \odot \nabla Q^n ),\, \Dd_q \nabla u^n \rangle_{L^2}
	\lesssim
					\|				\nabla Q^n	\|_{L^\infty}
					\|				\nabla u^n	\|_{\Hh^s	}
					\|				\nabla Q^n	\|_{\Hh^s	},
\end{equation}

\vspace{0.2cm}
\noindent
\underline{Estimate of $\langle \Delta Q^n Q^n -Q^n \Delta Q^n, \,  \nabla u^n\rangle_{\Hh^s}$ }
\vspace{0.2cm}

\noindent
Now, we carry out of $\langle \Delta Q^n Q^n -Q^n \Delta Q^n, \,  \nabla u^n\rangle_{\Hh^s}$. This is the first non trivial term to evaluate. We choose to use the decomposition \eqref{simmmetric_decomposition}, presented in the preliminaries, instead of the classical Bony decomposition (which we have used until now). We will remark the presence of a term inside 
such decomposition, which is hard to control. However we will see that such drawback is going to be erased. Let us begin controlling 
$ \langle \,	Q^n \Delta Q^n, 			\, 	 			\nabla 	u^n	\rangle_{\Hh^s}$:
\begin{equation*}
		\langle \,	Q^n \Delta Q^n, 			\, 	 			\nabla 	u^n	\rangle_{\Hh^s} 
	= 
	\sum_{q\in\ZZ} 2^{2qs}
		\langle		\Dd_q (Q^n \Delta Q^n),		\, 		\Dd_q 	\nabla 	u^n	\rangle_{L^2} = 
	\sum_{q\in\ZZ}\sum_{i=1}^4 2^{2qs}
		\langle 	\J_q^i (Q^n,\,\Delta Q^n),	\,		\Dd_q 	\nabla 	u^n	\rangle_{L^2}.
\end{equation*}
where $\J^i_q$ has been defined by \eqref{simmmetric_decomposition}, for $i=1,\dots, 4$.
When $i=1$, we point out that
\begin{align*}
	\langle\, 	\J_q^1 (Q^n,\,\Delta Q^n)		,	\,\Dd_q \nabla u^n		\rangle_{L^2} 
	&= \sum_{|q-q'|\leq 5}
						\langle 	[ \Dd_q,\,\Sd_{q'-1}Q^n]\Dd_{q'}\Delta Q^n,	\, \Dd_q\nabla u^n		\rangle_{L^2} \\
	&\lesssim 	\sum_{|q-q'|\leq 5}
						2^{-q'}	\|	\Sd_{q'-1}\nabla Q^n	\|_{L^\infty}		\|	\Dd_{q'} \Delta	Q^n \|_{L^2}	\| \Dd_q \nabla u	\|_{L^2}\\
	&\lesssim					\|			\nabla Q^n		\|_{L^\infty}
				\sum_{|q-q'|\leq 5}												\| 	\Dd_{q'} \nabla Q^n \|_{L^2}	\| \Dd_q \nabla u	\|_{L^2}.
\end{align*}
which yields
\begin{equation}\label{sec3_est5}
	\sum_{q\in\ZZ}2^{2qs} 	\langle 	\J_q^1 (Q^n,\,\Delta Q^n)		,	\,\Dd_q \nabla u^n		\rangle_{L^2} 
	\lesssim
	\|			\nabla Q^n									\|_{L^\infty}
	\| 			\nabla Q^n									\|_{\Hh^s	}
	\|			\nabla u^n									\|_{\Hh^s	}.
\end{equation}
On the other hand, for $i=2$, we proceed as follows:
\begin{align*}
						\langle 	\J_q^2 (Q^n,\,\Delta Q^n)		,	\,\Dd_q \nabla u^n		\rangle_{L^2} 
	&= \sum_{|q-q'|\leq 5}
						\langle 	(\Sd_{q'-1}Q^n-\Sd_{q-1}Q^n)\Dd_q\Dd_{q'}\Delta Q^n,	\, \Dd_q\nabla u^n		\rangle_{L^2} \\
	&\lesssim \sum_{|q-q'|\leq 5}	
				\|	\Sd_{q'-1}Q^n-\Sd_{q-1}Q^n 	\|_{L^\infty}
				\| 	\Dd_q\Dd_{q'}\Delta Q^n		\|_{L^2		}
				\|	\Dd_q\nabla u^n				\|_{L^2		}\\
	&\lesssim	\sum_{|q-q'|\leq 5}	
				\|	\Sd_{q'-1}	\nabla	Q^n-\Sd_{q-1}	\nabla	Q^n 	\|_{L^\infty}			
				\| 	\Dd_q	\nabla Q^n									\|_{L^2		}
				\|	\Dd_q	\nabla u^n									\|_{L^2		}\\
	&\lesssim	\|	\nabla Q^n											\|_{L^\infty}
				\| 	\Dd_q	\nabla Q^n									\|_{L^2		}
				\|	\Dd_q	\nabla u^n									\|_{L^2		},
\end{align*}
which yields
\begin{equation}\label{sec3_est6}
	\sum_{q\in\ZZ}2^{2qs}\langle 	\J_q^2 (Q^n,\,\Delta Q^n)		,	\,\Dd_q \nabla u^n		\rangle_{L^2} 
	\lesssim
	\|			\nabla Q^n									\|_{L^\infty}
	\| 			\nabla Q^n									\|_{\Hh^s	}
	\|			\nabla u^n									\|_{\Hh^s	}.
\end{equation}
The case $i=4$ is handled as follows:
\begin{align*}
			\langle 	\J_q^4 (Q^n,\,\Delta Q^n)					,\,		\Dd_q \nabla u^n	\rangle_{L^2}
	&= \sum_{q'\geq q-5}
			\langle 	\Dd_q[\,\Dd_{q'} Q^n\Sd_{q'+2} \Delta Q^n]	,\, 	\Dd_q \nabla u^n	\rangle_{L^2}\\
	&\lesssim	
	\sum_{q'\geq q-5}	
			\|	\Dd_{q'} Q^n 				\|_{L^2		}
			\|	\Sd_{q'+2} 	\Delta 	Q^n		\|_{L^\infty}
			\|	\Dd_q		\nabla	u^n		\|_{L^2		}\\
	&\lesssim	
	\sum_{q'\geq q-5}	
			\|	\Dd_{q'} 	\nabla 	Q^n		\|_{L^2		}
			\|	\Sd_{q'+2} 	\nabla 	Q^n		\|_{L^\infty}
			\|	\Dd_q 		\nabla 	u^n		\|_{L^2		}\\
	&\lesssim
			\|	\nabla 				Q^n		\|_{L^\infty}
			\|	\Dd_q 		\nabla 	u^n		\|_{L^2		}
	\sum_{q'\geq q-5}
			\|	\Dd_{q'} 	\nabla 	Q^n		\|_{L^2		}.
\end{align*}
Therefore, multiplying by $2^{2qs}$ and taking the sum as $q\in \ZZ$, 
\begin{align*}
	\sum_{q\in\ZZ}
		2^{2qs}
		\langle 	
	&		\J_q^4 (Q^n,\,\Delta Q^n),\,		\Dd_q \nabla u^n	
		\rangle_{L^2}
	\lesssim\\
	&\lesssim
			\|				\nabla Q^n		\|_{L^\infty}
	\sum_{q\in\ZZ}
	\Big(
		2^{qs}
			\|	\Dd_q 		\nabla	u^n		\|_{L^2		}
		\sum_{q'\in\ZZ}
		2^{(q-q')s}1_{(-\infty,5)}(q-q')
		2^{q's}
		\|	\Dd_{q'} 	\nabla 	Q^n	\|_{L^2		}
	\Big)\\
	&\lesssim
			\|				\nabla Q^n		\|_{L^\infty}
			\|				\nabla u^n		\|_{\Hh^s}
	\Big\{
	\sum_{q\in\ZZ}
	\Big(
		\sum_{q'\in\ZZ}
		2^{(q-q')s}1_{(-\infty,5)}(q-q')
		2^{q's}
		\|	\Dd_{q'} 	\nabla 	Q^n	\|_{L^2		}
	\Big)^2
	\Big\}^{\frac{1}{2}}, 
\end{align*}
so that, by convolution and the Young inequality
\begin{equation}\label{sec3_est7}
	\sum_{q\in\ZZ}2^{2qs}\langle 	\J_q^4 (Q^n,\,\Delta Q^n),\,		\Dd_q \nabla u^n	\rangle_{L^2}
	\lesssim
			\|				\nabla	Q^n		\|_{L^\infty}
			\|				\nabla u^n		\|_{\Hh^s}
			\|				\nabla Q^n		\|_{\Hh^s}.
\end{equation}
It remains to control the term related to $\J^3_q$, namely
\begin{equation}\label{sec3_est10}
	\sum_{q\in\ZZ} 2^{2qs} \langle \J^3(Q^n, \Delta Q^n), \Dd_q \nabla u^n\rangle_{L^2} = 	
	\sum_{q\in\ZZ} 2^{2qs} \langle \Sd_{q-1}Q^n\Dd_q\Delta Q^n, \Dd_q \nabla u^n\rangle_{L^2}
\end{equation}
As already remark in the beginning, such term presents some difficulties. For instance, fixing $q\in\ZZ$ in the sum, the more natural estimate is the following one:
\begin{equation*}
	\langle \Sd_{q-1}Q^n\Dd_q\Delta Q^n, \Dd_q \nabla u^n\rangle_{L^2} 
	\leq 
			\|  \Sd_{q-1}				Q^n	\|_{L^\infty}
			\|  \Dd_{q  }		\Delta 	Q^n \|_{L^2}
			\| 	\Dd_{q  } 		\nabla 	u^n \|_{L^2}.
\end{equation*}
The presence of the low frequencies $\Sd_{q-1}$ in the first norm doesn't permit to transport a gradient to $Q^n$, so the best expectation is the following one:
\begin{equation*}
	\sum_{q\in\ZZ} 2^{2qs} \langle \Sd_{q-1}Q^n\Dd_q\Delta Q^n, \Dd_q \nabla u^n\rangle_{L^2}
	\lesssim
			\|  						Q^n	\|_{L^\infty}
			\|  				\Delta 	Q^n \|_{\Hh^s	}
			\| 			 		\nabla 	u^n \|_{\Hh^s	}.
\end{equation*}
Of course such inequality is not useful for our purpose, i.e. an Osgood type inequality. For example there isn't a term that appears in the time derivative of the left-hand side of \eqref{sec3_eq_energy_Dq}. Even if there exists a way to overcome such challenging evaluation, we will see that \eqref{sec3_est10} is going to be erased.

\vspace{0.1cm}
\noindent
Now, let us keep on our control. We have to examine $\langle \Delta Q^n\,  Q^n ,\, \nabla u^n\rangle_{\Hh^s}$. Observing that an equivalent formulation is $\langle Q^n \Delta Q^n   ,\, \tr\nabla u^n\rangle_{\Hh^s}$ ($Q^n$ and $\Delta Q^n$ are symmetric matrices) we recompute the previous inequality (with $\tr \nabla u$ instead of $\nabla u$), so that 
\begin{equation}\label{sec_est11}
	\sum_{q\in\ZZ}\sum_{i=1,2,4} 2^{2qs}
		\langle 	\J_q^i (Q^n,\,\Delta Q^n),	\,		\Dd_q 	\tr \nabla 	u^n	\rangle_{L^2}
	\lesssim
			\|				\nabla	Q^n		\|_{L^\infty}
			\|				\nabla 	u^n		\|_{\Hh^s}
			\|				\nabla 	Q^n		\|_{\Hh^s}.
\end{equation}

\noindent
As before, $\J^3_q$ is an inflexible term, so that, recalling \eqref{sec3_est10}, we need to erase what follows:
\begin{equation}\label{control1}
\begin{aligned}
	\sum_{q\in\ZZ} 
	2^{2qs} 
	\Big\{
		\langle 		\Sd_{q-1}Q^n\Dd_q\Delta Q^n, 								&\Dd_q    	\nabla u^n	\rangle_{L^2} - 
		\langle 		\Sd_{q-1}Q^n\Dd_q\Delta Q^n,  								 \Dd_q \tr	\nabla u^n	\rangle_{L^2}
	\Big\} = \\& =
	\sum_{q\in\ZZ} 
	2^{2qs}
		\langle 		\Sd_{q-1}Q^n\Dd_q\Delta Q^n -\Dd_q\Delta Q^n\Sd_{q-1}Q^n,\,  \Dd_q    	\nabla u^n	\rangle_{L^2}
\end{aligned}
\end{equation}

\vspace{0.2cm}
\noindent
\underline{Estimate of $\langle\, \Omega^n Q^n - Q^n \Omega^n,\,\Delta Q^n \rangle_{\Hh^s}$ }
\vspace{0.2cm}

\noindent
Now, let us continue estimating $\langle\, \Omega^n Q^n - Q^n \Omega^n,\,\Delta Q^n \rangle_{\Hh^s}$. The strategy as the same organization of the previous evaluation. We begin analyzing $\langle\, Q^n \Omega^n,\,\Delta Q^n\rangle_{\Hh^s}$
\begin{equation*}
	\langle\, Q^n \Omega^n,\,\Delta Q^n\rangle_{\Hh^s} = 
	\sum_{q\in \ZZ}2^{2qs}\langle \Dd_q (Q^n \Omega^n),\, \Dd_q\Delta Q^n\rangle_{L^2} = 
	\sum_{q\in \ZZ}\sum_{i=1}^4 2^{2qs}\langle \J^i_q (Q^n,\, \Omega^n),\, \Dd_q\Delta Q^n\rangle_{L^2} 
\end{equation*}
First, considering $i=1$ and $q\in\ZZ$, we get
\begin{align*}
	\langle \J^1_q (Q^n,\, \Omega^n),\, \Dd_q\Delta Q^n\rangle_{L^2} 
	&= \sum_{|q-q'|\leq 5} 
		\langle [ \Dd_q,\,\Sd_{q'-1}Q^n]\Dd_{q'} \Omega^n,\, \Dd_q \Delta Q^n\rangle_{L^2}\\
	&\lesssim 
	\sum_{|q-q'|\leq 5} 
			2^{-q}
			\|	\Sd_{q'-1}		\nabla Q^n 	\|_{L^\infty} 
			\|	\Dd_{q'  } 		\nabla u^n 	\|_{L^2		}
			\|	\Dd_{q	 }		\Delta Q^n	\|_{L^2		}\\
	&\lesssim
			\|					\nabla Q^n 	\|_{L^\infty}
			\|	\Dd_{q   }		\nabla Q^n	\|_{L^2		}
	\sum_{|q-q'|\leq 5} 
			\|	\Dd_{q'  } 		\nabla u^n 	\|_{L^2		}. 
\end{align*}
therefore, taking the sum as $q\in\ZZ$, 
\begin{equation}\label{sec3_est8}
	\sum_{q\in\ZZ} 2^{2qs}\langle \J^1_q (Q^n,\, \Omega^n),\, \Dd_q\Delta Q^n\rangle_{L^2}
	\lesssim
			\|					\nabla Q^n	\|_{L^\infty}
			\|					\nabla Q^n	\|_{\Hh^s	}
			\|			 		\nabla u^n 	\|_{\Hh^s	}.
\end{equation}
By a similar method as for proving \eqref{sec3_est8} or \eqref{sec3_est6}, the case $i=2$ produces
\begin{equation*}
	\sum_{q\in\ZZ} 
	2^{2qs}
	\langle 	\J^2_q (Q^n,\, \Omega^n),	\, 	\Dd_q \Delta Q^n		\rangle_{L^2}
	\lesssim
	\|				\nabla Q^n	\|_{L^\infty}^2
	\|				\nabla Q^n	\|_{\Hh^s	}^2+
	\frac{\nu}{100}
	\|			 	\nabla u^n 	\|_{\Hh^s	}^2,
\end{equation*}
while, for $i=4$, we get
\begin{align*}
	\langle 	\J_q^4 (Q^n,\,\Omega^n),					\,	\Dd_q \Delta Q^n	\rangle_{L^2}
	&= \sum_{q'\geq q-5}
	\langle 	\Dd_q[\,\Dd_{q'} Q^n\Sd_{q'+2} \Omega^n],	\, 	\Dd_q \Delta Q^n	\rangle_{L^2}\\
	&\lesssim	
	\sum_{q'\geq q-5}	
			\|	\Dd_{q'  } 				Q^n 	\|_{L^2		}
			\|	\Sd_{q'+2} 				\Omega^n\|_{L^\infty}
			\|	\Dd_{q   }		\Delta	Q^n		\|_{L^2		}\\
	&\lesssim	\sum_{q'\geq q-5}	
			\|	\Dd_{q'  }		\nabla 	Q^n		\|_{L^2		}
			\|	\Sd_{q'+2} 				u^n		\|_{L^\infty}
			\|	\Dd_{q   }		\Delta 	Q^n		\|_{L^2		}\\
	&\lesssim
			\|							u^n		\|_{L^\infty}
			\|	\Dd_{q   } 		\Delta 	Q^n		\|_{L^2		}
	\sum_{q'\geq q-5}
			\|	\Dd_{q'  } 		\nabla 	Q^n		\|_{L^2		}.
\end{align*}
Thus, multiplying by $2^{2qs}$ and taking the sum as $q\in \ZZ$, we realize that
\begin{align*}
	\sum_{q\in\ZZ}2^{2qs}
	\langle 	&\J_q^4 (Q^n,\,\Omega^n),\,		\Dd_q \Delta Q^n	\rangle_{L^2}
	\lesssim\\
	&\lesssim
		\|	u^n						\|_{L^\infty}
	\sum_{q\in\ZZ}
	\Big(
		2^{qs}
		\|	\Dd_q \Delta 		Q^n	\|_{L^2		}
		\sum_{q'\in\ZZ}
		2^{(q-q')s}1_{(-\infty,5)}(q-q')
		2^{q's}
		\|	\Dd_{q'} 	\nabla 	Q^n	\|_{L^2		}
	\Big)\\
	&\lesssim
	\|	u^n						\|_{L^\infty}
	\|	\Delta Q^n				\|_{\Hh^s}
	\Big[
	\sum_{q\in\ZZ}
	\Big(
		\sum_{q'\in\ZZ}
		2^{(q-q')s}1_{(-\infty,5)}(q-q')
		2^{q's}
		\|	\Dd_{q'} 	\nabla 	Q^n	\|_{L^2		}
	\Big)^2
	\Big]^{\frac{1}{2}}, 
\end{align*}
so that, passing through the Young inequality,
\begin{equation}\label{sec3_est9}
	\sum_{q\in\ZZ}2^{2qs}\langle 	\J_q^4 (Q^n,\,\Omega^n),\,		\Dd_q \Delta Q^n	\rangle_{L^2}
	\lesssim
	\|	u^n						\|_{L^\infty}
	\|	\Delta Q^n				\|_{\Hh^s}
	\|	\nabla Q^n				\|_{\Hh^s}.
\end{equation}
As the reader has already understood, the challenging term is the one related to $\J_q^3$, that is
\begin{equation}\label{sec3_est11}
	\sum_{q\in\ZZ}2^{2qs}\langle 	\J_q^3 (Q^n,\,\Omega^n),\,		\Dd_q \Delta Q^n	\rangle_{L^2} = 
	\sum_{q\in\ZZ}2^{2qs}\langle 	\Sd_{q-1}Q^n,\Dd_q\Omega^n,\,		\Dd_q \Delta Q^n	\rangle_{L^2} 
\end{equation}
As \eqref{control1}, we are not capable to control it, so we claim that such obstacle is going to be simplified.

\noindent
Going on, we observe that $\langle\, \Omega^n Q^n,\, \Delta Q^n\rangle_{\Hh^s}$ can be reformulated as $\langle\, Q^n \Omega^n ,\, \Delta Q^n\rangle_{\Hh^s}$, which we have just analyzed. Hence we need to control \eqref{sec3_est11} twice, that is
\begin{equation}\label{control2}
	\sum_{q\in\ZZ} 2^{2qs} 2\langle S_{q-1} Q^n \Dd_q \Omega_n , \Dd_q \Delta Q^n\rangle_{L^2} = 
	\sum_{q\in\ZZ} 2^{2qs} \langle S_{q-1} Q^n \Dd_q \Omega_n - \Dd_q \Omega_n S_{q-1} Q^n, \Dd_q \Delta Q^n\rangle_{L^2},  
\end{equation}

\vspace{0.2cm}
\noindent
\underline{The Simplification}
\vspace{0.2cm}

\noindent
Recalling \eqref{control1} and \eqref{control2}, we have not evaluated
\begin{align*}
	\sum_{q\in\ZZ} 2^{2qs}
	\big\{\,
		\langle \Sd_{q-1}Q^n\Dd_q\Delta Q^n -&\Dd_q\Delta Q^n\Sd_{q-1}Q^n,\, \Dd_q    	\nabla u^n\rangle_{L^2} + \\&+ 
		\langle S_{q-1} Q^n \Dd_q \Omega_n - \Dd_q \Omega_n S_{q-1} Q^n, \Dd_q \Delta Q^n\rangle_{L^2}\,
	\big\},
\end{align*}
yet. However, this is a series whose coefficients are null, thanks to Theorem \ref{apx_lemma_omega_Q_u}. Hence, we have overcome all the previous lacks, so that the following inequality is fulfilled:
\begin{equation}\label{sec3_major_est4}
\begin{aligned}
	\langle \Delta Q^n Q^n -Q^n \Delta Q^n, \,  \nabla u^n\rangle_{\Hh^s} - &\langle\, \Omega^n Q^n - Q^n \Omega^n,\,\Delta Q^n \rangle_{\Hh^s} 
	 \lesssim\\
	&\lesssim 
	\|	(		u^n,\,	\nabla Q^n	)	\|_{L^\infty}
	\|	(\nabla u^n,\, 	\Delta Q^n	)	\|_{\Hh^s}
	\|	(		u^n,\,	\nabla Q^n	)	\|_{\Hh^s}.
\end{aligned}
\end{equation}

\vspace{0.3cm}
\noindent
\underline{Estimate of $\langle \Pp(Q^n),\, \Delta Q^n\rangle_{\Hh^s}$ }
\vspace{0.1cm}

\noindent
Finally, the last term to estimate is $\langle \Pp(Q^n),\, \Delta Q^n\rangle_{\Hh^s}$. Such evaluation is not a problematic, however it is computationally demanding, therefore we put forward in the appendix the proof of the following inequality:
\begin{equation}\label{estimate_P(Q)_DeltaQ}
	\langle\, \Pp(Q^n),\, \Delta Q^n \rangle_{\Hh^s} \lesssim (1+ \|Q\|_{H^2} + \|Q\|_{H^2}^2 )\|\nabla Q\|_{\Hh^s}^2,
\end{equation}
where we remind that $H^2$ is a non-homogeneous Sobolev Space.

\vspace{0.2cm}
\noindent
\underline{The Final Step}
\vspace{0.2cm}

\noindent
Summarizing the equality \eqref{sec3_eq_energy_Dq} and the inequalities \eqref{sec3_major_est1},   \eqref{sec3_major_est3}, \eqref{sec3_major_est2}, \eqref{sec3_major_est4} and \eqref{estimate_P(Q)_DeltaQ}, we deduce
\begin{equation}\label{sec3_equ1}
\begin{aligned}
	&\frac{\dd}{\dd t}\Big[ \|u^n \|_{\Hh^s}^2 + L \|\nabla Q^n \|_{\Hh^s}^2 \Big] 
	+	\nu 		\| 	\nabla u^n \|_{\Hh^s}^2
	+	\Gamma L^2 	\|	\Delta Q^n \|_{\Hh^s}^2
	\lesssim\\
	&\lesssim
		\|	(u^n,\,\nabla Q^n)	\|_{L^\infty	}
		\|	(u^n,\,\nabla Q^n)	\|_{\Hh^{s} 	}
		\|	(u^n,\,\nabla Q^n)	\|_{\Hh^{1+s} 	} + 
		(1+ \|Q^n\|_{H^2} + \|Q^n\|_{H^2}^2 )\|\nabla Q^n\|_{\Hh^s}^2.
\end{aligned}
\end{equation}
For $t\geq 0$, we define the following time-functions 
\begin{equation*}
	\Phi(t) := \|u^n \|_{\Hh^s}^2 + \|\nabla Q^n \|_{\Hh^s}^2 ,\quad\quad \Psi(t) := \| 	\nabla u^n \|_{\Hh^s}^2+	\|	\Delta Q^n \|_{\Hh^s}^2,
\end{equation*}
so that \eqref{sec3_equ1} yields
\begin{align*}
	\Phi'(t) + \Psi(t)
	\lesssim
		\|	(u^n(t),\,\nabla Q^n(t))	\|_{L^\infty	}
		\|	(u^n(t),\,\nabla Q^n(t))	\|_{\Hh^{s} 	}&
		\|	(u^n(t),\,\nabla Q^n(t))	\|_{\Hh^{1+s} 	} + \\ &+
		(1+ \|Q^n(t)\|_{H^2} + \|Q^n(t)\|_{H^2}^2 )\Phi(t).
\end{align*}
Then, fixing a positive integer $N=N(t)$, we apply Lemma \ref{Lemma_pre_Osgood}, obtaining
\begin{equation}\label{sec3_est12}
\begin{aligned}
	\Phi' + \Psi
	\lesssim
		\Big\{
			\|	(u^n,\,\nabla Q^n)	\|_{L^2			} +
	&		\sqrt{N}
			\|	(u^n,\,\nabla Q^n)	\|_{\Hh^1		} + 
			2^{-Ns}
			\|	(u^n,\,\nabla Q^n)	\|_{\Hh^{s+1}	}
		\Big\}
		{\scriptstyle \times}\\
	&	{\scriptstyle \times}
		\|	(u^n,\,\nabla Q^n)	\|_{\Hh^{s} 	}
		\|	(u^n,\,\nabla Q^n)	\|_{\Hh^{1+s} 	} + 
		(1+ \|Q^n\|_{H^2} + \|Q^n\|_{H^2}^2 )\Phi.
\end{aligned}
\end{equation}
For simplicity, let us define
\begin{equation*}
	f_1 := \|	(u^n,\,\nabla Q^n)	\|_{L^2	}^2 + 1+ \|Q^n\|_{H^2} + \|Q^n\|_{H^2}^2 , \quad
	f_2 := \|	(u^n,\,\nabla Q^n)	\|_{\Hh^1}^2,
\end{equation*}
hence \eqref{sec3_est12} implies
\begin{equation}\label{sec3_est13}
	\Phi'(t) + \Psi(t) \leq C\big\{ f_1(t)\,\Phi(t) + N f_2(t) \Phi(t) +  2^{-Ns}\|	(u^n,\,\nabla Q^n)(t)	\|_{\Hh^{s} 	} \Psi(t) \big\},
\end{equation}
for a positive constant $C$. Now, choosing $N(t)$ to be a positive integer which fulfills
\begin{equation*}
	\frac{1}{s}\log_2 \{ 2 + 4C + \Phi(t) \} \leq N(t) \leq  \frac{1}{s}\log_2 \{ 2 + 4C + \Phi(t) \}+1
\end{equation*}
it turns out from \eqref{sec3_est13}
\begin{equation*}
	\Phi'(t) + \Psi(t) \leq C\big\{ f_1(t)\,\Phi(t) +  f_2(t) \Phi(t) (\frac{1}{s}\log_2 \{ 2 + 4C + \Phi(t) \}+1) \big\}+  
	\frac{1}{2}\Psi(t),
\end{equation*}
so that, finally, increasing the value of $C$, we obtain
\begin{equation}\label{sec3_est14}
	\Phi'(t) + \Psi(t)
	\leq C\big( f_1(t) +  f_2(t)\big) \Phi(t) \log_2 \{ 2 + 4C + \Phi(t) \},
\end{equation}
which yields
\begin{equation*}
	\Phi'(t)\leq \frac{C}{\ln 2}\big( f_1(t) +  f_2(t)\big) (2+4C +\Phi(t)) \ln \{ 2 + 4C + \Phi(t) \}.
\end{equation*}
By integrating this differential inequality, we obtain
\begin{equation*}
	2+4C+\Phi(t)\leq (2+4C+\Phi(0))^{\exp\{ \frac{C}{\ln 2}\int_0^t ( f_1(s) +  f_2(s)  )\dd s\}}.
\end{equation*}
Recalling the definition of $\Phi$, $f_1$ and $f_2$, we obtain
\begin{equation*}
	 \|(u^n,\,\nabla Q^n)(t) \|_{\Hh^s}^2 \leq  
	 (2+4C+\|(u_0,\,\nabla Q_0) \|_{\Hh^s}^2)^{	\exp\{ \frac{C}{\ln 2}\int_0^t( \|	(u^n(s),\,\nabla Q^n(s))	\|_{L^2	}^2 + 1+ \|Q(s)\|_{H^2} + \|Q(s)\|_{H^2}^2)\dd s\}}
\end{equation*}
Moreover, integrating \eqref{sec3_est14} in time, we get
\begin{equation*}
	\int_0^t \Psi(s)\dd s \leq \Phi(0) + C\int_0^t \big( f_1(t) +  f_2(t)\big) \Phi(t) \log_2 \{ 2 + 4C + \Phi(t) \}
\end{equation*}
that is
\begin{align*}
	\int_0^t \|(u^n,\,\nabla Q^n)(\tau)\|_{\Hh^{s+1}}^2\dd \tau	&\leq 
	\|(u_0,\,\nabla Q_0) \|_{\Hh^s}^2 + 
	C\int_0^t \big\{ \|	(u^n,\,\nabla Q^n)	\|_{L^2	}^2 + 1+ \|Q^n\|_{H^2} + \\
	&+ \|Q^n\|_{H^2}^2\big\}(\tau)\dd \tau  \|(u^n,\,\nabla Q^n)(t) \|_{\Hh^s}^2
	\log_2 \{ 2 + 4C + \|(u^n,\,\nabla Q^n)(t) \|_{\Hh^s}^2\},
\end{align*}
Since such estimates are uniform in $n$, we pass to the limit as $n$ goes to $\infty$, obtaining
\begin{equation*}
	\|(u,\,\nabla Q)\|_{L^\infty_T \Hh^s} \leq 
	(2+4C+\|(u_0,\,\nabla Q_0) \|_{\Hh^s}^2)^{	\frac{1}{2}\exp\{ \frac{C}{\ln 2}\int_0^T( \|	(u^n(s),\,\nabla Q^n(s))	\|_{L^2	}^2 + 1+ \|Q(s)\|_{H^2} + \|Q(s)\|_{H^2}^2)\dd s\}},
\end{equation*}
and
\begin{align*}
	\int_0^t \|(u,\,\nabla Q)(\tau)\|_{\Hh^{s+1}}^2\dd \tau	&\leq 
	\|(u_0,\,\nabla Q_0) \|_{\Hh^s}^2 + 
	C\int_0^t \big\{ \|	(u,\,\nabla Q)	\|_{L^2	}^2 + 1+ \|Q\|_{H^2} + \\
	&+ \|Q\|_{H^2}^2\big\}(\tau)\dd \tau  \|(u,\,\nabla Q)(t) \|_{\Hh^s}^2
	\log_2 \{ 2 + 4C + \|(u,\,\nabla Q)(t) \|_{\Hh^s}^2\},
\end{align*}
where $(u,\,Q)$ is solution of \eqref{main_system} with $(u_0, \,Q_0)$ as initial data. This concludes the proof of Theorem \eqref{Thm_Regularity}.
\end{proof}

\appendix
\section{}
\subsection{Useful tools}

\begin{lemma}\label{apx_lemma_omega_Q_u}
	Let $Q_1$ and $Q_2$ be two $3\times 3$ symmetric matrices with entries in $H^2(\RR^2)$. Assume that $u$ is a $3$-vector with components in 
	$H^1(\RR^2)$ and let $\Omega$ be the $3\times 3$ matrix defined by $1/2(\nabla u - \tr \nabla u)$. Then the following identity is satisfied:
	\begin{equation*}
		\int_{\RR^2}\trc\{(\Omega Q_2 - Q_2 \Omega) \Delta Q_1 \} + \int_{\RR^2} \trc\{ ( \Delta Q_1  Q_2 - Q_1 \Delta Q_2  )\nabla u \} = 0
	\end{equation*}
\end{lemma}
\begin{proof}
	The proof is straightforward, indeed by a direct computation
	\begin{align*}
		\int_{\RR^2}\trc\{(\Omega Q_2 &-Q_2 \Omega) \Delta Q_1 \}	 
												= 	\int_{\RR^2}\big[	\trc\{\Omega Q_2 \Delta Q_1 \} - \trc\{ Q_2 \Omega \Delta Q_1 		\}	\big]	
												=	\int_{\RR^2}\big[	\trc\{\Omega Q_2 \Delta Q_1 \} -\\&- \trc\{ \Delta Q_1 \tr \Omega  Q_2 	\}	\big] 
												=  2\int_{\RR^2}		\trc\{\Omega Q_2 \Delta Q_1 \}													
												=	\int_{\RR^2}		\trc\{\nabla u Q_2 \Delta Q_1 - \tr \nabla u Q_2 \Delta Q_1 		\}=\\			
												 	&\quad\quad\quad\quad\quad\quad\quad\quad\quad\quad\quad\quad\quad\quad\quad\quad\quad\quad\quad\quad
												=	\int_{\RR^2} 		\trc\{ (Q_1 \Delta Q_2 - \Delta Q_1  Q_2 )\nabla u 					\}.
	\end{align*}
\end{proof}

\begin{lemma}\label{Lemma_pre_Osgood}
	Let $f$ be a function in $H^1\cap \Hh^{1+s}$ with $s>0$. Then, there exists $C>0$ such that
	\begin{equation*}
		\| f \|_{L^\infty} \leq C\big(\,\| f \|_{L^2} + \sqrt{N}\| f \|_{H^1} + 2^{-Ns}\| f \|_{\Hh^{1+s}} \big),
	\end{equation*}
	for any positive integer $N$.
\end{lemma}
\begin{proof}
	Let us fix $N>0$. Then $f = \Sd_{N+1}f + (\Id - \Sd_{N+1} )f$ fulfills
	\begin{equation*}
		\| f \|_{L^\infty} \leq  \| \Sd_{N+1}f \|_{L^\infty} +  \| \sum_{q\geq N }\Dd_q f \|_{L^\infty} \leq 
	\underbrace{	
		\sum_{q <    N }	\| \Dd_q f\|_{L^\infty}}_{\Aa} + 
	\underbrace{
		\sum_{q \geq N }	\| \Dd_q f \|_{L^\infty}}_{\Bb}.
	\end{equation*}
	First, let us analyze $\Aa$:
	\begin{align*}
			\sum_{q <    N }					\| \Dd_q f\|_{L^\infty}
		&=
			\sum_{ q \leq 0 }					\| \Dd_q f\|_{L^\infty} +
			\sum_{  q = 1  }^N					\| \Dd_q f\|_{L^\infty}
		 \lesssim	
			\sum_{ q \leq 0 }	2^{ q  	   }	\| \Dd_q f\|_{L^2} +
			\sum_{  q = 1   }^N 2^{ q 	   }	\| \Dd_q f\|_{L^2}\\
		&\lesssim	
			\sum_{ q \leq 0 }					\| \Dd_q f\|_{L^2} +
			\sqrt{N}
				\|		f		\|_{	\Hh^{1}	}
		 \lesssim
				\|		f		\|_{	L^2		}	+
			\sqrt{N}
				\|		f		\|_{	\Hh^{1}	}.
	\end{align*}
	Finally, from the definition of $\Bb$
	\begin{equation*}
		\sum_{q \geq N }					\| \Dd_q f \|_{L^\infty}
		=
		\sum_{q	\geq N}			2^{ q	 }	\| \Dd_q f \|_{L^2}
		=
		\sum_{q	\geq N}	2^{-sq}	2^{q(1+s)}	\| \Dd_q f \|_{L^2}
		\lesssim
		2^{-Ns} \|	f	\|_{		\Hh^{1+s}	},
	\end{equation*}
	which concludes the proof of the lemma.
\end{proof}

\begin{proof}{proof of Theorem \ref{product_Hs_Ht}}
	At first we identify the Sobolev Spaces $\Hh^s$ and $\Hh^t$ with the Besov Spaces $\BB_{2,2}^s$ and $\BB_{2,2}^t$ respectively. 
	We claim that $ab$ belongs to $\BB_{2,2}^{s+t-N/2}$ and 
	\begin{equation*}
		\|a b\|_{\BB_{2,2}^{s+t-N/2}}\leq C\|a\|_{\BB_{2,2}^{s}}\|b\|_{\BB_{2,2}^t},
	\end{equation*}
	for a suitable positive constant. 
	
	\noindent We decompose the product $ab$ through the Bony decomposition, namely $ab = \dot{T}_{a}b + \dot{T}_{b}a + R(a,b)$, where
	\begin{equation*}
		\dot{T}_{a}b	:= \sum_{q\in\ZZ}							\Dd_q 		a\, \Sd_{q-1}	b,\quad\quad
		\dot{T}_{b}a	:= \sum_{q\in\ZZ}							\Sd_{q-1}	a\,	\Dd_q 		b,\quad\quad
		\Rd(a,b)		:= \sum_{\substack{q\in\ZZ\\|\nu|\leq 1}}	\Dd_{q}		a\,	\Dd_{q+\nu} b.
	\end{equation*}
	For any $q\in \ZZ$, we have
	\begin{align*}
		2^{q(s+t-\frac{N}{2})}
		&\| (\Dd_q \dot{T}_ab ,\, \Dd_q \dot{T}_b a) \|_{L^2}
		\lesssim\\	&\lesssim
		\sum_{|q-q'|\leq 5}		2^{q's} 				\|\Dd_q		a \|_{L^2} 		2^{q'(t-\frac{N}{2})} 	\|	\Sd_{q-1} b\|_{L^\infty} +
		\sum_{|q-q'|\leq 5}		2^{q'(s-\frac{N}{2})} 	\|\Sd_{q-1} a \|_{L^\infty}	2^{q't} 				\| 	\Dd_q b \|_{L^2},
	\end{align*}
	so that we determine the following feature
	\begin{equation*}
	 	\| (\dot{T}_{a}b,\, \dot{T}_b a) \|_{\BB_{2,2}^{s+t-\frac{N}{2}}}\leq
	 	\| (\dot{T}_{a}b,\, \dot{T}_b a) \|_{\BB_{2,1}^{s+t-\frac{N}{2}}}\lesssim
	 	\| a \|_{\BB_{2, 2}^s}						\| b \|_{\BB_{\infty, 2}^{t-\frac{N}{2}}} + 
	 	\| a \|_{\BB_{\infty, 2}^{s-\frac{N}{2}}}	\| b \|_{\BB_{2, 2}^t}\lesssim
	 	\| a \|_{\BB_{2, 2}^s}							\| b \|_{\BB_{2, 2}^t},
	\end{equation*}	
	where we have used the embedding $\BB_{2,2}^{\sigma}\hookrightarrow \BB_{\infty, 2}^{\sigma -N/2}$, for any $\sigma\in\RR$ and Proposition \ref{prop_besov_s_negative}.
	
	\vspace{0.1cm}
	\noindent In order to conclude the proof, we have to handle the rest $\Rd(a,b)$. By a direct computation, for any $q\in \ZZ$, 
	\begin{equation*}
		2^{(t+s)q} \| \Dd_q \Rd(a,b) \|_{L^1} \leq 
		\sum_{\substack{q'\geq q-5\\|\nu|\leq 1}} 2^{(q-q')(s+t)} 2^{q's} \|\Dd_{q'} a\|_{L^2} 2^{(q'+\nu)t}\|\Dd_{q'+\nu} a\|_{L^2}, 
	\end{equation*}
	so that, thanks to the Young inequality, we deduce
	\begin{equation*}
		\| \Rd(a,b) \|_{\BB_{2,2}^{s+t-\frac{N}{2}}} \lesssim
		\| \Rd(a,b) \|_{\BB_{1,1}^{s+t}} \lesssim 
		\| 	a		\|_{\BB_{2,2}^s}
		\| 	b		\|_{\BB_{2,2}^t},
	\end{equation*}
	where we have used the embedding $\BB_{1,1}^{s+t}\hookrightarrow \BB_{2,2}^{s+t-N/2}$ and moreover that $\sum_{q\leq 5} 2^{q(s+t)}$ is 
	finite, since $s+t$ is positive.	 
\end{proof}

\section{}
\subsection{Proof of estimate  \eqref{estimate_P(Q)_DeltaQ}}$ $

\noindent
The purpose of this section is to estimate the following term
\begin{equation*}
	\langle \Pp(Q^n),\, \Delta Q^n \rangle_{\Hh^s}.
\end{equation*}
In order to facilitate the reader, we are not going to indicate the index $n$, from here on. We have to examine
\begin{align*}
	\langle \Pp(Q),\, \Delta Q \rangle_{\Hh^s} 	&= \langle -aQ + b [Q^2 - \trc\{Q^2\}\frac{\Id}{3}]  - c\trc\{Q^2\}Q,\,\Delta Q\rangle_{\Hh^s}\\
												&= \langle -aQ + b  Q^2 							 - c\trc\{Q^2\}Q,\,\Delta Q\rangle_{\Hh^s},
\end{align*}
where $\langle \trc\{Q^2\}\Id, \,\Delta Q\rangle_{\Hh^s}=0$ since $\Delta Q$ has null trace.
It is trivial that 
\begin{equation}\label{appx_first_estimate}
-\langle a Q ,\,\Delta Q\rangle_{\Hh^s}\lesssim \|\nabla Q\|_{\Hh^s}.
\end{equation} 
Now, let us consider $b\langle Q^2,\, \Delta Q\rangle_{\Hh^s}$. By definition we have
\begin{align*}
	b\langle Q^2,\, \Delta Q\rangle_{\Hh^s} 
	&= b\sum_{q\in\ZZ} 2^{2qs} \langle \Dd_q[ Q^2],\,\Dd_q \Delta Q\rangle_{L^2}\\ 
	&= 
	b\sum_{q\in\ZZ} 2^{2qs}
	\big[ 
		2\underbrace{
			\langle  	\Dd_q \dot{T}_QQ 	,\, \Dd_q 	\Delta Q\rangle_{L^2}}_{\Aa_q} + 
		 \underbrace{
		  	\langle		\Dd_q \dot{R}(Q,\,Q),\, \Dd_q	\Delta Q\rangle_{L^2}}_{\Bb_q}
	\big]
\end{align*}
We concentrate on $\Aa_q$, getting
\begin{align*}
	\Aa_q  
	&\leq \sum_{|q-q'|\leq 5} \|\Sd_{q'-1} Q \Dd_{q'} Q\|_{L^2} \| \Dd_q \Delta Q\|_{L^2}
	\lesssim \|Q\|_{L^\infty}\| \Dd_q \nabla Q\|_{L^2}\sum_{|q-q'|\leq 5}\|\Dd_{q'} \nabla Q\|_{L^2},
\end{align*}
so that
\begin{equation}\label{appx_est_A_q}
	 b\sum_{q\in\ZZ} 2^{2qs}\Aa_q \lesssim \|Q\|_{L^\infty}\|  \nabla Q\|_{\Hh^s}^2.
\end{equation}
Now, analyzing $\Bb_q$, we observe that
\begin{equation*}
	\Bb_q 
	\leq 		\sum_{\substack{q'\geq q- 5\\|l|\leq 1}} \| \Dd_{q'}Q \Dd_{q'+l} Q \|_{L^2} \|\Dd_q \Delta Q\|_{L^2} 
	\lesssim	\| Q \|_{L^\infty} \|\Dd_q \nabla Q\|_{L^2} \sum_{q'\geq q- 5}2^{q-q'}\| \Dd_{q'}\nabla Q \|_{L^2},
\end{equation*}
so that
\begin{equation*}
	 b\sum_{q\in\ZZ} 2^{2qs}\Bb_q \lesssim 
	 \| Q \|_{L^\infty}b\sum_{q\in\ZZ} 2^{qs}\|\Dd_q \nabla Q\|_{L^2} \sum_{q'\in\ZZ}2^{(q-q')(s+1)}1_{(-\infty,5)}(q-q')\| \Dd_{q'}\nabla Q \|_{L^2}.
\end{equation*}
Thus, by convolution and young inequality
\begin{equation*}
	b\sum_{q\in\ZZ} 2^{2qs}\Bb_q\lesssim  \|Q\|_{L^\infty}\|  \nabla Q\|_{\Hh^s}^2,
\end{equation*}
and recalling \eqref{appx_est_A_q}, we finally get
\begin{equation}\label{appx_second_estimate}
	b\langle Q^2,\, \Delta Q\rangle_{\Hh^s} \lesssim  \|Q\|_{L^\infty}\|  \nabla Q\|_{\Hh^s}^2.
\end{equation}
Now, it remains to examine $c\langle  Q\trc \{Q^2\} ,\, \Delta Q\rangle_{\Hh^s}$. The procedure is quietly similar to the previous one. At first we use the Bony decomposition as follows:
\begin{align*}
	\langle  Q\trc \{Q^2\} ,\, \Delta Q\rangle_{\Hh^s} 
	= \sum_{q\in\ZZ} 2^{2qs} \langle \Dd_q( Q &\trc\{Q^2\}) , \, \Dd_q \Delta Q \rangle_{L^2} 
	= 	\sum_{q\in\ZZ} 2^{2qs}
	\Big[ 
		\underbrace{
			\langle \Dd_q	\dot{T}_Q (\trc\{Q^2\}\Id),		\,\Dd_q 	\Delta Q	\rangle_{L^2}}_{\Aa_q} + \\&+
		\underbrace{	
			\langle \Dd_q	\dot{T}_{\trc\{Q^2\}\Id}Q 		,\,\Dd_q 	\Delta Q	\rangle_{L^2}}_{\Bb_q} +
		\underbrace{	
			\langle	\Dd_q	\dot{R}(Q,\trc\{Q^2\}\Id)		,\,\Dd_q 	\Delta Q	\rangle_{L^2}}_{\Cc_q} 
	\Big]
\end{align*}
First, we concentrate on $\Aa_q$, the more computationally demanding term, obtaining
\begin{align*}
	\Aa_q \leq 			\sum_{[q-q'|\leq 5} \|\Sd_{q'-1}Q 	\Dd_{q'} (\trc\{Q^2\}\Id)\|_{L^2} \| \Dd_q \Delta Q		\|_{L^2} \lesssim
	\|Q \|_{L^\infty}	\sum_{[q-q'|\leq 5} \|				\Dd_{q'} (	Q^2			 )\|_{L^2} \| \Dd_q	\Delta Q	\|_{L^2}\\
	\lesssim
	\|Q \|_{L^\infty}	\sum_{[q-q'|\leq 5}
	\Big[
		\underbrace{
		2	\|				\Dd_{q'}	\dot{T}_Q Q									\|_{L^2}\| \Dd_q \Delta Q		\|_{L^2}}_{I_{q,q'}} +
		\underbrace{
			\|				\Dd_{q'}	\dot{R}(Q,Q)								\|_{L^2}\| \Dd_q \Delta Q		\|_{L^2}}_{II_{q,q'}}
	\Big]
\end{align*}
The term $I_q$ is the simpler one, indeed
\begin{equation*}
	I_{q,q'} \lesssim \sum_{|q'-q''|\leq 5} \| \Sd_{q''-1} Q \Dd_{q''} Q\|_{L^2}\| \Dd_q \Delta Q		\|_{L^2} 
	\lesssim
	\| Q \|_{L^\infty}
	\sum_{|q'-q''|\leq 5}
	\|\Dd_{q''} Q\|_{L^2}\| \Dd_q \Delta Q		\|_{L^2},
\end{equation*}
so that
\begin{align*}
	\sum_{q\in\ZZ} \|Q\|_{L^\infty}
	\sum_{[q-q'|\leq 5}I_{q,q'}
	&\lesssim
	\|Q \|_{L^\infty}^2\sum_{q\in\ZZ}	\sum_{[q-q'|\leq 5}
	\sum_{|q'-q''|\leq 5}
	\|\Dd_{q''} Q\|_{L^2}\| \Dd_q \Delta Q		\|_{L^2} \\
	&\lesssim
	\|Q \|_{L^\infty}^2\sum_{q\in\ZZ}	
	\sum_{|q-q''|\leq 10}
	\|\Dd_{q''}\nabla Q\|_{L^2}\| \Dd_q \nabla Q		\|_{L^2}\lesssim 
	\|Q \|_{L^\infty}^2\| \nabla Q\|_{\Hh^s}. 
\end{align*}
We overcome the term $II_{q,q'}$ as follows:
\begin{align*}
	II_{q,q'}\lesssim 
	 \|\Dd_q \Delta Q\|_{L^2}\sum_{\substack{q''\geq q'-5\\ |l|\leq 1}} \|\Dd_{q''}Q \|_{L^2}\|\Dd_{q''+l}Q \|_{L^\infty}\lesssim
	 \| Q \|_{L^\infty}\|\Dd_q \nabla Q\|_{L^2} \sum_{q''\geq q'-5}2^{q-q''} \|\Dd_{q''}\nabla Q \|_{L^2},
\end{align*}
so that
\begin{align*}
	\sum_{q\in\ZZ} \|Q\|_{L^\infty}
	\sum_{[q-q'|\leq 5}II_{q,q'}
	&\lesssim
	\|Q \|_{L^\infty}^2\sum_{q\in\ZZ}2^{2qs}\|\Dd_q \nabla Q\|_{L^2}\sum_{[q-q'|\leq 5}
	\sum_{q''\geq q'-5} 2^{q-q''} \|\Dd_{q''}\nabla Q \|_{L^2}\\
	&\lesssim
	\|Q \|_{L^\infty}^2\sum_{q\in\ZZ}2^{2qs}\|\Dd_q \nabla Q\|_{L^2}
	 \sum_{q''\geq q-10} 2^{q-q''} \|\Dd_{q''}\nabla Q \|_{L^2}\\
	&\lesssim
	\|Q \|_{L^\infty}^2\sum_{q\in\ZZ}2^{qs}\|\Dd_q \nabla Q\|_{L^2} 
	\sum_{q''\geq q-10} 2^{(q-q'')(s+1)} 2^{q''s} \|\Dd_{q''}\nabla Q \|_{L^2},
\end{align*}
so that, by convolution and Young inequality
\begin{equation*}
	\sum_{q\in\ZZ} \|Q\|_{L^\infty}
	\sum_{[q-q'|\leq 5}II_{q,q'} \lesssim \|Q \|_{L^\infty}^2\| \nabla Q\|_{\Hh^s}.
\end{equation*}
Summarizing the previous inequalities, we get
\begin{equation*}
	\sum_{q\in\ZZ}2^{2qs}\Aa_q\lesssim \| Q \|_{L^\infty}^2\| \nabla Q\|_{\Hh^s}^2.
\end{equation*}
In order to examine $\Bb_q$ it is sufficient to observe that
\begin{align*}
	\sum_{q\in\ZZ}2^{2qs}&\Bb_q \lesssim \sum_{q\in\ZZ}2^{2qs} \sum_{|q-q'|\leq 5}\| \Sd_{q'-1}(\trc\{Q^2\}\Id) \Dd_{q'} Q \|_{L^2}\|\Dd_q \Delta Q\|_{L^2}\\
	&\lesssim \|Q^2 \|_{L^\infty}\sum_{q\in\ZZ}2^{2qs} \sum_{|q-q'|\leq 5}\|\Dd_{q'}\nabla Q \|_{L^2}\|\Dd_q \nabla Q\|_{L^2}
	\lesssim \|Q \|_{L^\infty}^2\|\nabla Q\|_{\Hh^s}^2.
\end{align*}
It remains indeed $\Cc_q$, which is straightforward, indeed
\begin{align*}
	\sum_{q\in\ZZ} 2^{2qs} \Cc_q 
	&\lesssim 
	\sum_{q\in\ZZ} 
	2^{2qs}
	\sum_{\substack{q'\geq q-5\\ |l|\leq 1}}
	\|  \Dd_{q'} Q \Dd_{q+l}(Q^2) 	\|_{L^2}
	\|	\Dd_{q}	\Delta Q			\|_{L^2}\\
	&\lesssim
	\|Q\|_{L^\infty}^2
	\sum_{q\in\ZZ} 2^{qs} 
	\|	\Dd_{q}	\nabla Q			\|_{L^2}
	\sum_{q'\geq q-5} 2^{(q-q')(s+1)}
	\|	\Dd_{q'}	\nabla Q			\|_{L^2},
\end{align*}
thus, by convolution and the Young inequality,
\begin{equation*}
	\sum_{q\in\ZZ} 2^{2qs} \Cc_q \lesssim \|Q\|_{L^\infty}^2 \|\nabla Q\|_{\Hh^s}^2.
\end{equation*}
Summarizing, we finally get 
\begin{equation*}
	c\langle  Q\trc \{Q^2\} ,\, \Delta Q\rangle_{\Hh^s}\lesssim \|Q\|_{L^\infty}^2 \|\nabla Q\|_{\Hh^s}^2
\end{equation*}
and recalling  \eqref{appx_first_estimate}-\eqref{appx_second_estimate}, we finally obtain
\begin{equation*}
	\langle \Pp(Q), \Delta Q \rangle_{\Hh^s}	
	\lesssim 
	(	1	+ \|	Q 	\|_{L^\infty}	+	\|	Q	\|_{L^\infty}^2	)	\|	\nabla	Q	\|_{	\Hh^s	}^2
	\lesssim
	(	1	+ \|	Q 	\|_{H^2}	+	\|	Q	\|_{H^2}^2	)	\|	\nabla	Q	\|_{	\Hh^s	}^2,
\end{equation*}
where the last inequality is due to the embedding $H^2(\RR^2)\hookrightarrow L^{\infty}(\RR^2)$. Hence, inequality \eqref{estimate_P(Q)_DeltaQ} is proven.

\pagestyle{empty}
\bibliographystyle{amsplain}
\providecommand{\bysame}{\leavevmode\hbox to3em{\hrulefill}\thinspace}
\providecommand{\MR}{\relax\ifhmode\unskip\space\fi MR }
\providecommand{\MRhref}[2]{%
  \href{http://www.ams.org/mathscinet-getitem?mr=#1}{#2}
}
\providecommand{\href}[2]{#2}

\end{document}